\documentclass[10pt]{article}

\usepackage{bm}
\usepackage{bbm}
\usepackage{fullpage}
\usepackage[utf8]{inputenc} % allow utf-8 input
\usepackage[T1]{fontenc}    % use 8-bit T1 fonts
\usepackage{hyperref}       % hyperlinks
\usepackage{url}            % simple URL typesetting
\usepackage{booktabs}       % professional-quality tables
\usepackage{amsfonts}       % blackboard math symbols
\usepackage{nicefrac}       % compact symbols for 1/2, etc.
\usepackage{microtype}      % microtypography
\usepackage[table]{xcolor}
\usepackage{authblk}
\newcommand{\Proj}{\text{Proj}}

\newcommand{\calA}{\mathcal{A}}
\newcommand{\calB}{\mathcal{B}}
\newcommand{\calS}{\mathcal{S}}
\newcommand{\calT}{\mathcal{T}}
\newcommand{\calF}{\mathcal{F}}

\newcommand{\E}{\mathbb{E}}

\newcommand{\ls}{\left}
\newcommand{\rs}{\right}
\newcommand{\la}{\ls\langle}
\newcommand{\ra}{\rs\rangle}
\newcommand{\prob}{\mathrm{Pr}}

\newcommand\numberthis{\addtocounter{equation}{1}\tag{\theequation}}

\usepackage[table]{xcolor}
\usepackage{tabstackengine}
\usepackage{appendix}
\usepackage{titletoc}
\usepackage{titlesec}
\usepackage{enumitem}
\usepackage{amsmath}
\usepackage{amssymb}
\usepackage{mathtools}
\usepackage{amsthm}
\usepackage{algorithm}
\usepackage{algorithmic}
\usepackage{graphicx}
\usepackage{subfigure}
\usepackage{makecell}
\usepackage{array}

\theoremstyle{plain}
\newtheorem{theorem}{Theorem}[section]
\newtheorem{proposition}[theorem]{Proposition}
\newtheorem{lemma}[theorem]{Lemma}

\newtheorem{remark}{Remark}[section]
\theoremstyle{definition}

\title{On the Convergence of Policy Mirror Descent with Temporal Difference Evaluation}

\begin{document}

\author{Jiacai Liu$^{*}$}
\author{Wenye Li\thanks{Equal contribution.}}
\author{Ke Wei}
\affil{School of Data Science, Fudan University, Shanghai, China.}%\vspace{.15cm}}
\date{\today}
\maketitle

\begin{abstract}
  Policy mirror descent (PMD) is a general policy optimization framework in reinforcement learning, which can cover a wide range of typical policy optimization methods by specifying different mirror maps. Existing analysis of PMD requires exact or approximate evaluation (for example unbiased estimation via Monte Carlo simulation) of action values solely based on  policy. In this paper, we consider policy mirror descent with temporal difference evaluation (TD-PMD). It is shown that, given the access to exact policy evaluations,  the dimension-free $O(1/T)$  sublinear convergence  still holds for TD-PMD with any constant step size and any initialization. {In order to achieve this result,  new monotonicity and shift invariance
arguments have been developed.} The dimension free $\gamma$-rate linear convergence of TD-PMD is also established provided the step size is selected adaptively. {For the two common instances of TD-PMD (i.e., TD-PQA and TD-NPG), it is further shown that they enjoy the convergence in the policy domain}. Additionally, we investigate TD-PMD in  the inexact setting and give the sample complexity for it to achieve the last iterate $\varepsilon$-optimality under a generative model, which improves the last iterate sample complexity for PMD over the dependence on $1/(1-\gamma)$. 
\end{abstract}

\section{Introduction}
\label{sec:introcution}
Reinforcement learning (RL), which attempts to maximize long-term reward in  sequential decision making, has achieved great success for example in video games~\cite{StarCraft}, pattern explorations~\cite{ref-AlphaFold,ref-AlphaTensor}, and control~\cite{robot2,robot3}.  RL is typically modeled by a Markov Decision Process (MDP), which can be represented as a tuple $\mathcal{M}(\calS, \calA, P, r, \gamma)$, where $\calS$ is the state space, $\calA$ is the action space, $P$ is the transition model, $r$ is the reward function, and $\gamma\in[0,1)$ is the discounted factor. The decision making process of an agent is determined by a policy $\pi$. More precisely, at state $s_t$, the agent selects an action $a_t \sim \pi(\cdot|s_t)$, receives a reward $r(s_t, a_t)$ and then transfers to a new state $s_{t+1} \sim P(\cdot|s_t, a_t)$. The state value function $V^\pi(s)$ is the expected discounted cumulative reward  over
the random trajectory starting from $s$ and following the policy $\pi$,
\begin{align*}
    V^{\pi}(s) = \E\ls[ \sum_{t=0}^\infty \gamma^t r(s_t, a_t) \; \Big| \; s_0 =s, \;\pi \rs].
\end{align*}
The  goal in RL is to find the optimal policy that maximizes the average state value
\begin{align}
    \max_{\pi} \; V^\pi(\mu),
    \label{eq:target-of-RL}
\end{align}
where $V^\pi(\mu) = \E_{s\sim \mu} [V^\pi(s)]$ and $\mu$ is some distribution over state space $\calS$.

Policy optimization is a family of RL methods that can be readily extended to high dimensional discrete or continuous  spaces, and is also  amenable  to a detailed analysis. The convergence  of various policy gradient methods \cite{sutton1999policy} has received intensive investigations recently, including the convergence for projected policy gradient \cite{Agarwal_Kakade_Lee_Mahajan_2019,Xiao_2022,Zhang_Koppel_Bedi_Szepesvari_Wang_2020}, softmax policy gradient \cite{mei2021normalized,Mei_Xiao_Szepesvari_Schuurmans_2020,Agarwal_Kakade_Lee_Mahajan_2019,li2023exponential,klein2024structure,pg-liu}, softmax natural policy gradient \cite{Agarwal_Kakade_Lee_Mahajan_2019,Khodadadian_Jhunjhunwala_Varma_Maguluri_2021,pg-liu}, as well as those based on general policy parameterizations \cite{zhang2019convergence,zhang2020global,yuan2022general,fatkhullin2023stochastic,mondal2024improved,wang2019neural}.  An exhaustive literature review towards this line of research is beyond the scope of this paper. Here we focus on policy mirror descent (PMD)~\cite{shani2020trpo,vaswani2021functional},
whose update is given by
\begin{align}
    \pi_{k+1}(\cdot|s) =\underset{p\in\Delta(\calA)}{\arg\max} \ls\{ \eta\, \langle p, \,Q^{\pi_k}(s,\cdot) \rangle - D_h(p, \pi_k(\cdot|s)) \rs\},
    \label{eq:intro-PMD-policy-update-form}
\end{align}
where $\Delta(\calA)$ denotes the probabilistic simplex on $\calA$, $\eta > 0$ is the step size, $D_h$ is the Bregman divergence induced by a function $h$,
\begin{align*}
    D_h(p, \,p^\prime) = h(p) - h(p^\prime) - \langle \nabla h(p^\prime), \; p-p^\prime \rangle,
\end{align*}
and $Q^{\pi_k}$ is the action value  at the $k$-th iteration,
\begin{align*}
    Q^{\pi_k}(s,a) = r(s,a) + \gamma \cdot \E_{s^\prime \sim P(\cdot|s,a)}[V^{\pi_k}(s^\prime)].
\end{align*}
{The PMD update fits into the mirror descent (MD) framework \cite{Xiao_2022}} and covers a wide range of policy gradient methods by specifying different $h$. %The convergence analysis of PMD has also received a lot of attention recently (see for example \cite{Xiao_2022, Lan_2021, Li_Zhao_Lan_2022, Zhan_Cen_Huang_Chen_Lee_Chi_2021}). {A detailed discussion on this topic as well as on other related works is presented in Appendix~\ref{app:pmd}}.
%When $Q^\pi$ cannot be assessed exactly, an estimation $\widehat{Q}^{\pi}$ is substituted in equation~\eqref{eq:intro-PMD-policy-update-form}, giving the stochastic PMD algorithm.

\subsection{Motivation and contributions}

%Most of existing PMD work focuses on the vanilla policy update form of equation~\eqref{eq:intro-PMD-policy-update-form} where the action value $Q^\pi$ is exactly assessed or stochastically approximated. However, few of work studies the case that $Q^\pi$ in equation~\eqref{eq:intro-PMD-policy-update-form} is replaced by $Q$, where $Q$ is \textit{not} the exact value or approximation of $Q^\pi$. For instance, a one-step temporal-difference evaluation (i.e., TD(0) algorithm,~\cite{sutton1998introduction}) considers learning $Q$ via

Existing convergence analysis of the general PMD update in \eqref{eq:intro-PMD-policy-update-form}  usually assumes that $Q^{\pi_k}$ can be exactly evaluated or an unbiased estimator of it can be obtained  via Monte Carlo simulation, to the best of our knowledge.  Instead of the exact function evaluation (or the Monte Carlo estimation), we may also approximate $Q^{\pi_k}$ by the temporal-difference (TD) evaluation in the following form~\cite{suttonRL}:
\begin{align}
    V^{k} = \calT^{\pi_k}V^{k-1}, \;\;\; Q^k(s,a) = r(s,a) + \gamma \cdot \E_{s^\prime}[V^k(s^\prime)],\label{eq:tdQeval}
\end{align}
where $\calT^{\pi_k}$ is Bellman operator, defined as
\begin{align}
    \calT^{\pi_k} V(s) = \E_{a\sim\pi_k(\cdot|s), \, s^\prime\sim P(\cdot|s,a)}[r(s,a) + \gamma V(s^\prime)].
    \label{eq:Bellman-operator}
\end{align}

As TD evaluation is also widely used in RL, it motivates us to study the following questions:
\begin{quote}\it Does PMD still converge if the action value is replaced  by the TD evaluation in~\eqref{eq:tdQeval}? If yes, what about the convergence rates?
\end{quote}

In this paper, we conduct extensive studies of TD-PMD (see Algorithm~\ref{alg:TD-PMD}) and provide affirmative answers to the above questions. In a nutshell, it is shown that
\begin{quote}
\it TD-PMD overall exhibits a similar convergence behavior to PMD despite the inexactness and bias in the state function evaluation. In the stochastic setting, the sample complexity for TD-PMD  to achieve the last iterate $\varepsilon$-optimality improves that for PMD by a factor of $1/(1-\gamma)$ since it does not require to sample a trajectory of length $O(1/(1-\gamma))$ for the estimation. 
\end{quote}

More precisely, the main contributions of this paper are summarized as follows:

\paragraph{Sublinear  and linear convergence of exact TD-PMD} Given the access to exact policy evaluations,  we first show that, for TD-PMD with any constant step size and \textit{any} initialization $\{ V^0, \, \pi_0 \}$, both  the state value estimation $V^T$ and the exact value function $V^{\pi_T}$  converge to the optimal state value at an $O(1/T)$ sublinear rate.  {This result has provided an affirmative answer to an open question proposed in \cite{Geist_Scherrer_Pietquin_2019}.}
The $\gamma$-rate linear convergence  is further established for TD-PMD with adaptive step sizes. %The analysis is inspired by that in \cite{Johnson_Pike-Burke_Rebeschini_2023}, but the details are different and essentially rely on the update rule of TD-PMD. Specifically, the linear convergence is firstly established for the  sequence $\{V^{T}\}$ and then is extended to the  sequence $\{V^{\pi_T}\}$ based on the relation between $V^T$ and $V^{\pi_T}$.

\paragraph{Policy convergence for representative TD-PMD instances.} We present an elementary policy analysis for two common TD-PMD instances: TD-PQA and TD-NPG, corresponding to  $h=(1/2)\| p\|_2^2$ and $h=\sum_{a\in\calA} p_a \log p_a$ respectively. It is shown that TD-PQA can find an optimal policy in a finite number of iterations while TD-NPG converges to some optimal policy {(note that the optimal policy is not necessarily unique)}. %These convergence results can be seemed as the extension from~\cite{ppgliu,li2025phi}, in which the algorithm requires exact $Q^{\pi_k}$ while TD-PMD not. 

\paragraph{Sample complexity for sample-based TD-PMD.} Based on the linear convergence, it is further shown that  TD-PMD can find the last iterate $\varepsilon$-optimal solution with\footnote{A poly-logarithmic factor is hidden in $\tilde{O}(\cdot)$.} $\tilde{O}(|\mathcal{S}||\mathcal{A}|(1-\gamma)^{-7}\varepsilon^{-2})$ samples under a generative model. In contrast, the sample complexity  established  in~\cite{Xiao_2022,Johnson_Pike-Burke_Rebeschini_2023} for PMD is $\tilde{O}(|\mathcal{S}||\mathcal{A}|(1-\gamma)^{-8}\varepsilon^{-2})$.

%%%%%%%%%%%%%
\subsection{Related works}\label{app:pmd}

\paragraph{Softmax NPG} As a widely studied instance, when $D_h$ is Kullback-Leibler (KL) divergence, PMD reduces to the natural policy gradient method~\cite{kakade2002npg} under the softmax policy parameterization (softmax NPG). It is shown that softmax NPG enjoys a dimension-free $O(1/T)$ sublinear convergence with any constant step size $\eta > 0$~\cite{Agarwal_Kakade_Lee_Mahajan_2019}. The linear convergence analysis of softmax NPG can be found in~\cite{Bhandari_Russo_2021,Khodadadian_Jhunjhunwala_Varma_Maguluri_2021,pg-liu}. Specifically,~\cite{Bhandari_Russo_2021} discusses the connection between PMD and the policy iteration method~\cite{VI-PI,suttonRL} and establishes the $\frac{1+\gamma}{2}$-rate linear convergence of softmax NPG with adaptive step sizes. The local linear convergence with any constant step size and the global linear and local superlinear convergence with adaptive step sizes are given in~\cite{Khodadadian_Jhunjhunwala_Varma_Maguluri_2021} by  analyzing the policy values on sub-optimal actions. The global linear convergence of softmax NPG with any constant step size is given in~\cite{pg-liu} by computing the policy improvement, though the convergence rate cannot be explicitly written out in terms of the problem parameters. NPG under log-linear policies and general policy parameterizations have been studied in~\cite{Agarwal_Kakade_Lee_Mahajan_2019,yuan2022linear}. There are also a series of works applying variance reduction techniques to improve the sample efficiency of NPG~\cite{feng2024global,liu2020improved,mondal2024improved}.
{There is also a line of research investigating the natural actor-critic algorithm (NAC)~\cite{peters2008natural} under the online Markovian sampling setting. In NAC, the critic is updated to evaluate the action values of current policy using TD learning, and the actor updates the policy using NPG. The Asymptotic convergence and convergence rate of NAC are investigated in~\cite{khodadadian2021finite,khodadadian2022finite,xu2020improving,yang2019provably,hong2023two,wang2019neural,xu2020non,wang2024non}, and the $\tilde{O}(\varepsilon^{-3})$ sample complexity to attain an $\varepsilon$-accurate global optimum is established in~\cite{xu2020improving}.}

\paragraph{Exact PMD} The $O(1/T)$ sublinear convergence for softmax NPG in~\cite{Agarwal_Kakade_Lee_Mahajan_2019} is subsequently extended to the general PMD method in~\cite{Xiao_2022} and~\cite{Lan_2021}. The key ingredient of in the extension of the analysis is the three-point descent lemma~\cite{chen1993convergence} within the mirror descent framework. The linear and local superlinear convergence of PMD with geometrically increasing step sizes is also provided in~\cite{Xiao_2022}. In~\cite{Johnson_Pike-Burke_Rebeschini_2023}, the $\gamma$-rate linear convergence of PMD with adaptive step sizes is established. For the entropy regularized MDP setting, the $\gamma$-rate global linear convergence can be achieved with large constant step size~\cite{Lan_2021}. The analysis of PMD for regularized MDP can also be found in~\cite{Zhan_Cen_Huang_Chen_Lee_Chi_2021,Li_Zhao_Lan_2022}, where~\cite{Zhan_Cen_Huang_Chen_Lee_Chi_2021} provides a global linear convergence for any constant step size compared to~\cite{Lan_2021}, and~\cite{Li_Zhao_Lan_2022} uses a diminishing  regularization term to obtain the policy convergence. PMD has also been studied in the general policy parameterization and function approximation settings, yielding practical PMD algorithms~\cite{tomar2020mirror,yuan2023general}. In a recent work, \cite{chelu2024functional} seeks to accelerate PMD with the momentum technique by replacing  $Q^{\pi_k}$ with $[Q^{\pi_k}+\beta (Q^{\pi_k}-Q^{\pi_{k-1}})]$   in the update of the $k$-th iteration, where $\beta$ is a  parameter. 

\paragraph{Inexact PMD under Monte Carlo estimation} {The existing sample complexity analysis of PMD is mainly based on the generative model}~\cite{Lan_2021,Xiao_2022,Johnson_Pike-Burke_Rebeschini_2023,Li_Zhao_Lan_2022,protopapas2024policy}. The overall idea is first extending the analysis for the exact PMD to the inexact setting and then computing the number of samples needed to meet the error tolerance via concentration inequalities. The $\tilde{O}((1-\gamma)^{-8}\varepsilon^{-2})$ sample complexity for PMD to achieve the last iterate  $\varepsilon$-optimal  policy is obtained by~\cite{Xiao_2022,Johnson_Pike-Burke_Rebeschini_2023}, and the $\tilde{O}((1-\gamma)^{-5}\varepsilon^{-2})$ sample complexity to achieve an $\varepsilon$-optimal average-time convergence is provided in~\cite{Lan_2021}. In addition, the $O(\varepsilon^{-1})$  sampled complexity is established  for the regularized PMD in~\cite{Lan_2021}. {A $h$-PMD is proposed in~\cite{protopapas2024policy} by incorporating  a $h$-step lookahead via Bellman optimal operator over the action value into PMD. When $h$ is sufficiently large, $h$-PMD is proved to improve the $O((1-\gamma)^{-8}\varepsilon^{-2})$ sample complexity in~\cite{Xiao_2022} to $O((1-\gamma)^{-7}\varepsilon^{-2})$ under the generative model. Note that $h$-PMD still requires the explicit approximation of $V^{\pi_k}$ to desired accuracy (using Monte Carlo rollouts), while we show that even one-step TD evaluation suffices to achieve the same sample complexity.  The  sample complexity of PMD is also investigated in~\cite{Lan_2021,li2024stochastic} under the Markovian sampling scheme (CTD in Section 5.2 of~\cite{Lan_2021}). It is worth noting that CTD is not the TD evaluation considered in our paper. In a nutshell, what is done in~\cite{Lan_2021,li2024stochastic} is applying the (weighted, stochastic) Bellman iteration for a sufficiently number of iterations to explicitly approximate the true value function based on the Markov chain sampling, while we only apply one-step TD evaluation without explicitly approximating the true value function.}

\begin{table}[ht!]
    \centering
    %\makegapedcells
    %\setcellgapes{3pt}
    \caption{Comparison of TD-PMD with mostly related works.}
    \begin{tabular}{|c|c|c|c|c|c|}
        % \hline
        %  &  \cite{Xiao_2022} 
        %  & \cite{Johnson_Pike-Burke_Rebeschini_2023}
        %  & \cite{Lan_2021,li2024stochastic} 
        %  & \cite{vieillard2020leverage}
        %  & \cite{protopapas2024policy}
        %  \\ 
         %\toprule
         \hline
       \textbf{Algorithm} & \makecell[cc]{PMD \\ \cite{Xiao_2022}} & \makecell[cc]{PMD \\ \cite{Johnson_Pike-Burke_Rebeschini_2023}} & \makecell[cc]{PMD \\ \cite{Lan_2021, li2024stochastic}} & \makecell[cc]{DA-VI$(\lambda,0)$ \\ \cite{vieillard2020leverage}} & \makecell[cc]{$h$-PMD \\ \cite{protopapas2024policy}} 
         \\
         \hline
         \makecell[ccc]{\textbf{Require evaluate}\\ $V^\pi$ \textbf{or} $Q^\pi$ \textbf{to} \\ \textbf{desired accuracy?}} & Yes & Yes & Yes &Yes & Yes
         \\
         \hline
         {\textbf{Policy} \textbf{update}}
         & PMD & PMD & PMD & NPG & \makecell[cc]{Lookahead \\ PMD}
         \\
         \hline
         %\makecell[ccc]{\textbf{Convergence} \\ \textbf{rate in exact setting} \\ (constant step size)}
          \multicolumn{1}{|>{\columncolor{blue!10}}c|}{\tabularCenterstack{c}{\textbf{Convergence} \\ \textbf{rate in exact setting}\\ \textbf{(constant step size)}}}
         & 
         $O(1/T)$ & --- &   \makecell[cc]{\cite{Lan_2021} : $O(1/T)$  \\ \cite{li2024stochastic} : ---} & $O(1/T)$ & ---
         \\ 
        \hline
         %\makecell[ccc]{\textbf{Convergence} \\ \textbf{rate in exact setting} \\ (varying step size $\eta_k$)}
          \multicolumn{1}{|>{\columncolor{blue!10}}c|}{\tabularCenterstack{c}{\textbf{Convergence} \\ \textbf{rate in exact setting}\\ \textbf{(varying step size)}}}
         & $\gamma$-rate & $\gamma$-rate & ---  & --- & $\gamma$-rate\\
        %\midrule[.1em]
        \hline
         \multicolumn{1}{|>{\columncolor{red!15}}c|}{\tabularCenterstack{c}{\textbf{Sampling}  \textbf{model}}}
         & Generative & Generative & \makecell[ccc]{CTD to mimic \\ stationary \\ distribution} & --- & \makecell[ccc]{Additive \\ unbiased \\ noise}
         \\         
        \hline
         \multicolumn{1}{|>{\columncolor{red!15}}c|}{\tabularCenterstack{c}{\textbf{Last}  \textbf{iterate?}}}
         & Yes & Yes & No & --- & Yes
         \\
        \hline
         \multicolumn{1}{|>{\columncolor{red!15}}c|}{\tabularCenterstack{c}{\textbf{Sample} \\ \textbf{complexity}}}
         & $\tilde{O}(\varepsilon^{-2})$ & $\tilde{O}(\varepsilon^{-2})$ &  $\tilde{O}(\varepsilon^{-2})$  & --- & $\tilde{O}(\varepsilon^{-2})$
         \\          
        \hline\hline
         % &\cite{tsitsiklis2002convergence}
         % & \cite{winnicki2023convergence}
         % & \cite{murthy2023performance}
         % & \cite{khodadadian2022finite}
         % & \textbf{ours}
         % \\ 
         \textbf{Algorithm} & \makecell[cc]{OPI \\ \cite{tsitsiklis2002convergence}} & \makecell[cc]{h-OPI \\ \cite{winnicki2023convergence}} & \makecell[cc]{API \\ \cite{murthy2023performance}} & \makecell[cc]{NAC \\ \cite{khodadadian2022finite,xu2020improving}} & \makecell[cc]{TD-PMD \\ (ours)}
         \\
         \hline
         \makecell[ccc]{\textbf{Require evaluate}\\ $V^\pi$ \textbf{or} $Q^\pi$ \textbf{to} \\ \textbf{desired accuracy?} } & Yes & Yes & Yes & No & No 
         \\
         \hline
         {\textbf{Policy}  \textbf{update}}
         & Greedy &  \makecell[cc]{Lookahead \\ Greedy}
         &  \makecell[ccc]{Greedy / \\ Softmax / \\ NPG}
         & NPG & PMD
         \\
         \hline
          \multicolumn{1}{|>{\columncolor{blue!10}}c|}{\tabularCenterstack{c}{\textbf{Convergence} \\ \textbf{rate in exact setting}\\ \textbf{(constant step size)}}}
         %\makecell[ccc]{\textbf{Convergence} \\ \textbf{rate in exact setting} \\ (constant step size)}
         & --- & --- & \makecell[cc]{$\gamma$-rate to \\ a neighborhood}  & --- & ${O}(1/T)$
         \\ 
        \hline
         %\makecell[ccc]{\textbf{Convergence} \\ \textbf{rate in exact setting} \\ (varying step size )}
          \multicolumn{1}{|>{\columncolor{blue!10}}c|}{\tabularCenterstack{c}{\textbf{Convergence} \\ \textbf{rate in exact setting}\\ \textbf{(varying step size)}}}
         & --- & --- & --- & --- & $\gamma$-rate
         \\ 
        \hline
         \multicolumn{1}{|>{\columncolor{red!15}}c|}{\tabularCenterstack{c}{\textbf{Sampling} \textbf{model}}}
         & \makecell[ccc]{Additive \\ unbiased \\ noise} & \makecell[ccc]{Additive \\ unbiased \\ noise} & Markovian &  Markovian & Generative 
        \\         
        \hline
         \multicolumn{1}{|>{\columncolor{red!15}}c|}{\tabularCenterstack{c}{\textbf{Last}  \textbf{iterate?}}}
         & --- & Yes & Yes & No & Yes
         \\
        \hline
         \multicolumn{1}{|>{\columncolor{red!15}}c|}{\tabularCenterstack{c}{\textbf{Sample} \\ \textbf{Complexity}}} & --- & --- & --- &  \makecell[cc]{\cite{khodadadian2022finite}: $\tilde{O}({\varepsilon^{-6}})$ \\ \cite{xu2020improving}: $\tilde{O}(\varepsilon^{-3})$}  & $\tilde{O}(\varepsilon^{-2})$ 
         \\          
        \hline
       % \bottomrule
    \end{tabular}
    \label{tab:sample}
\end{table}

\paragraph{Detailed comparison of TD-PMD with mostly related works} 
{We compare our work with some mostly related ones in Table~\ref{tab:sample}, which includes not only those on PMD mentioned earlier, but also  other typical ones.
}
Compared with TD-PMD, DA-VI$(\lambda,0)$ in~\cite{vieillard2020leverage} performs the same PMD policy update but with the TD evaluation regularized by the KL divergence. The more general DA-MPI$(\lambda,\tau)$ further considers the entropy regularized policy update and a $m$-step (regularized) TD evaluation. Though DA-MPI covers various existing RL algorithms, its analysis is not applicable for TD-PMD because (1) regularized TD evaluation is applied in DA-MPI$(\lambda,\tau)$, which does not cover TD-PMD, and (2) its analysis relies on the regularization forms of KL divergence and entropy. OPI combined with $\text{TD}(\lambda)$ is studied in~\cite{tsitsiklis2002convergence} and the asymptotic convergence is established. The policy update of OPI is greedy while the policy update of TD-PMD is PMD. Even it can be viewed as a regularized greedy policy update, the analysis in~\cite{tsitsiklis2002convergence} does not apply, and we have established the global sublinear, linear convergence of TD-PMD, rather than the asymptotic analysis. For $h$-OPI in~\cite{winnicki2023convergence}, it generalizes OPI by applying a $h$-step lookahead over the state value in a policy update, and further applying a $m$-step Bellman operator to evaluate the policy. Approximate policy iteration (API) in~\cite{murthy2023performance} is a general algorithmic framework, where the authors also study a special case that PMD is used to update the policy and TD($\lambda$) is used to approximate the value function. It is noted that OPI, $h$-OPI and API all require to evaluate the value functions to some desired accuracy, which is the essential difference to our TD-NPG. Finally, the NAC algorithm investigated by works for example~\cite{khodadadian2022finite} considers NPG with TD evaluation under the online Markovian sampling scheme, and it does not require to evaluate $V^\pi$ or $Q^\pi$ to a desired accuracy. It is interesting to investigate the convergence of TD-PMD under such sampling scheme, which is one of our future directions. 
%%%%%%%%%%%%%
\subsection{Outline of this paper}
{This paper is organized as follows. In Section~\ref{sec:preliminary}, we provide a formal description of TD-PMD, together with more introduction on the MDP problem. Section~\ref{sec:exact-TD-PMD} contains the convergence results of TD-PMD under the exact policy gradient setting, while Section~\ref{sec:sample-complexity} presents the sample complexity of TD-PMD under the generative model is established. In Section~\ref{sec:conclusion} we conclude this paper with potential future directions.
}
%\paragraph{Key Technical Innovations.}
%Novel techniques are applied for the analysis of TD-PMD as there are two essential differences compared to the vanilla PMD: (1) the evaluation value $V^k$ is \textit{not} the value function of current policy $V^{\pi_k}$, and (2) the evaluation value $V^k$ is highly correlated to the initialization $V^0$. To get over these two obstacles, for sublinear analysis we first consider the ``good'' initialization. Under such case we show the monotonicity property of TD-PMD (different from the monotonicity decent property of vanilla PMD~\cite{Xiao_2022,Lan_2021}, see Remark~\ref{rem:monotnicity}) to establish the convergence results, which are later extended to the general initialization case. For linear analysis, we additionally explore the relation between the policy error $\| V^* - V^{\pi_k} \|$ and value error $\| V^* - V^k \|$. It is further extended to the inexact TD-PMD case where the monotonicity property does not hold.

%%%%%%%%%%%%%%%%%%%%

\section{Preliminary}
\label{sec:preliminary}

\paragraph{More about MDP.} In this paper we consider the tabular setting (i.e. $|\calS|, |\calA| < \infty$) and assume the reward function $r(s,a) \in [0,1]$ for simplicity.

Denote by $\Delta(\calS)$ and $\Delta(\calA)$ the probabilistic simplex on $\calS$ and $\calA$, respectively.  The set of all admissible policies is denoted by $\Pi = \ls\{ \pi=(\pi(\cdot|s))_{s\in\calS}: \; \pi(\cdot|s)\in\Delta(\calA)\rs\}$. It is well known that there exists a deterministic optimal policy $\pi^*$ (not necessarily unique) which simultaneously maximizes $V^\pi$ and $Q^\pi$ for all the states, and thus solves ~\eqref{eq:target-of-RL}, see for example~\cite{VI-PI}. The corresponding optimal state and action value functions are denoted by $V^*$ and $Q^*$, respectively.

The definition of the Bellman operator $\calT^\pi$ is already given in equation~\eqref{eq:Bellman-operator}, while the optimal Bellman operator $\calT$ is defined as
\begin{align*}
    \calT V(s) = \max_{\pi\in\Pi} \, \calT^\pi V(s) = \max_{a\in\calA} \, Q(s,a),
\end{align*}
where $Q$ is induced from $V$,
\begin{align}\label{eq:induced-Q}
Q(s,a) = r(s,a) + \gamma \cdot \E_{s^\prime \sim P(\cdot|s,a)}[V(s^\prime)].
\end{align}

%We call $\calT^\pi V - V$ the improvement of $\pi$ over value $V$. Similar to $V^\pi(\mu)$, we let $V(\mu) = \E_{s\sim\mu}[V(s)]$ for any $\mu\in\Delta(\calS)$. We also define 
The visitation measure $d^\pi_\mu \in \Delta(\calS)$ plays an important role in the analysis of the policy gradient methods, which is defined as
\begin{align*}
    d^\pi_\mu(s) = (1-\gamma) \sum_{t=0}^\infty \gamma^t \prob \ls( s_t = s \, | \, s_0 \sim \mu , \, \pi  \rs),
\end{align*}
where $\prob \ls( s_t = s \, | \, s_0 \sim \mu , \, \pi  \rs)$ is the probability of state $s$ being visited at time $t$ when starting at $s_0\sim\mu$ and following the policy $\pi$. {Finally, the optimal action value gap is defined as}
\begin{align*}
    \Delta= \min_{s\in\tilde{S}, \,\, a^\prime\not\in  \calA^*_s} \, [\max_{a\in\calA} \, Q^*(s,a) - Q^*(s,a^\prime)],
\end{align*}
{where $\calA^*_s = \arg\max_{a\in\calA} \, Q^*(s,a)$ is the optimal action set of state $s$ and $\tilde\calS = \{ s\in\calS: \; \calA^*_s\neq\calA \}$.}

The following general performance difference lemma (has also been used for example in~\cite{cheng2020policy,wagener2021safe,vieillard2020leverage})
shows that $V^\pi(\mu)-V(\mu)$ can also be expressed as the weighted sum of the one-step improvement even when $V$ is an arbitrary vector. 
This lemma can be proved similarly to the classical one~\cite{kakade2002approximately}. %. We adopt the extended variant of it (appear in for example~\cite{cheng2020policy,wagener2021safe}), which shows that $V^\pi - V$ can be represented as the weighted summation of the improvement. 
We include the proof in Appendix~\ref{sec:pf:lem:extended-pdl} for completeness.

\begin{lemma}[General performance difference lemma] \label{lem:extended-pdl}
    For policy $\pi \in \Pi$ and any  vector $V \in \mathbb{R}^{|\calS|}$, there holds 
    \[
        V^\pi(\mu) - V(\mu) = \frac{1}{1-\gamma} [\calT^\pi V - V](d^\pi_\mu),
    \]
where $[\calT^\pi V - V](d^\pi_\mu)=\sum\limits_sd^\pi_\mu(s)[\calT^\pi V(s) - V(s)]$.
\end{lemma}

\paragraph{Formal description of TD-PMD.} In contrast to PMD where the action value is evaluated completely based on the current policy,  TD-PMD keeps track of a sequence of TD evaluations of  the state value estimations which involves not only the current policy but also the state value estimation from the last iteration. More precisely, letting  $V^k$ be the current state value estimation, TD-PMD first updates $\pi_k$ using $Q^k$ induced from $V^k$ (see equation~\eqref{eq:induced-Q}),
\begin{align}
    \pi_{k+1}(\cdot|s) \!=\! \underset{p\in\Delta(\calA)}{\arg\max} \ls\{ \eta_k \langle p, \, Q^{k}(s,\cdot) \rangle \!-\! D^p_{\pi_k}(s)\rs\}.
    \label{eq:TD-PMD-policy-step}
\end{align}
{Note that here and in the rest of the paper, we will omit the subscript $h$  in $D_h$ and use the short notation $D^{\pi}_{\pi^\prime}(s)$ for $D_h(\pi(\cdot|s), \pi^\prime(\cdot|s))$.} Then instead of obtaining the new $V^{k+1}$ completely based on $\pi_{k+1}$,  TD-PMD only conducts one step of TD evaluation by applying $\mathcal{T}^{\pi_{k+1}}$ to $V^k$, that is
\begin{align}
    V^{k+1} = \calT^{\pi_{k+1}} V^k.
    \label{eq:TD-PMD-value-step}
\end{align}
A complete description of TD-PMD is presented in Algorithm~\ref{alg:TD-PMD}. Note that in this algorithm, the initial state value estimation $V^0$ can be any vector,  the initial policy $\pi_0$ can be any policy, and {\em $V^0$ is not necessarily the state value of $\pi_0$}.%Compared to the vanilla PMD, it does not require the complete evaluation of the current policy (i.e., the assess to $V^{\pi_k}$ or $Q^{\pi_k}$), but only applies a one-step TD evaluation.
\begin{algorithm}[tb]
   \caption{TD-PMD}
   \label{alg:TD-PMD}
\begin{algorithmic}
   \STATE {\bfseries Input:} Initial state value estimation $V^0$, initial policy $\pi_0$, iteration number $T$, step size $\{ \eta_k \}$.
   \FOR{$k=0$ {\bfseries to} $T-1$}
   \vspace{0.2cm}
   \STATE {\bfseries (Policy improvement)} 
   Compute $Q^k$ by 
   \begin{align*}
    Q^k\left( s,a \right) = \mathbb{E} _{s^{\prime}\sim P\left( \cdot |s,a \right)}\left[ r\left( s,a \right) +\gamma V^k\left( s^{\prime} \right) \right],
   \end{align*}
   and update the policy by
   \begin{align*}
       \pi_{k+1}(\cdot|s) \!=\! \underset{p\in\Delta(\calA)}{\arg\max} \ls\{ \eta_k \langle p, \, Q^{k}(s,\cdot) \rangle \!-\! D^p_{\pi_k}(s) \rs\}.\numberthis\label{eq:alg1pi}
   \end{align*}
   \STATE {\bfseries (TD evaluation)} Update the state value estimation by
   \begin{align*}
       V^{k+1} = \calT^{\pi_{k+1}} V^k.\numberthis\label{eq:alg1v}
   \end{align*}
   \ENDFOR
   \STATE {\bfseries Output:} Last iterate policy $\pi_T$, last iterate state value estimation $V^T$.
\end{algorithmic}
\end{algorithm}

\begin{remark}
{Note that in Algorithm~\ref{alg:TD-PMD}, we maintain a list of $V^k$ and compute $Q^k$ from $V^k$. In contrast, we can also maintain a list of $Q^k$ directly and similar convergence results can be established, see Appendix~\ref{sec:Q-TD-PMD} for   Q-TD-PMD.}
\end{remark}

\section{Convergence of Exact TD-PMD}
\label{sec:exact-TD-PMD}
{This section  investigates the convergence of TD-PMD in the exact setting when   $Q^k$ and $\calT^{\pi_{k+1}} V^k$ can be computed exactly, including the sublinear convergence with constant step size, $\gamma$-rate linear convergence with adaptive step size, and the policy convergence of TD-PQA and TD-NPG.
\begin{itemize}
    \item For the analysis of sublinear convergence, a key challenge  is that   $V^k$ is not the state value of $\pi_k$ (i.e., $V^k\neq V^{\pi_k}$) and hence $Q^k\neq Q^{\pi_k}$, and thus the standard analysis for PMD in for example \cite{Xiao_2022} does not apply directly. In order to address this issue, we need to discuss two different cases of the initialization and construct an auxiliary sequence for the general case.
    \item The analysis of the $\gamma$-rate convergence is inspired by that in \cite{Johnson_Pike-Burke_Rebeschini_2023}, but the details are different and essentially rely on the update rule of TD-PMD. Specifically, the linear convergence is firstly established for the  sequence $\{V^{T}\}$ and then is extended to the  sequence $\{V^{\pi_T}\}$ based on the relation between $V^T$ and $V^{\pi_T}$.
    \item The establishment of the policy convergence of TD-PMD and TD-NPG is highly non-trivial. It requires us to conduct an elementary analysis over the probability of the optimal and suboptimal   actions explicitly based on the policy update formula. Moreover, the proofs for TD-PQA/TD-NPG and Q-TD-PQA/Q-TD-NPG (i.e., the algorithms that maintain the Q value directly, see Appendix~\ref{sec:Q-TD-PMD}) are essentially different as convergence in terms of the state values cannot be obtained for the latter ones.
\end{itemize}
}

\subsection{Sublinear convergence}
\label{sec:sublinaer-convergence}
%Our first result is the following sublinear convergence under arbitrary constant step size.
%  Roughly speaking,  the following dimension-free sublinear convergence is established for TD-PMD with any constant step size.
% \begin{theorem}[Informal version of sublinear convergence]
%     Consider TD-PMD with constant step size $\eta_k=\eta>0$. For any $\pi_0$ and $V^0$, one has 
%     \begin{align*}
%         \ls\| V^* - V^{\pi_T} \rs\|_\infty &\leq \frac{C_0}{T+1}, \;\;\;\; \ls\| V^* - V^T \rs\|_\infty \leq \frac{C_0}{T+1} + \gamma^T C_0^\prime,
%     \end{align*}
%     where $C_0$ and $C_0^\prime$ are non-negative constants that are related to the initialization and the step size.
%     \label{thm:sublinear-convergence-of-exact-TD-PMD-informal}
% \end{theorem} 

% The formal description of the sublinear convergence  can be found in Theorem~\ref{thm:sublinear-convergence-of-exact-TD-PMD-formal}. A key challenge in the establishment of this result is that overall  $V^k$ is not the state value of $\pi_k$ (i.e., $V^k\neq V^{\pi_k}$) and hence $Q^k\neq Q^{\pi_k}$. Therefore, the standard analysis for PMD in for example \cite{Xiao_2022} does not apply directly. 
%As already mentioned, we need to discuss two different cases based on the initialization to address this issue. %Specifically, we need to discuss two different initializations:
The following two cases have to be discussed in order to establish the sublinear convergence of TD-PMD.

\begin{itemize}
    \item ``Good'' initialization that satisfies $\calT^{\pi_0} V^0 \geq V^0$: %Consequently, the constant $C^\prime_0=0$ in Theorem \ref{thm:sublinear-convergence-of-exact-TD-PMD-informal}. 
    In this case, we can establish the monotonicity property of $V^k$ and further show that  $V^{\pi_k}\geq V^k$ by induction. Combining these with the standard PMD analysis (see for example~\cite{Xiao_2022,Lan_2021}) yields the sublinear convergence.% In this case, one also has $C^\prime_0=0$ in Theorem~\ref{thm:sublinear-convergence-of-exact-TD-PMD-informal}.
    \item General initialization where $V_0$ and $\pi_0$ may not satisfy $\calT^{\pi_0} V^0 \geq V^0$:  %and the constant $C^\prime_0 > 0$ in Theorem \ref{thm:sublinear-convergence-of-exact-TD-PMD-informal} thus enjoys a slower convergence compared with the good initialization case. For such initializations, 
    The key idea here is to construct another initialization $\{\tilde{V}^0, \tilde{\pi}_0\}$ that satisfies $\calT^{\tilde{\pi}_0}\tilde{V}^0 \geq \tilde{V}^0$ by adding a constant shift and then show that this will not change the iterates of the policy. To this end,  we need to  build up the  relation between $\{ \tilde{V}^T, \tilde{\pi}_T \}$ and $\{ V^T, \pi_T \}$, so that the convergence results of $\{ \tilde{V}^T, \tilde{\pi}_T \}$ can be translated to $\{ V^T, \pi_T \}$. %It is also worth noting that $C^\prime_0 > 0$ in this case, and this basically means that TD-PMD has a slower convergence for the ``bad'' initialization which is natural.
\end{itemize}
%In the following, we discuss these two aspects in details.

\subsubsection{Sublinear convergence for initialization satisfying $\mathcal{T}^{\pi_0}V_0\geq V_0$}
\label{sec:sublinear-convergence-spec-init}
{We first perform a standard PMD analysis~\cite{Xiao_2022,Lan_2021} using the three-point descent lemma~\cite{chen1993convergence} and the general performance difference lemma (Lemma~\ref{lem:extended-pdl}) to obtain the following useful result.}
\begin{lemma}
    Consider TD-PMD with constant step size $\eta_k=\eta>0$ (Algorithm~\ref{alg:TD-PMD}). There holds
    \begin{align}
        \frac{1}{T}\sum_{k=0}^{T-1} [V^* - V^k](\mu) &\leq \frac{1}{T(1-\gamma)} [V^{T} - V^0](d^*_\mu) + \frac{1}{T\eta(1-\gamma)}[D^{\pi^*}_{\pi_0} - D^*_{\pi_{T}}](d^*_\mu), \label{eq:TD-PMD-error-upper-bound-1}
    \end{align}
    where $[D^{\pi^*}_{\pi_{0}} - D^{\pi^*}_{\pi_T}](d^*_\mu) = \E_{s\sim d^*_\mu} [D^{\pi^*}_{\pi_{0}}(s) - D^{\pi^*}_{\pi_T}(s)]$.
    \label{lem:sublinear-convergence-lemma}
\end{lemma}

The proof of Lemma~\ref{lem:sublinear-convergence-lemma} is presented in Appendix~\ref{sec:pf:lem:sublinear-convergence-lemma}. Noting that for PMD, there holds $V^* \geq V^k = V^{\pi_k}\geq V^{\pi_{k-1}}$, so the left hand side of equation~\eqref{eq:TD-PMD-error-upper-bound-1} is greater or equal than $V^*-V^{\pi_T}$, from which the sublinear convergence can be further obtained. However, for TD-PMD, even $V^* \geq V^k$ may not hold any more, for example consider an extreme case where $V^*\ll V^0$. %Intuitively, it is related to the initialization as $V^0$ can be set as $V^* \ll V^0$. The following lemma gives the answer that under the initialization of $\calT^{\pi_0} V^0 \geq V^0$, we have $V^* \geq V^k$ even $V^k \neq V^{\pi_k}$.
Despite this, the following lemma shows that $V^*\geq V^k$ can be guaranteed if $\calT^{\pi_0} V^0 \geq V^0$.
\begin{lemma}
Consider TD-PMD with constant step size $\eta_k=\eta>0$. Assume the initialization satisfies $\calT^{\pi_0} V^0 \geq V^0$.    Then, for $k=0,1,...,T-1$, there holds    
    \begin{align}
        V^* \geq V^{\pi_{k+1}} \geq V^{k+1} = \calT^{\pi_{k+1}}V^k \geq \calT^{\pi_k} V^k \geq V^k.
        \label{eq:monotonicity-property}
    \end{align}
    \label{lem:monotonicity-property}
\end{lemma} 
We will prove Lemma~\ref{lem:monotonicity-property} by induction using the properties of Bellman operator, see Appendix~\ref{sec:pf:lem:monotonicity-property} for details.
%Note that it suffices to show the second and the last  inequalities of equation~\eqref{eq:monotonicity-property} since the third one $\calT^{\pi_{k+1}}V^k \geq \calT^{\pi_k}V^k$ follows directly from equation~\eqref{eq:TD-PMD-three-point-1} in Lemma~\ref{lem:TD-PMD-three-point}.
%The last inequality $\calT^{\pi_k} V^k \geq V^k$ is given by noting $V^k = \calT^{\pi_k} V^{k-1}$ and then using the induction assumption $V^{k} \geq V^{k-1}$. For the second inequality $V^{\pi_{k+1}} \geq V^{k+1}$, we show that $\calT^{\pi_{k+1}}V^{k+1} \geq V^{k+1}$ and then apply the extended performance difference lemma (Lemma~\ref{lem:extended-pdl}). The complete proof is given in Appendix~\ref{sec:pf:lem:monotonicity-property}.

%The following monotonicity property of value $V^k$ is directly implied from Lemma~\ref{lem:monotonicity-property}.
%\begin{corollary}
%    Under the initialization of $\calT^{\pi_0} V^0 \geq V^0$, 
%    \begin{align}
%        V^* \geq V^{T} \geq V^{T-1} \geq \cdots \geq V^k \geq \cdots \geq V^0.
 %       \label{eq:monotonicity-property-of-value}
%    \end{align}
%    \label{cor:monotonicity-property-of-value}
%\end{corollary}
\begin{remark}
Lemma~\ref{lem:monotonicity-property} shows that $V^k,\, k = 0,\cdots, T$ is a monotonic sequence. However, this is not true for the state values of $\pi_k$. For example, consider  the initialization  $\pi_0 = \pi^*$ but $V^0 \neq V^*$. In such case, $V^{\pi_0} = V^*$ but ${\pi_1}$ might not be optimal as it is obtained through the non-optimal $V^0$.
    In contrast, for PMD one always has $V^{\pi_{k+1}} \geq V^{\pi_k}$~\textup{\cite{Xiao_2022,Lan_2021}}. 
    \label{rem:monotnicity}
\end{remark}

\begin{remark}
    Lemma~\ref{lem:monotonicity-property} also shows that $V^{\pi_k}\geq V^k$ for $k=1,\cdots, T$. Note that this is also true for $k=0$ under the condition $\mathcal{T}^{\pi_0}V^0\geq V^0$, which can be proved by a direct application of the general performance difference lemma.
\end{remark}
%Another corollary of Lemma~\ref{lem:monotonicity-property} can be obtained. This result gives the relation between the TD evaluation value $V^k$ and the exact policy evaluation value $V^{\pi_k}$.
%\begin{corollary}
%    Under the initialization of $\calT^{\pi_0} V^0 \geq V^0$, at any $k = 0, ..., T-1$ iteration of TD-PMD there holds
%    \begin{align}
%        V^* \geq V^{\pi_k} \geq V^k.
%        \label{eq:monotonicity-property-of-policy-geq-value}
%    \end{align}
%    \label{cor:monotonicity-property-of-policy-geq-value}
%\end{corollary}
%It is direct to prove Corollary~\ref{cor:monotonicity-property-of-policy-geq-value} by applying the extended performance difference lemma (Lemma~\ref{lem:extended-pdl}) to the result of $\calT^{\pi_k} V^k \geq V^k$ in Lemma~\ref{lem:monotonicity-property}.

%After establishing the monotonicity property and the relation between $V^k$ and $V^{\pi_k}$, we are ready to derive the sublinear convergence under such specific initialization from Eq~\eqref{eq:TD-PMD-error-upper-bound-1}. The complete proof is given in Appendix~\ref{sec:pf:thm:sublinear-convergence-value-spec-init}.
By combining Lemma~\ref{lem:monotonicity-property} and equation~\eqref{eq:TD-PMD-error-upper-bound-1}, we are able to show the sublinear convergence of TD-PMD under the  specific initialization, as presented in the following theorem.  The detailed proof  can be found in Appendix~\ref{sec:pf:thm:sublinear-convergence-value-spec-init}.
\begin{theorem}
 Consider TD-PMD with constant step size $\eta_k=\eta>0$. Assume the initialization satisfies $\calT^{\pi_0} V^0 \geq V^0$. Then, one has
    \begin{align*}
        &\phantom{\leq\,\,\,}\ls\| V^* - V^{\pi_T} \rs\|_\infty \leq \ls\| V^* - V^T \rs\|_\infty \leq \frac{1}{T+1} \ls( \frac{1}{(1-\gamma)^2} \!+\! \frac{\ls\| V^0 \rs\|_\infty}{1-\gamma} \!+\! \frac{\ls\| D^{\pi^*}_{\pi_0} \rs\|_\infty}{\eta(1-\gamma)} \rs),
    \end{align*}
    where $\ls\| D^{\pi^*}_{\pi_0} \rs\|_\infty = \max_{s\in\calS} \, D^{\pi^*}_{\pi_0}(s)$.
    \label{thm:sublinear-convergence-value-spec-init}
\end{theorem}

%%%%%%%%%%%%%%%%%%%
\subsubsection{Sublinear convergence for general initialization}
\label{sec:Sublinear Convergence for General Initialization}

Next we discuss the case of general initialization, in which $\calT^{\pi_0} V^0 \geq V^0$ may not hold. The key idea is to construct another initialization $\{ \tilde{V}^0, \tilde{\pi}_0 \}$ based on $\{ V^0, \pi_0 \}$ such that $\calT^{\tilde{\pi}_0} \tilde{V}^0 \geq \tilde{V}^0$, and then show that the sequences produced based on these two pair of initializations approach each other. 

Specifically, the new initialization is constructed in the following way:
\begin{align}
    \tilde{V}^0 = V^0 - \kappa_0 \cdot \mathbf{1}, \quad \tilde{\pi}_0 = \pi_0,
    \label{eq:init}
\end{align}
where 
\[\kappa_0 :=  \max \; \ls\{0, \;\; \frac{1}{1-\gamma} \max_{s\in\calS} \, [V^0 - \calT^{\pi_0}V^0](s) \rs\} , \]
and $\mathbf{1}$ is a vector whose entries are all one. Note that $\kappa_0=0$ if and only if $\calT^{\pi_0} V^0 \geq V^0$. %We denote $\{ \tilde{V}^k, \tilde{\pi}_k \}_{k=0}^T$ the sequence generated by TD-PMD with such initialization and the same constant step size. It can be shown that the initialization $\{ \tilde{V}^0, \, \tilde{\pi}_0 \}$ is ``good'', i.e. $\calT^{\tilde{\pi_0}} \tilde{V}^0 \geq \tilde{V}^0$. Hence, the convergence results in Section~\ref{sec:sublinear-convergence-spec-init} can be applied for $\{ \tilde{V}^T, \tilde{\pi}_T \}$.
Moreover, one can show that $\calT^{\tilde{\pi}_0} \tilde{V}^0 \geq \tilde{V}^0$.

\begin{proposition}
    Consider the initialization in equation~\eqref{eq:init}, there holds $\calT^{\tilde{\pi}_0} \tilde{V}^0 \geq \tilde{V}^0$.
    \label{pro:init-is-good}
\end{proposition}
The proof of Proposition~\ref{pro:init-is-good} is given in Appendix~\ref{sec:pf:pro:init-is-good}.

Let $\{ \tilde{V}^k, \tilde{\pi}_k \}_{k=0}^T$ be the sequence generated by TD-PMD starting from $\{\tilde{V}^0,\tilde{\pi}_0\}.$  By Theorem~\ref{thm:sublinear-convergence-value-spec-init}, sublinear convergence can be established for this sequence. If we can show that the two sequences will approach each other as $T\rightarrow \infty$, the sublinear convergence can also be obtained for $\{{V}^k, {\pi}_k \}_{k=0}^T$. In fact, 
there is an   exact relation between $\{ V^T, \pi_T \}$ and $\{ \tilde{V}^T, \tilde{\pi}_T \}$, as presented in the next lemma.

\begin{lemma}
    For any $k=0, ..., T$, there holds
    \begin{align}
        V^k = \tilde{V}^k + \gamma^k \cdot\kappa_0 \cdot \mathbf{1}, \quad \pi_k = \tilde{\pi}_k.
        \label{eq:relation-between-two-sequences}
    \end{align}
    \label{lem:relation-between-two-sequences}
\end{lemma}
We will prove Lemma~\ref{lem:relation-between-two-sequences}  by induction, see Appendix~\ref{sec:pf:lem:relation-between-two-sequences} for details. Together with Proposition~\ref{pro:init-is-good} and Theorem~\ref{thm:sublinear-convergence-value-spec-init}, the sublinear convergence of the exact TD-PMD can be established for any initialization.

\begin{theorem}[Sublinear convergence]
  Consider TD-PMD with constant step size $\eta_k=\eta>0$. For any $\{V^0,\pi_0\}$, one has
    \begin{align*}
        \ls\| V^* - V^{\pi_T} \rs\|_\infty &\leq \frac{1}{T+1} \ls( \frac{1}{(1-\gamma)^2} \!+\! \frac{\ls\| V^0 \rs\|_\infty + \kappa_0 }{1-\gamma} \!+\! \frac{\ls\| D^{\pi^*}_{\pi_0} \rs\|_\infty}{\eta(1-\gamma)} \rs)
        \numberthis \label{eq:sublinear-convergence-of-output-policy} \\
        \ls\| V^* - V^T \rs\|_\infty &\leq \frac{1}{T+1} \ls( \frac{1}{(1-\gamma)^2} \!+\! \frac{ \ls\| V^0 \rs\|_\infty + \kappa_0}{1-\gamma} \!+\! \frac{\ls\| D^{\pi^*}_{\pi_0} \rs\|_\infty }{\eta(1-\gamma)} \rs) + \gamma^T \kappa_0.
        \numberthis \label{eq:sublinear-convergence-of-output-value}
    \end{align*}
    \label{thm:sublinear-convergence-of-exact-TD-PMD-formal}
\end{theorem}
The proof of Theorem~\ref{thm:sublinear-convergence-of-exact-TD-PMD-formal} is presented in Appendix~\ref{sec:pf:thm:sublinear-convergence-of-exact-TD-PMD-formal}.

%\begin{remark}
%    Theorem~\ref{thm:sublinear-convergence-of-exact-TD-PMD-formal} shows that both the evaluation value $V^T$ and the value function of output policy $V^{\pi_T}$ converge to the optimal value function $V^*$. That is, both of the output policy $\pi_T$ and the policy greedily induced by $V^T$ are asymptotically optimal.
%\end{remark}

\begin{remark}
%    If one set the initialization with $V^0 = 0$, then the sublinear convergence result of $V^{\pi_T}$ (i.e. Eq~\eqref{eq:sublinear-convergence-of-output-policy}) is the same with the result for vanilla PMD in~\cite{Xiao_2022}. That means a one-step TD evaluation is sufficient for PMD to achieve the sublinear convergence. 
In particular, if $V^0 = 0$, one has $\kappa_0=0$. Then, the result of equation~\eqref{eq:sublinear-convergence-of-output-policy} is exactly the same as the sublinear convergence rate obtained for PMD in~\textup{\cite{Xiao_2022}}, which is surprising since PMD requires an exact evaluation of the action value $Q^{\pi_k}$ in each iteration but TD-PMD only requires a one-step TD evaluation. 
    Note that for  general initialization, $\{V^k\}$ is no longer guaranteed to be a monotonic sequence. To see this, consider the case $V^0 \gg V^*$. As $V^T \to V^*$, it is impossible that $V^{k+1} \geq V^k$ holds for all $k$.
    %Numerical experiments have been conducted to verify the above two different  types of sublinear convergence of TD-PMD, as well as to compare the performance between TD-PMD and PMD, see Appendix~\ref{sec:numerical-simulations}.
\end{remark}

{\begin{remark}
    Note that the idea in TD-PMD by using the TD evaluation is indeed  natural
and similar ideas also appear for example in \textup{\cite{winnicki2023convergence,murthy2023performance,Geist_Scherrer_Pietquin_2019}}. In particular, TD-PMD can be viewed as an instance of the general modified policy iteration approach   studied in \textup{\cite{Geist_Scherrer_Pietquin_2019}}. However, we would like to emphasize that  only the convergence in terms of the regret \textup{(}i.e., the average error\textup{)} has been established in \textup{\cite{Geist_Scherrer_Pietquin_2019}}. In fact, it is conjectured in \textup{\cite{Geist_Scherrer_Pietquin_2019}} that ``The
convergence rate of the loss of MD-MPI is an open question, but a sublinear rate is quite possible''. Therefore,  we have indeed provided an affirmative answer to this open question. 
Note that the results in Theorem~\ref{thm:sublinear-convergence-of-exact-TD-PMD-formal} can also be extended to the case  where 
equation~\eqref{eq:alg1v} in Algorithm~\ref{alg:TD-PMD} is replaced by the TD\textup{(}$n$\textup{)} evaluation
\[
V^{k+1}=\left(\mathcal{T}^{\pi_{k+1}}\right)^n V^k,
\] leading to the following convergence results
        \begin{align*}
        \ls\| V^* - V^{\pi_T} \rs\|_\infty &\leq \frac{1}{T+1} \ls( \frac{1}{(1-\gamma)^2} \!+\! \frac{\ls\| V^0 \rs\|_\infty}{1-\gamma} \!+\! \frac{\ls\| D^{\pi^*}_{\pi_0} \rs\|_\infty}{\eta(1-\gamma)} \rs),\\
        \ls\| V^* - V^T \rs\|_\infty &\leq \frac{1}{T+1} \ls( \frac{1}{(1-\gamma)^2} \!+\! \frac{ \ls\| V^0 \rs\|_\infty + \kappa_0}{1-\gamma} \!+\! \frac{\ls\| D^{\pi^*}_{\pi_0} \rs\|_\infty }{\eta(1-\gamma)} \rs) + \gamma^{{Tn}} \kappa_0.
    \end{align*}
Moreover,  a similar sublinear convergence result also holds for 
    the more general TD\textup{(}$\lambda$\textup{)} evaluation where equation~\eqref{eq:alg1v}
     is replaced by the TD\textup{(}$\lambda$\textup{)} evaluation
\[
V^{k+1}=\mathcal{T}_\lambda^{\pi_{k+1}}V^k,
\]
with
\begin{align*}
    \mathcal{T}^\pi_{\lambda}=\mathbb{E}_{n\sim \text{Geo}(1-\lambda)}\left[(\mathcal{T}^\pi)^n\right].
\end{align*}
For this case, one has
\begin{align*}
        \ls\| V^* - V^{\pi_T} \rs\|_\infty &\leq \frac{1}{T+1} \ls( \frac{1}{(1-\gamma)^2} \!+\! \frac{\ls\| V^0 \rs\|_\infty}{1-\gamma} \!+\! \frac{\ls\| D^{\pi^*}_{\pi_0} \rs\|_\infty}{\eta(1-\gamma)} \rs),\\
        \ls\| V^* - V^T \rs\|_\infty &\leq \frac{1}{T+1} \ls( \frac{1}{(1-\gamma)^2} \!+\! \frac{ \ls\| V^0 \rs\|_\infty + \kappa_0}{1-\gamma} \!+\! \frac{\ls\| D^{\pi^*}_{\pi_0} \rs\|_\infty }{\eta(1-\gamma)} \rs) + \left(\frac{1-\lambda}{1-\lambda\gamma}\right)^T\cdot \gamma^T\kappa_0.
    \end{align*}
Since the proofs are overall similar, we omit the details. 
\end{remark}
}
%%%%%%%%%%%%%%%%%%%%
\subsubsection{Empirical Validation}
\label{sec:numerical-simulations}
Here we conduct numerical experiments to justify the correctness of the theoretical results in Section~\ref{sec:sublinaer-convergence}, 
 {as well as to compare the empirical performance of PMD and TD-PMD}. Tests are conducted on random MDPs with $|\calS| = 50$, $|\calA| = 10$ and $\gamma=0.95$. The reward $r(s,a)$ and transition probability $P(s^\prime | s,a)$ are uniformly generated from $[0,1]$ ($P$ is further normalized to be a probability matrix). The initial state  is uniformly distributed over $\calS$. Two instances of TD-PMD (i.e., TD-PQA and TD-NPG) mentioned in Section~\ref{sec:TD-PQA-and-TD-NPG} are tested.
% We consider two specific algorithms that belong to TD-PMD. 
% The first one is TD-PQA which uses $
% D_{\pi _k}^{p}\left( s \right) =\frac{1}{2}\left\| p-\pi _k\left( \cdot |s \right) \right\| _{2}^{2}$ as the Bregman divergence. The corresponding policy update rule is
% \begin{align*}
% \left( \mathrm{TD}\mathrm{-}\mathrm{PQA} \right) \quad \forall k\in \mathbb{N} :s\in \mathcal{S} \,:\pi _{k+1}\left( \cdot |s \right) =\mathrm{Proj}_{\Delta \left( \mathcal{A} \right)}\left( \pi _k\left( \cdot |s \right) +\eta\, Q^k\left( s,\cdot \right) \right),
% \end{align*}
% where $\mathrm{Proj}_{\Delta \left( \mathcal{A} \right)}\left( \cdot \right)$ denotes the projection onto the probability simplex  $\Delta(\mathcal{A})$.
% The second one is TD-NPG which uses $
% D_{\pi _k}^{p}\left( s \right) =\mathrm{KL}\left( p,\pi _k\left( \cdot |s \right) \right)$ as the Bregman divergence in \eqref{eq:alg1pi}. The corresponding policy update rule is 
% \begin{align*}
% \left( \mathrm{TD}\mathrm{-}\mathrm{NPG} \right) \quad \forall k\in \mathbb{N} : s\in \mathcal{S} ,a\in \mathcal{A} \,\,: \pi _{k+1}\left( a|s \right) =\frac{\pi _k\left( a|s \right) \cdot \exp \left\{ \eta\,  Q^k\left( s,a \right) \right\}}{\mathbb{E} _{a^{\prime}\sim \pi _k\left( \cdot |s \right)}\left[ \exp \left\{ \eta\, Q^k\left( s,a^{\prime} \right) \right\} \right]}.
% \end{align*}

\paragraph{Sublinear convergence for ``good'' initialization}  We initialize $\pi_0$ as the uniform policy and $V^0=\textbf{0}$. It's easy to verify that $
\mathcal{T} ^{\pi _0}V^0\ge V^0$ since the rewards in the experiments are non-negative. The step size is set to be $\eta=0.1$ for both TD-PQA and TD-NPG. The simulation results are presented in Figure \ref{fig:sub1-td}. It can be observed from the figure that the blue curve (representing $\|V^*-V^T\|_\infty$) is always on top of the yellow curve (representing $\|V^*-V^{\pi_T}\|_\infty$). Additionally, both curves decay to zero, which demonstrates the correctness of Lemma \ref{lem:monotonicity-property} and Theorem \ref{thm:sublinear-convergence-value-spec-init}.

\begin{figure}[ht!]
    \subfigure[TD-PQA]{
        \begin{minipage}[t]{0.45\linewidth}
            \centering
            \includegraphics[width=1\linewidth]{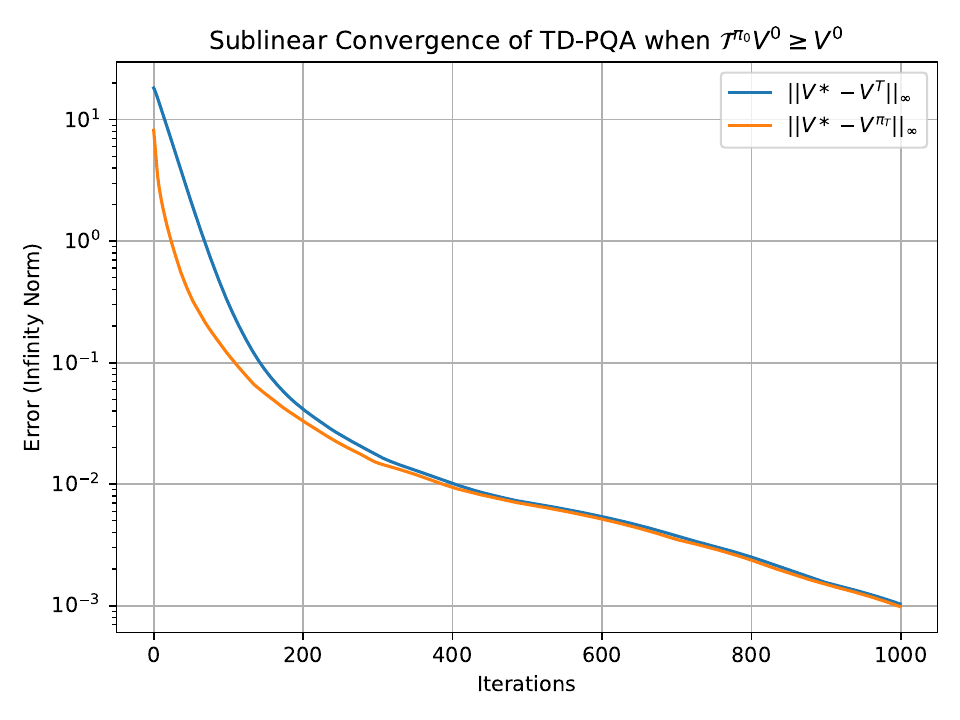}
            \label{fig:sub1-td-pqa}
        \end{minipage}
    }
    \subfigure[TD-NPG]{
        \begin{minipage}[t]{0.45\linewidth}
            \centering
            \includegraphics[width=1\linewidth]{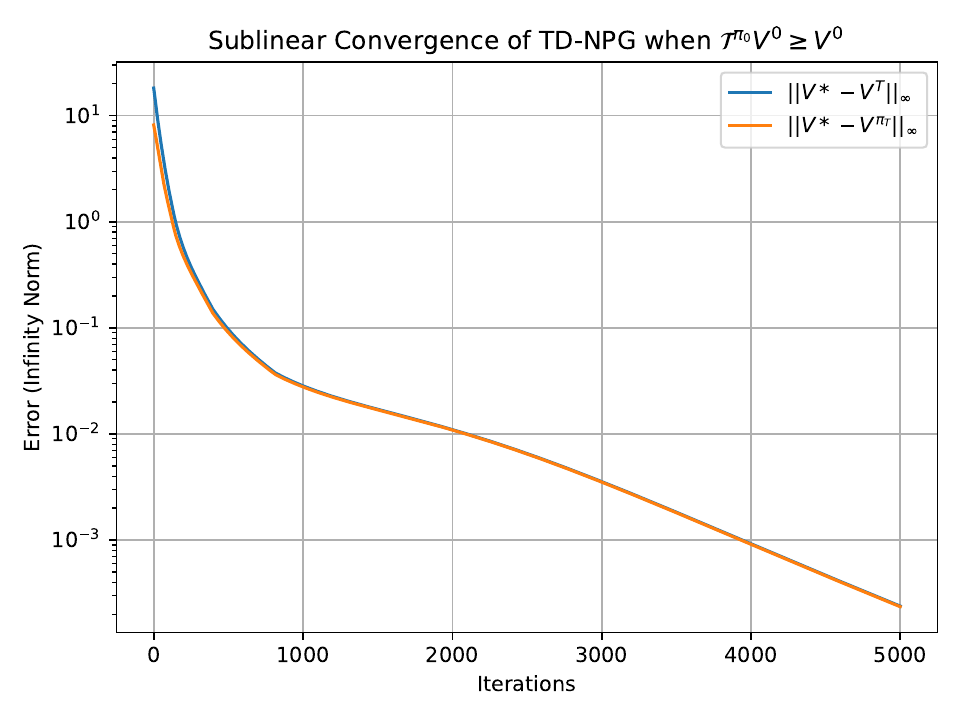}
            \label{fig:sub1-td-npg}
        \end{minipage}
     }
     \centering
    \caption{Sublinear convergence of the exact TD-PQA (a) and TD-NPG (b) with constant step size $\eta=0.1$ on a random MDP. Here, the initialization satisfies $
\mathcal{T} ^{\pi _0}V^0\ge V^0$. }
    \label{fig:sub1-td}
\end{figure}

\paragraph{Sublinear convergence for  general initialization} %We now run experiments of both TD-NPG and TD-PQA with general initializations that the condition $
%\mathcal{T} ^{\pi _0}V^0\ge V^0$ may not met.
We initialize $\pi_0$ as the uniform policy and sample $V^0$   randomly from the uniform distribution between 0 and $\frac{1}{1-\gamma}$. As above, the  step size $\eta=0.1$ is used. The simulation results are presented in Figure \ref{fig:sub2-td}. It can be observed from the figure that the value error of the original value function $V$( see blue curve) may not satisfy the monotonic property if $V_0$ is not a good initialization. Nevertheless, after adding a constant shift to $V_0$ so that the new initialization $\tilde{V}_0$ satisfies
$\mathcal{T}^{\pi_0}\tilde{V}^0\geq \tilde{V}^0$,  the value error sequence induced by $\tilde{V}_0$ (see green curve) is monotonic. It is also worth noting that all three curves converges to the same limit point, demonstrating the correctness of Theorem \ref{thm:sublinear-convergence-of-exact-TD-PMD-formal}.

\begin{figure}[ht!]
    \subfigure[TD-PQA]{
        \begin{minipage}[t]{0.45\linewidth}
            \centering
            \includegraphics[width=1\linewidth]{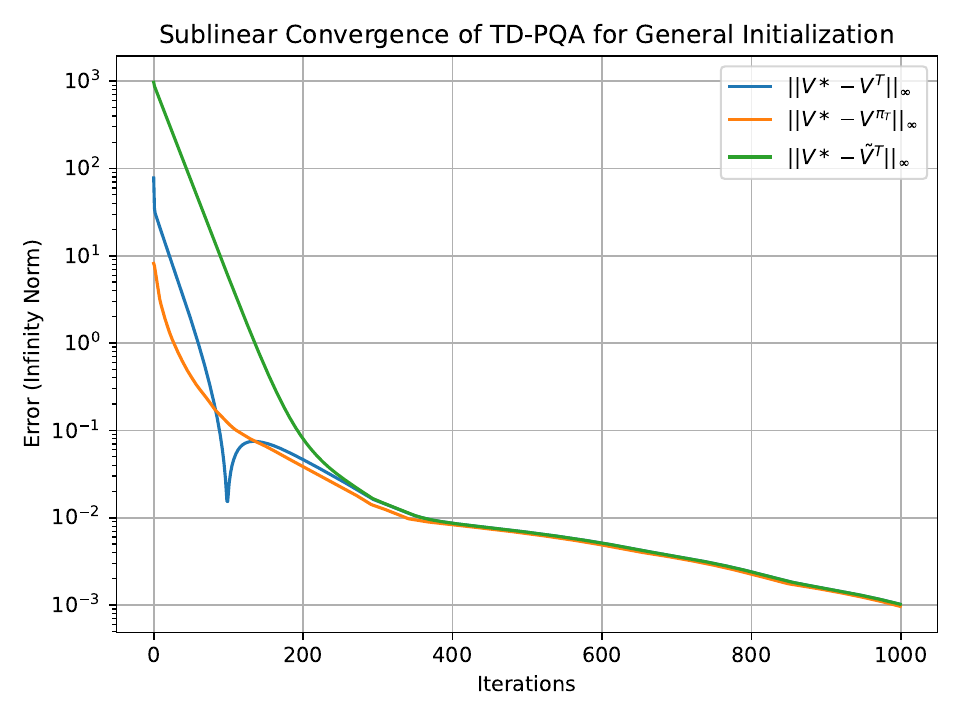}
            \label{fig:sub2-td-pqa}
        \end{minipage}
    }
    \subfigure[TD-NPG]{
        \begin{minipage}[t]{0.45\linewidth}
            \centering
            \includegraphics[width=1\linewidth]{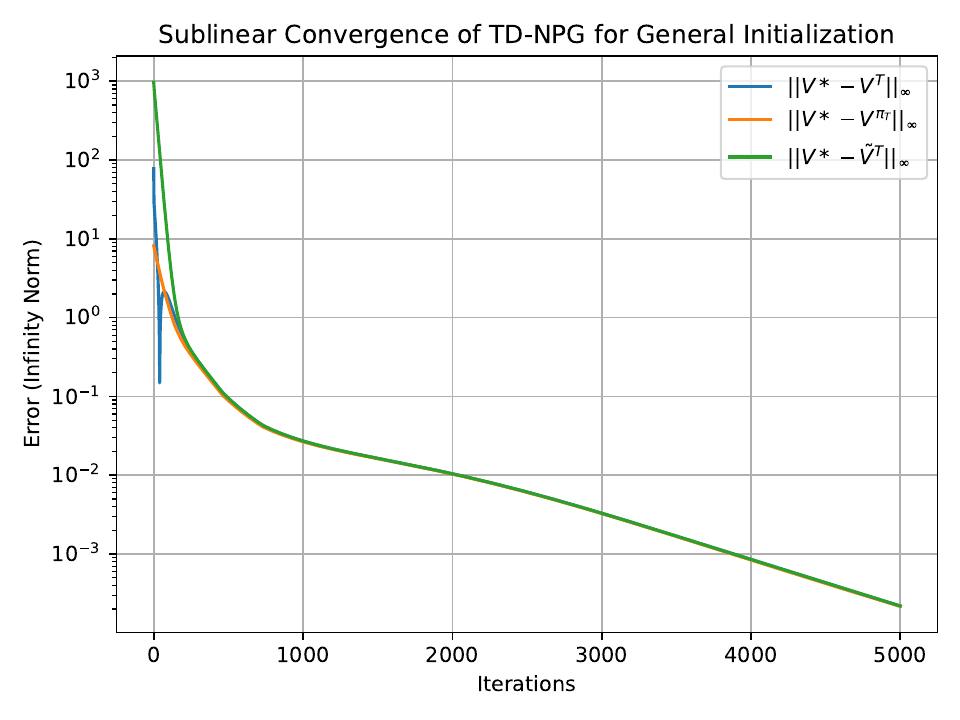}
            \label{fig:sub2-td-npg}
        \end{minipage}
     }
     \centering
    \caption{Sublinear convergence of exact TD-PQA (a) and TD-NPG (b) with constant step size $\eta=0.1$ on a random MDP  for a randomly generated initialization. }
    \label{fig:sub2-td}
\end{figure}

{\paragraph{Empirical comparison between TD-PMD and PMD} We also compare the performance of TD-PMD and PMD based on the two particular $h$, see Figure~\ref{fig:sub2-td-pmd-pmd} for the comparison.  The error
plots can clearly show that they exhibits similar convergence behavior.}

\begin{figure}[ht!]
    \subfigure[TD-PQA vs PQA]{
        \begin{minipage}[t]{0.45\linewidth}
            \centering
            \includegraphics[width=1\linewidth]{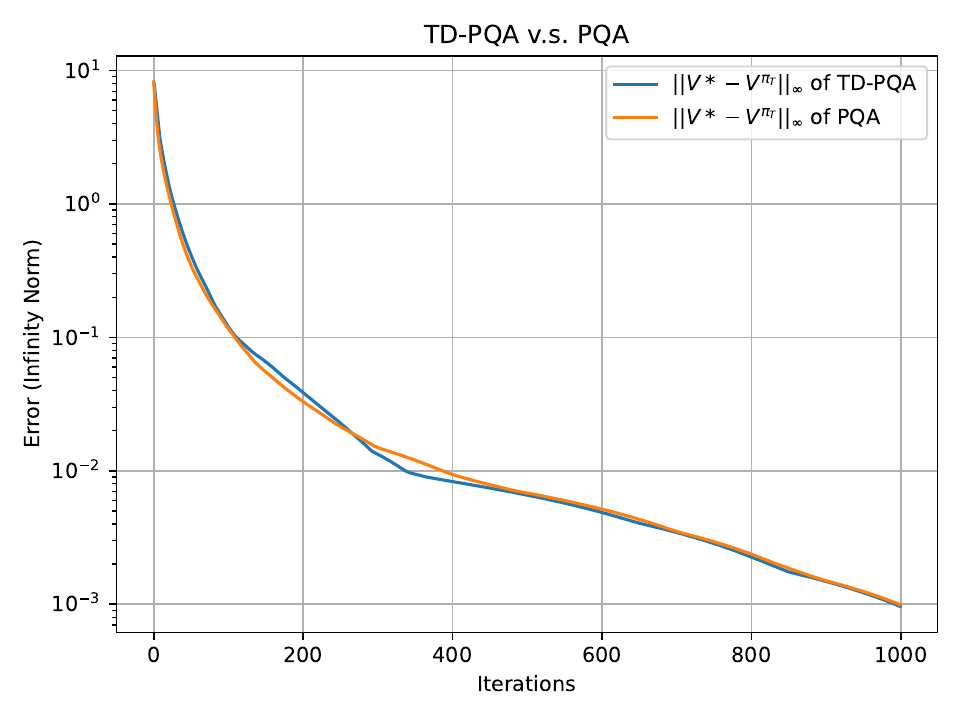}
            \label{fig:sub2-td-pqa-pqa}
        \end{minipage}
    }
    \subfigure[TD-NPG vs NPG]{
        \begin{minipage}[t]{0.45\linewidth}
            \centering
            \includegraphics[width=1\linewidth]{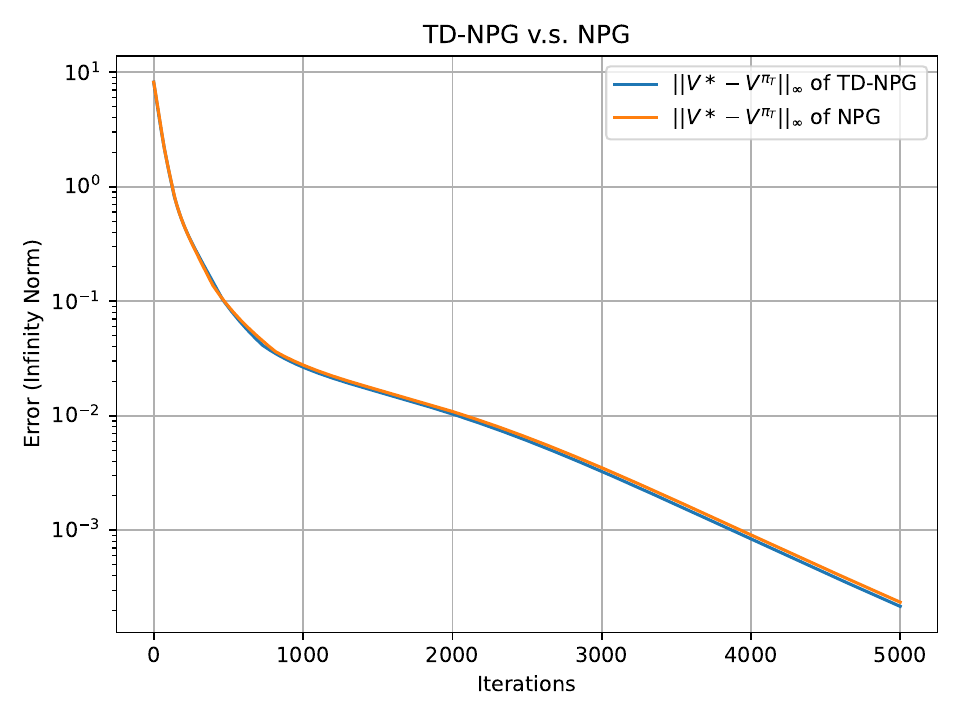}
            \label{fig:sub2-td-npg-npg}
        \end{minipage}
     }
     \centering
    \caption{Empirical comparison between  TD-PQA and PQA (a), between TD-NPG and NPG (b) with constant step size $\eta=0.1$ on a random MDP  for a randomly generated initialization. }
    \label{fig:sub2-td-pmd-pmd}
\end{figure}
%%%%%%%%%%%%%%%%%%%%
\subsection{Linear convergence}
\label{sec:linear-convergence}
Next, we will show the $\gamma$-rate linear convergence of the exact TD-PMD with adaptive step sizes. In contrast to the analysis for sublinear convergence, the analysis here does not rely on the monotonicity property since adaptive step sizes are allowed. We first show that for TD-PMD, the policy error $\| V^* - V^{\pi_T} \|$ can be bounded by the value error $\| V^* - V^{T} \|$.
\begin{lemma}
    For the exact TD-PMD, there holds 
    \begin{align}
        \ls\| V^{*} - V^{\pi_T} \rs\|_\infty \leq \frac{1}{1-\gamma} \ls( \ls\| V^* - V^T \rs\|_\infty + \ls\| V^* - V^{T-1} \rs\|_\infty\rs).
    \end{align}
    \label{lem:error-bounds}
\end{lemma}
The proof of Lemma~\ref{lem:error-bounds} is presented in Appendix~\ref{sec:pf:lem:error-bounds}. In the upcoming analysis, the linear convergence in terms of the value error is established at first. Then together with Lemma~\ref{lem:error-bounds}, the linear convergence in terms of the policy error can be subsequently obtained.
% A direct computation gives
% \begin{align*}
%     \ls\| V^{\pi_T} - V^T \rs\|_\infty \!\!&= \ls\| \calT^{\pi_T}V^{\pi_T} - \calT^{\pi_T} V^{T-1} \rs\|_\infty \\
%     &\leq \gamma \ls\| V^{\pi_T} - V^{T-1} \rs\|_\infty \\
%     &\leq \gamma \ls\| V^{\pi_T} - V^T \rs\|_\infty \!\!+ \gamma \ls\| V^{T} - V^{T-1} \rs\|_\infty,
% \end{align*}
% where we use the fact that $\calT^{\pi_T} V^{\pi_T} = V^{\pi_T}$, $V^T = \calT^{\pi_T}V^{T-1}$, and the contraction property of Bellman operator. By rearranging the terms we further obtain
% \begin{align*}
%     &\ls\| V^{\pi_T} - V^T \rs\|_\infty \\
%     &\leq \frac{\gamma}{1-\gamma} \ls\| V^T - V^{T-1} \rs\|_\infty \\
%     &\leq \frac{\gamma}{1-\gamma} \ls( \ls\| V^* - V^T \rs\|_\infty \!+ \ls\| V^* - V^{T-1} \rs\|_\infty \rs). \numberthis \label{eq:intrinsic-error-bound}
% \end{align*}
% Eq~\eqref{eq:intrinsic-error-bound} shows that for TD-PMD, the intrinsic error can be bounded by the value errors. Combining with Eq~\eqref{eq:error-decomposition}, it shows that the policy error is bounded by the value errors. Therefore, it suffices to establish the linear convergence of the evaluation value $V^T$. 

The following lemma is central to the analysis.
\begin{lemma}
Consider TD-PMD with adaptive step sizes $\{\eta_k\}$.    For $k=0,1,...,T-1$, there holds
    \begin{align}
        [\calT^{\pi_{k+1}} V^k - \calT V^k](s) \geq -\frac{1}{\eta_k}\ls\| D_{\pi_k}^{\tilde{\pi}_k} \rs\|_\infty, \;\; \ls\|D_{\pi_k}^{\tilde{\pi}_k} \rs\|_\infty :=\max_s D_{\pi_k}^{\tilde{\pi}_k}(s),
        \label{eq:TD-PMD-three-point-linear}
    \end{align}
    where $\tilde{\pi}_k$ is any policy that satisfies $\langle\tilde{\pi}_k(\cdot|s),Q^k(s,\cdot)\rangle = \max_aQ^{k}(s,a),\;\forall s$.
    \label{lem:TD-PMD-three-point-linear}
\end{lemma}
The proof of Lemma~\ref{lem:TD-PMD-three-point-linear} is given in Appendix~\ref{sec:pf:lem:TD-PMD-three-point-linear}, which is based on the three-point descent lemma. %This lemma provides a lower bound of the improvement. %The proof of Lemma~\ref{lem:TD-PMD-three-point-linear} is given in Appendix~\ref{sec:pf:lem:TD-PMD-three-point-linear}. 
Noting that $\calT V^k$ is actually a value iteration (VI, see e.g.~\cite{suttonRL}) update, Lemma~\ref{lem:TD-PMD-three-point-linear} implies that the TD evaluation of TD-PMD can be close to VI by setting large enough step size $\eta_k$. As VI is known to converge $\gamma$-linearly~\cite{suttonRL}, the $\gamma$-rate linear convergence of TD-PMD can be established by enlarging $\eta_k$ so that the divergence term is well controlled. 

\begin{theorem}[Linear convergence]
    Consider  TD-PMD  with adaptive step sizes $\{ \eta_k \}$. Assume $\eta_k \geq {\| D_{\pi_k}^{\tilde{\pi}_k} \|_\infty}/{(c\gamma^{2k+1})}$,
    where $c > 0$ is an arbitrary positive constant. Then, we have
    \begin{align}
        \ls\| V^* - V^T \rs\|_\infty &\leq \gamma^T \ls[ \ls\| V^* - V^0 \rs\|_\infty + \frac{c}{1-\gamma} \rs]
        \label{eq:linear-convergence-of-output-value} \\
        \ls\| V^* - V^{\pi_T} \rs\|_\infty &\leq \frac{2\gamma^{T-1}}{1-\gamma} \ls[ \ls\| V^* - V^0 \rs\|_\infty + \frac{c}{1-\gamma} \rs].
        \label{eq:linear-convergence-of-output-policy}
    \end{align}
    \label{thm:linear-convergence-of-exact-TD-PMD}
\end{theorem}

The proof of Theorem~\ref{thm:linear-convergence-of-exact-TD-PMD} is given in Appendix~\ref{sec:pf:thm:linear-convergence-of-exact-TD-PMD}.

\begin{remark}
%    The linear convergence of the evaluation value $V^T$ (i.e. Eq~\eqref{eq:linear-convergence-of-output-value}) exactly matches the result for vanilla PMD in~\cite{Johnson_Pike-Burke_Rebeschini_2023} with similar adaptive step sizes scheduling. That means a one-step TD evaluation is also sufficient for PMD to achieve the $\gamma$-rate linear convergence.
The step size selection rule in Theorem~\ref{thm:linear-convergence-of-exact-TD-PMD} is inspired by  \textup{\cite{Johnson_Pike-Burke_Rebeschini_2023}}. It is worth noting that the rate in \eqref{eq:linear-convergence-of-output-value}  exactly matches the result for  PMD in~\textup{\cite{Johnson_Pike-Burke_Rebeschini_2023}} where the exact evaluation of $Q^{\pi_k}$ is required. It suggests that one-step TD evaluation is  sufficient for the algorithm to achieve the $\gamma$-rate linear convergence.
\end{remark}

\subsection{Policy convergence for common instances of TD-PMD}
\label{sec:TD-PQA-and-TD-NPG}
By specifying different $h$, PMD covers a wide range of policy gradient methods. Among them, projected Q-ascent (PQA,~\cite{Xiao_2022,ppgliu}) and softmax natural policy gradient (softmax NPG,~\cite{kakade2002npg,Agarwal_Kakade_Lee_Mahajan_2019,Khodadadian_Jhunjhunwala_Varma_Maguluri_2021,pg-liu,Xiao_2022})   are two mostly studied instances. Thus, it is natural to study these two instances under the one-step TD evaluation setting. By setting $h(p)=(1/2) \| p \|_2^2$ in TD-PMD, one obtains \textbf{TD-PQA}, which has an explicit policy update of the form 
\begin{align}
    \mbox{(TD-PQA)} \quad \forall\, k\in\mathbb{N}, \, s\in\calS: \quad \pi_{k+1}(\cdot|s) = \Proj_{\Delta(\calA)} \ls( \pi_k(\cdot|s) + \eta\, Q^k(s,\cdot) \rs),
    \label{eq:TD-PQA-update}
\end{align}
where $\Proj_{\Delta(\calA)}(\cdot)$ denotes the  projection onto the probability simplex $\Delta(\calA)$ under the Euclidean distance. Similarly, setting $h(p)=\sum_{a\in\calA}p_a \log p_a$ yields  \textbf{TD-NPG},
\begin{align}
    \mbox{(TD-NPG)} \quad \forall\, k\in\mathbb{N}, \, s\in\calS: \quad \pi_{k+1}(a|s) = \frac{\pi_k(a|s) \cdot \exp \ls\{ \eta\, Q^k(s,a) \rs\}}{\sum_{a^\prime\in\calA} \pi_k(a^\prime|s) \cdot \exp \ls\{ \eta\, Q^k(s,a^\prime) \rs\}}.
    \label{eq:TD-NPG-update}
\end{align}
% It is worth noting that the induced Bregman divergences in TD-PQA and TD-NPG are $l_2$ distance and Kullback-Leibler (KL) divergence, respectively. 
In this section, we show that similar to PQA and NPG \cite{ppgliu,li2025phi}, TD-PQA and TD-NPG also enjoy the convergence in the policy domain (TD-PQA is indeed able to find an optimal policy in a finite number of iterations). The proofs of these results are technically quite involved and  require elementary computations over the policy space and thus highly depend on the explicit policy update formulas. {For conciseness, we only consider constant step size and policy convergence can also be established for increasing step sizes.}

\subsubsection{Finite iteration convergence of TD-PQA}
%Under any constant step size, PQA outputs the optimal policy in finite iterations~\cite{ppgliu}.
%By using the property of simplex projection, it is shown in~\cite{ppgliu} that a step of PQA update will output the optimal policy when the value error $\| V^* - V^{\pi_k} \|_\infty$ is smaller than a threshold.
%For TD-PQA, we establish a similar finite iteration convergence, which shows that the policies after a finite time are all optimal.
\begin{theorem}[Finite iteration convergence of TD-PQA]
    With any constant step size $\eta_k = \eta > 0$ and initialization $\{ V^0, \, \pi_0 \}$, there exists a finite time 
    \begin{align*}
        T_0 \!=\! \begin{cases}
            \!\ls\lceil \max  \ls\{ \dfrac{2\gamma}{\varepsilon} \!\ls( \dfrac{1}{(1-\gamma)^2}\!+\!\dfrac{\ls\| V_0 \rs\|_\infty \!+\! \kappa_0}{1-\gamma} \!+\! \dfrac{\ls\| D_0^* \rs\|_\infty}{\eta(1-\gamma)} \rs)\!, \; \dfrac{\log\varepsilon \!-\! \log 2\gamma \!-\! \log \kappa_0}{\log \gamma} \rs\} \!\rs\rceil, & \!\mbox{if } \kappa_0 > 0, \\[.7em]
            \!\ls\lceil \dfrac{2\gamma}{\varepsilon} \ls( \dfrac{1}{(1-\gamma)^2}\!+\!\dfrac{\ls\| V_0 \rs\|_\infty \!+\! \kappa_0}{1-\gamma} + \dfrac{\ls\| D_0^* \rs\|_\infty}{\eta(1-\gamma)} \rs) \rs\rceil, & \!\mbox{if } \kappa_0 = 0,
        \end{cases}  
    \end{align*}
    where $\varepsilon=\eta\gamma\Delta^2/(2\eta\gamma\Delta+2)$, such that for all $k\ge T_0$, the policies $\pi_k$ are optimal, implying that $Q^{\pi_k}=Q^*$ and $V^{\pi_k}=V^*$.
    \label{thm:finite-iteration-convergence-of-TD-PQA}
\end{theorem}
The proof of Theorem~\ref{thm:finite-iteration-convergence-of-TD-PQA} is given in Appendix~\ref{sec:pf:thm:finite-iteration-convergence-of-TD-PQA} which leverages the property of the projection onto the simplex. %Note that $Q^k$ is not guaranteed to converge in finite iterations.

\subsubsection{Policy convergence of TD-NPG}

\begin{theorem}[Policy convergence of TD-NPG]
    With any constant step size $\eta_k=\eta>0$, the policy generated by TD-NPG converges to some optimal policy, i.e., $\pi_k \to \pi^*$ as $k \to \infty$.
    \label{thm:policy-convergence-of-TD-NPG}
\end{theorem}

% The convergence of softmax NPG in the policy domain is firstly established in~\cite{li2025phi} by showing that $\{ \pi^k(a|s) \}$ is a Cauchy sequence.
The proof of Theorem~\ref{thm:policy-convergence-of-TD-NPG} is given in Section~\ref{TD-NPG-policy-convergence}, which consists of three steps.
We  first extend the analysis in~\cite{Khodadadian_Jhunjhunwala_Varma_Maguluri_2021,li2025phi} of characterizing the probability of suboptimal actions to obtain the local linear convergence of $V^{\pi_k}$ and $Q^{\pi_k}$. Then, we bound $\| Q^* - Q^k \|$ in terms of $\| Q^* - Q^{\pi_k} \|$ to establish the local linear convergence of $Q^k$. Lastly, by combining the policy update formula (equation~\eqref{eq:TD-NPG-update}) and the local linear convergence of $Q^k$, we show that the policy generated by TD-NPG converges to some optimal policy.

%%%%%%%%%%%%%%%%%%%%
%\subsection{Numerical Experiments}

%%%%%%%%%%%%%%%%%%%%
\section{Sample complexity under a generative model}
\label{sec:sample-complexity}
% In this section we study the convergence of sampling-based variant of TD-PMD (Algorithm \ref{alg:TD-PMD}). That is, at each iteration $k$, we need to estimate both  the action value function $Q^k$ and $\mathcal{T}^{\pi_{k+1}}V^k$ by sampling. We assume having access to a generative model $\mathcal{G}$, which allows us to sample $s^\prime \sim P(\cdot|s,a)$ from  $\mathcal{G}$ at each state action pair $(s,a)\in \mathcal{S}\times\mathcal{A}$. The detailed procedure of sampling-based TD-PMD is presented in Algorithm \ref{alg:sampling-TD-PMD}. Estimation by sampling naturally introduces noise and larger noise yields worse performance. Thus we first study the convergence of Algorithm \ref{alg:sampling-TD-PMD} with error level $\delta$, i.e.,
%\begin{align}
%\underset{k\le T-1}{\mathrm{sup}}\,\,\left\| \widehat{Q}^k-Q^k \right\| _{\infty}\le \delta
%\label{Q esitmation error}
%\end{align}
%and 
%\begin{align}
%\underset{k\le T-1}{\mathrm{sup}}\,\,\left\| V^{k+1}-\mathcal{T} ^{\pi _{k+1}}V^k \right\| _{\infty}\le \delta.
%\label{V esitmation error}
%\end{align}
%Then we figure out how much samples we need to satisfy Eq~\eqref{Q esitmation error} and Eq~\eqref{V esitmation error} simultaneously with high probability. We first present the inexact version of Lemma \ref{lem:error-bounds}.
In this section, {we study the sample complexity for the sample-based TD-PMD to find the last iterate $\varepsilon$-optimal solution based on the generative model that has been used in the sample analysis of PMD \cite{Xiao_2022,Johnson_Pike-Burke_Rebeschini_2023,protopapas2024policy}}. To this end, it is helpful to first study the convergence of the inexact TD-PMD where there are estimation errors when computing $Q^k$ and $\mathcal{T}^{\pi_{k+1}}V^k$.
%%%%%%%%%%%%%%%%%
\subsection{Convergence of inexact TD-PMD}
Consider the inexact TD-PMD, where $Q^k$ in \eqref{eq:alg1pi} is replaced by $\widehat{Q}^k$ that satisfies 
\begin{align*}\| \widehat{Q}^k-Q^k \| _{\infty}\le \delta \numberthis\label{Q esitmation error}
\end{align*}
where $\delta>0$ is the error level, and equation \eqref{eq:alg1v} is replaced by\footnote{$\widehat{\mathcal{T}}^{\pi_k}$ denotes an approximate evaluation of ${\mathcal{T}}^{\pi_k}$.}
\begin{align}
    V^{k+1}=\widehat{\mathcal{T}}^{\pi_{k+1}} V^k \;\; \mbox{with} \; \left\| V^{k+1}-\mathcal{T} ^{\pi _{k+1}}V^k \right\| _{\infty}\le \delta.\label{V esitmation error}
\end{align}
Under the same adaptive step size selection rule as in Theorem~\ref{thm:linear-convergence-of-exact-TD-PMD}, the linear convergence result similar to Theorem~\ref{thm:linear-convergence-of-exact-TD-PMD} can be obtained by additionally separating the error $\delta$, see Theorem~\ref{thm:linear-convergence-of-exact-TD-PMD-inexact} below. The proof of Theorem~\ref{thm:linear-convergence-of-exact-TD-PMD-inexact} is presented in Appendix~\ref{sec:pf:thm:linear-convergence-of-exact-TD-PMD-inexact}.

% \begin{lemma} 
%     Consider the inexact TD-PMD. There holds
%     \small{
%         \begin{align*}
%         &\ls\| V^*-V^{\pi_T}\rs\| _{\infty} \le \frac{1}{1-\gamma}\left( \ls\| V^*-V^T \rs\| _{\infty}+\ls\| V^*-V^{T-1}\rs\| _{\infty} \right) +\frac{\delta}{1-\gamma}.
%     \end{align*}
%     }
%     \label{lem:inexact-policy-error-bounded-by-value-error}
% \end{lemma}
% The proof of Lemma~\ref{lem:inexact-policy-error-bounded-by-value-error} is given in Appendix~\ref{sec:pf:lem:inexact-policy-error-bounded-by-value-error}. 
% \begin{lemma}
%     Consider the inexact TD-PMD  with adaptive step sizes $\{ \eta_k \}$. For $k=0, 1, ..., T-1$,  there holds
%     \begin{align}
%        [\widehat{\calT}^{\pi_{k+1}} V^k - \calT V^k](s) \geq -\frac{1}{\eta
%          _k}\ls\| {D}_{\pi_k}^{\widehat{\pi}_k} \rs\|_\infty - 3\delta, \;\;
%         \|{D}_{\pi_k}^{\widehat{\pi}_k}\|_\infty = \max_{s} D_{\pi_k}^{\widehat{\pi}_k}(s),
%         \label{eq:TD-PMD-three-point-linear-inexact}
%     \end{align}
%     and $\widehat{\pi}_k$ is any policy that satisfies $\langle\widehat{\pi}_k(\cdot|s),\widehat{Q}^k(s,\cdot)\rangle = \max_a\widehat{Q}^{k}(s,a),\quad\forall s$.
%     \label{lem:TD-PMD-three-point-linear-inexact}
% \end{lemma}
% The proof of Lemma~\ref{lem:TD-PMD-three-point-linear-inexact} is given in Appendix~\ref{sec:pf:lem:TD-PMD-three-point-linear-inexact}.

\begin{theorem}[Linear convergence with error level $\delta$]
    Consider the inexact TD-PMD  with adaptive step sizes $\{\eta_k\}$. Assume $\eta_k \geq {\| {D}_{\pi_k}^{\widehat{\pi}_k} \|_\infty}/{(c \gamma^{2k+1})}$, where $c > 0$ is an arbitrary positive constant. Then, we have
    \begin{align*}
        \ls\| V^* - V^T \rs\|_\infty &\leq \gamma^T \ls[ \ls\| V^* - V^0 \rs\|_\infty + \frac{c}{1-\gamma} \rs] + \frac{3\delta}{1-\gamma} \\
        % \label{eq:linear-convergence-of-output-value-inexact}
        \ls\| V^* - V^{\pi_T} \rs\|_\infty &\leq \frac{2\gamma^{T-1}}{1-\gamma} \ls[ \ls\| V^* - V^0 \rs\|_\infty + \frac{c}{1-\gamma} \rs] + \frac{7\delta}{(1-\gamma)^2}.
        % \label{eq:linear-convergence-of-output-policy-inexact}
    \end{align*}
    \label{thm:linear-convergence-of-exact-TD-PMD-inexact}
\end{theorem}

\subsection{Sample complexity}
%We now study the sample complexity to achieve $\varepsilon$-optimal policy, i.e., $
%\left\| V^*-V^{\pi _T} \right\| _{\infty}\le \varepsilon$, with high probability.  We first show that the value functions in sampling-based TD-PMD is bounded by $\frac{1}{1-\gamma}$ if $V_0$ is bounded by $\frac{1}{1-\gamma}$.
%Assume there is a generative model which allows us to a
In the sample-based TD-PMD, we assume that there is a generative model which allows us to estimate $Q^k$ and $\mathcal{T}^{\pi_{k+1}}V^k$ via a number of independent samples at every $(s,a)$. More concretely, at each iteration $k$, for each $(s,a) \in \calS\times\calA$ we sample i.i.d. $\{ s^\prime_{(i)} \}_{i=1}^{M_Q} \sim P(\cdot|s,a)$ to compute $\hat{Q}^k$,
\begin{align*}
    \forall\, (s,a) \in \calS\times\calA: \quad \hat{Q}^k(s,a) = \frac{1}{M_Q} \sum\nolimits_{i=1}\nolimits^{M_Q} \ls[ r(s,a) + \gamma V^k(s^\prime_{(i)}) \rs],
\end{align*}
and for each $s\in\calS$ we sample i.i.d. $\{ (a_{(i)}, \, s^\prime_{(i)}) \}_{i=1}^{M_V} \sim \pi^{k+1}(*|s) \times P(\cdot|\,s,*)$ to compute
\begin{align*}
   \forall\, s\in\calS: \quad V^{k+1}(s) = \widehat{\calT}^{\pi_{k+1}} V^k(s) = \frac{1}{M_V} \sum\nolimits_{i=1}^{M_V} \ls[ r(s,a_{(i)}) + \gamma V^k(s^\prime_{(i)}) \rs].
\end{align*}
A complete description of such sample-based TD-PMD is presented in Algorithm~\ref{alg:sampling-TD-PMD}. This section investigates the sample complexity to achieve the $\varepsilon$-optimal policy, i.e., $
\left\| V^*-V^{\pi _T} \right\| _{\infty}\le \varepsilon$, with high probability. Using the Hoeffding’s inequality (Lemma~\ref{lem:hoeffding}), we can obtain the number of samples needed to satisfy \eqref{Q esitmation error} and \eqref{V esitmation error} with high probability.

\begin{algorithm}[ht!]
   \caption{Sample-based TD-PMD}
   \label{alg:sampling-TD-PMD}
\begin{algorithmic}
   \STATE {\bfseries Input:} Initial evaluation state value $V^0$, initial policy $\pi_0$, iteration number $T$, step size $\{ \eta_k \}$, the number of samples $M_Q$ and $M_V$ for estimating $Q^k$ and $\mathcal{T}^{\pi_{k+1}}V^k$.

   \FOR{$k=0$ {\bfseries to} $T-1$}
   \vspace{0.2cm}
   \STATE {\bfseries (Q estimation)} For each state action $(s,a)\in \mathcal{S}\times\mathcal{A}$, sample $\{ s^\prime_{(i)} \} \stackrel{\mathrm{i.i.d.}}{\sim} P(\cdot|\,s,a)$ and compute
   \begin{align*}
 \widehat{Q}^k(s,a) = \frac{1}{M_Q} \sum_{i=1}^{M_Q} \ls[ r(s,a) + \gamma V^k(s^\prime_{(i)}) \rs].
   \end{align*}

   \vspace{0.2cm}

   \STATE {\bfseries (Policy improvement)} 
   Update the policy by
   \begin{align}
       \pi_{k+1}(\cdot|s) \!=\! \underset{p\in\Delta(\calA)}{\arg\max} \ls\{ \eta_k \langle p, \, \widehat{Q}^{k}(s,\cdot) \rangle \!-\! D^p_{\pi_k}(s) \rs\}.
       \label{eq:TD-PMD-policy-step-inexact}
   \end{align}
   \STATE {\bfseries (TD evaluation)} For each state $s\in \mathcal{S}$, sample $\{ (a_{(i)}, \, s^\prime_{(i)}) \} \stackrel{\mathrm{i.i.d.}}{\sim} \pi^{k+1}(*|s) \times P(\cdot|\,s,*)$ and compute
    \begin{align*}
       V^{k+1}= \widehat{\calT}^{\pi_{k+1}} V^k(s) = \frac{1}{M_V} \sum_{i=1}^{M_V} \ls[ r(s,a_{(i)}) + \gamma V^k(s^\prime_{(i)}) \rs].
    \end{align*}
   % Then update value function
   % \begin{align*}
   %     V^{k+1}=\widehat{\calT}^{\pi_{k+1}} V^k(s).
   % \end{align*}
   \ENDFOR
   \STATE {\bfseries Output:} Last iterate policy $\pi_T$, last iterate state  value  estimation $V^T$.
\end{algorithmic}
\end{algorithm}
% \begin{lemma} Consider the sample-based TD-PMD. If $\ls\| V_0 \rs\|_\infty \leq {1}/{(1-\gamma)}$, then there also holds $\left\| V^k \right\| _{\infty}\le {1}/{(1-\gamma)}$ for any $k\geq 1$.
% \label{lem:bounded-values-sample}
% \end{lemma}

\begin{lemma} Consider the sample-based TD-PMD with  $\ls\| V_0 \rs\|_\infty \leq {1}/{(1-\gamma)}$. For any $\alpha \in (0,1)$, if 
\begin{align*}M_Q \geq \frac{1}{2(1-\gamma)^2 \delta^2} \log \frac{4T|\calS||\calA|}{\alpha}, \quad M_V \geq \frac{1}{2(1-\gamma)^2 \delta^2} \log \frac{4T|\calS|}{\alpha},\end{align*} then \eqref{Q esitmation error} and \eqref{V esitmation error} are satisfied for $k=0,\cdots,T-1$ with  probability at least $1-\alpha$.
\label{lem:sample-complexity-for-error-level-delta}
\end{lemma}

The proof of Lemma~\ref{lem:sample-complexity-for-error-level-delta} is given in Appendix~\ref{sec:pf:lem:sample-complexity-for-error-level-delta}. Together with Theorem \ref{thm:linear-convergence-of-exact-TD-PMD-inexact}, one can finally obtain the sample complexity for sample-based TD-PMD.

\begin{theorem} Consider the sample-based TD-PMD with $\ls\| V_0 \rs\|_\infty \leq {1}/{(1-\gamma)}$ and  $  \eta_k \geq {\ls\| {D}_{\pi_k}^{\widehat{\pi}_k} \rs\|_\infty}/{(c \cdot \gamma^{2k+1})},$
    where $c > 0$ is an arbitrary positive constant.
    For any $\alpha \in (0,1)$, if $M_Q$ and $M_V$ satisfy the conditions in Lemma~\ref{lem:sample-complexity-for-error-level-delta}, then 
    \begin{align*}
\ls\| V^*-V^{\pi _T}\rs\| _{\infty}\le \frac{1}{\left( 1-\gamma \right) ^2}\left[ 2\left( 2+c \right) \gamma ^{T-1}+7\delta \right]
    \end{align*}
   holds with probability at least $1-\alpha$.
    \label{thm:sample-complexity}
\end{theorem}
The proof of Theorem~\ref{thm:sample-complexity} is presented in Appendix~\ref{sec:pf:thm:sample-complexity}. Setting $T = \tilde{O}((1-\gamma)^{-1}
)$ and $M_Q=M_V=\tilde{O}(\varepsilon^{-2}(1-\gamma)^{-6})$, it follows immediately from Theorem~\ref{thm:sample-complexity} that the sample-based TD-PMD can achieve the last iterate $\varepsilon$-optimal solution with
\[
    T\times \ls(|\calS| M_V + |\calS||\calA| M_Q \rs) = \tilde{O} \ls( \varepsilon^{-2} |\calS||\calA| (1-\gamma)^{-7} \rs).
\]
number of samples. 
\begin{remark}
The sample complexity established for the sample-based PMD in \textup{\cite{Xiao_2022, Johnson_Pike-Burke_Rebeschini_2023}} is  $\tilde{O}(\varepsilon^{-2}|\calS||\calA|(1-\gamma)^{-8})$. This is basically because one needs to sample a trajectory of length $H=O((1-\gamma)^{-1})$ in the sample-based PMD, while only one-step lookahead is needed in the sample-based TD-PMD.  {Note that the sample complexity is in terms of the state-action pairs. For sample-based PMD and TD-PMD, since the value evaluation is given by the average of the immediate rewards associated with each
state-action pair, the total computational complexity for the value evaluation is overall similar to the
total sample complexity. Therefore, the computational complexity of TD-PMD is also smaller than that of PMD to achieve the $\varepsilon$-optimal solution.}

\end{remark}

\begin{remark}
A better value by applying the $h$-step Bellman optimal operator over the action value is utilized in \textup{\cite{protopapas2024policy}}, which also establishes  $\tilde{O}(\varepsilon^{-2}|\calS||\calA|(1-\gamma)^{-7})$ sample complexity under the generative model when $h$ is sufficiently large. In contrast, we show that even a simple one-step TD evaluation suffices to achieve the same sample complexity.   
More precisely, the value evaluation in $h$-PMD is given by
\begin{align*}
\hat V^{\pi_k}(s) &= \frac{1}{M_V} \sum_{j=1}^{M_V} \sum_{t=0}^{H-1} \gamma ^t r(s_t^j, a_t^j), \quad \mbox{where } s_0^j = s, \;\; a_t^j \sim \pi_k(\cdot|s_t^j), \;\; s_{t+1}^j \sim P(\cdot|s_t^j, a_t^j),\\
  \hat Q^{\pi_k}_1(s,a) &= r(s,a) + \frac{\gamma}{M_Q} \sum_{i=1}^{M_Q}  \hat V^{\pi_k}(s_i^\prime), \quad \mbox{where } s_i^\prime \sim P(\cdot | s,a),\\
    \hat Q^{\pi_k}_h(s,a) &= r(s,a) + \frac{\gamma}{M_Q} \sum_{i=1}^{M_Q} \max_{a^\prime \in \mathcal{A}} \; \hat Q^{\pi_k}_{h-1}(s_i^\prime,a^\prime), \quad \mbox{where } s_i^\prime \sim P(\cdot | s,a).
\end{align*}
Letting $h$, $K$,  $H$,  $M_Q$ and $M_V$ be 
\begin{align*}
h &= O((1-\gamma)^{-1}),\quad
    K = \tilde O (h^{-1}(1-\gamma)^{-1})=\tilde O(1),\quad
     H = \tilde O((1-\gamma)^{-1}),\\
     M_Q &= \tilde O ((1-\gamma)^{-6} \varepsilon^{-2}), \quad
     M_V = \tilde O ((1-\gamma)^{-4} \varepsilon^{-2}  \cdot \gamma^{2h} / (1-\gamma^h)^2)= \tilde O ((1-\gamma)^{-4} \varepsilon^{-2})
\end{align*}
then the sample complexity is $(K \times M_Q \times h \times |\mathcal{S}| \times |\mathcal{A}|) + (K \times H \times M_V \times |\mathcal{S}|) = \tilde O ((1-\gamma)^{-7} \varepsilon^{-2} |\mathcal{S}| |\mathcal{A}|)$.
Noting that the max operation overall requires $O(|\mathcal{A}|)$ flops (as it requires to scan $\hat Q^{\pi_k}_{h-1}$ over all $|\mathcal{A}|$), the overall computational complexity of $h$-PMD is $(K \times M_Q \times (h-1) \times |\mathcal{S}| \times |\mathcal{A}|^2) + (K \times M_Q \times |\mathcal{S}| \times |\mathcal{A}|) + (K \times H \times M_V \times |\mathcal{S}|) = \tilde O ((1-\gamma)^{-7} \varepsilon^{-2} |\mathcal{S}| |\mathcal{A}|^2)$. Therefore, even though $h$-PMD and TD-PMD have the same sample complexity, the computational complexity of TD-PMD is smaller than that of $h$-PMD.
\end{remark}

\section{Conclusion and future work}\label{sec:conclusion}
This paper studies policy mirror descent with temporal difference evaluation, and the dimension free sublinear convergence and $\gamma$-rate linear convergence have been established in the exact setting for constant and adaptive step sizes, respectively. {The convergence in the policy domain is also established for the two common instances}. Novel and elementary analysis techniques have been developed to achieve these results. The sample complexity of TD-PMD is also provided under a generative model, which is better than that for PMD. 

For future direction, {it is likely to extend the analysis to the scenario with entropy regularization}. Additionally, it is also interesting to investigate the convergence of TD-PMD under other sampling schemes, for example the Markovian sampling.

%%%%%%%%%%%%%%%%%%
% To put the Bibtex.
\bibliographystyle{plain}
\bibliography{references}

\begin{thebibliography}{10}

\bibitem{Agarwal_Kakade_Lee_Mahajan_2019}
Alekh Agarwal, Sham~M. Kakade, Jason~D. Lee, and Gaurav Mahajan.
\newblock On the theory of policy gradient methods: {O}ptimality,
  approximation, and distribution shift.
\newblock {\em Journal of Machine Learning Research}, 22(98):1--76, 2021.

\bibitem{yuan2023general}
Carlo Alfano, Rui Yuan, and Patrick Rebeschini.
\newblock A novel framework for policy mirror descent with general
  parameterization and linear convergence.
\newblock In {\em Advances in Neural Information Processing Systems}, 2023.

\bibitem{Bhandari_Russo_2021}
Jalaj Bhandari and Daniel Russo.
\newblock On the linear convergence of policy gradient methods for finite
  {MDP}s.
\newblock In {\em International Conference on Artificial Intelligence and
  Statistics}, volume 130, pages 2386--2394, 2021.

\bibitem{chelu2024functional}
Veronica Chelu and Doina Precup.
\newblock Functional acceleration for policy mirror descent.
\newblock {\em arXiv preprint arXiv:2407.16602}, 2024.

\bibitem{chen1993convergence}
Gong Chen and Marc Teboulle.
\newblock Convergence analysis of a proximal-like minimization algorithm using
  bregman functions.
\newblock {\em SIAM Journal on Optimization}, 3(3):538--543, 1993.

\bibitem{cheng2020policy}
Ching-An Cheng, Andrey Kolobov, and Alekh Agarwal.
\newblock Policy improvement via imitation of multiple oracles.
\newblock {\em Advances in Neural Information Processing Systems},
  33:5587--5598, 2020.

\bibitem{fatkhullin2023stochastic}
Ilyas Fatkhullin, Anas Barakat, Anastasia Kireeva, and Niao He.
\newblock Stochastic policy gradient methods: Improved sample complexity for
  fisher-non-degenerate policies.
\newblock In {\em International Conference on Machine Learning}, pages
  9827--9869. PMLR, 2023.

\bibitem{ref-AlphaTensor}
Alhussein Fawzi, Matej Balog, Aja Huang, Thomas Hubert, Bernardino
  Romera-Paredes, Mohammadamin Barekatain, Alexander Novikov, Francisco~J.
  R.~Ruiz, Julian Schrittwieser, Grzegorz Swirszcz, David Silver, Demis
  Hassabis, and Pushmeet Kohli.
\newblock Discovering faster matrix multiplication algorithms with
  reinforcement learning.
\newblock {\em Nature}, 610:47--53, 2022.

\bibitem{feng2024global}
Jie Feng, Ke~Wei, and Jinchi Chen.
\newblock Global convergence of natural policy gradient with hessian-aided
  momentum variance reduction.
\newblock {\em Journal of Scientific Computing}, 101(2):44, 2024.

\bibitem{Geist_Scherrer_Pietquin_2019}
Matthieu Geist, Bruno Scherrer, and Olivier Pietquin.
\newblock A theory of regularized markov decision processes.
\newblock In {\em Proceedings of the 36th International Conference on Machine
  Learning}, pages 2160--2169, 2019.

\bibitem{hong2023two}
Mingyi Hong, Hoi-To Wai, Zhaoran Wang, and Zhuoran Yang.
\newblock A two-timescale stochastic algorithm framework for bilevel
  optimization: Complexity analysis and application to actor-critic.
\newblock {\em SIAM Journal on Optimization}, 33(1):147--180, 2023.

\bibitem{Johnson_Pike-Burke_Rebeschini_2023}
Emmeran Johnson, Ciara Pike-Burke, and Patrick Rebeschini.
\newblock Optimal convergence rate for exact policy mirror descent in
  discounted markov decision processes.
\newblock In {\em Advances in Neural Information Processing Systems}, 2023.

\bibitem{ref-AlphaFold}
John Jumper, Richard Evans, Alexander Pritzel, Tim Green, Michael Figurnov,
  Olaf Ronneberger, Kathryn Tunyasuvunakool, Russ Bates, Augustin Žídek, Anna
  Potapenko, Alex Bridgland, Clemens Meyer, Simon A.~A. Kohl, Andrew~J.
  Ballard, Andrew Cowie, Bernardino Romera-Paredes, Stanislav Nikolov, Rishub
  Jain, Jonas Adler, Trevor Back, Stig Petersen, David Reiman, Ellen Clancy,
  Michal Zielinski, Martin Steinegger, Michalina Pacholska, Tamas Berghammer,
  Sebastian Bodenstein, David Silver, Oriol Vinyals, Andrew~W. Senior, Koray
  Kavukcuoglu, Pushmeet Kohli, and Demis Hassabis.
\newblock Highly accurate protein structure prediction with alphafold.
\newblock {\em Nature}, 596:583--589, 2021.

\bibitem{kakade2002npg}
Sham Kakade.
\newblock A natural policy gradient.
\newblock In {\em Advances in Neural Information Processing Systems}, pages
  1531--1538, 2001.

\bibitem{kakade2002approximately}
Sham~M. Kakade and John Langford.
\newblock Approximately optimal approximate reinforcement learning.
\newblock In {\em International Conference on Machine Learning}, pages
  267--274, 2002.

\bibitem{khodadadian2021finite}
Sajad Khodadadian, Zaiwei Chen, and Siva~Theja Maguluri.
\newblock Finite-sample analysis of off-policy natural actor-critic algorithm.
\newblock In {\em International Conference on Machine Learning}, pages
  5420--5431. PMLR, 2021.

\bibitem{khodadadian2022finite}
Sajad Khodadadian, Thinh~T Doan, Justin Romberg, and Siva~Theja Maguluri.
\newblock Finite-sample analysis of two-time-scale natural actor--critic
  algorithm.
\newblock {\em IEEE Transactions on Automatic Control}, 68(6):3273--3284, 2022.

\bibitem{Khodadadian_Jhunjhunwala_Varma_Maguluri_2021}
Sajad Khodadadian, Prakirt~Raj Jhunjhunwala, Sushil~Mahavir Varma, and
  Siva~Theja Maguluri.
\newblock On the linear convergence of natural policy gradient algorithm.
\newblock In {\em IEEE Conference on Decision and Control}, pages 3794--3799,
  2021.

\bibitem{klein2024structure}
Sara Klein, Xiangyuan Zhang, Tamer Ba{\c{s}}ar, Simon Weissmann, and Leif
  D{\"o}ring.
\newblock Structure matters: Dynamic policy gradient.
\newblock {\em arXiv preprint arXiv:2411.04913}, 2024.

\bibitem{Lan_2021}
Guanghui Lan.
\newblock Policy mirror descent for reinforcement learning: {L}inear
  convergence, new sampling complexity, and generalized problem classes.
\newblock {\em Mathematical Programming}, 198(1):1059--1106, 2021.

\bibitem{robot2}
Joonho Lee, Jemin Hwangbo, Lorenz Wellhausen, Vladlen Koltun, and Marco Hutter.
\newblock Learning quadrupedal locomotion over challenging terrain.
\newblock {\em Science Robotics}, 5(47), 2020.

\bibitem{li2023exponential}
Gen Li, Yuting Wei, Yuejie Chi, and Yuxin Chen.
\newblock Softmax policy gradient methods can take exponential time to
  converge.
\newblock {\em Mathematical Programming}, 201:707--802, 2023.

\bibitem{li2024stochastic}
Tianjiao Li, Feiyang Wu, and Guanghui Lan.
\newblock Stochastic first-order methods for average-reward markov decision
  processes.
\newblock {\em Mathematics of Operations Research}, 2024.

\bibitem{li2025phi}
Wenye Li, Jiacai Liu, and Ke~Wei.
\newblock $\phi$-update: A class of policy update methods with policy
  convergence guarantee.
\newblock In {\em International Conference on Learning Representations}, 2025.

\bibitem{Li_Zhao_Lan_2022}
Yan Li, Guanghui Lan, and Tuo Zhao.
\newblock Homotopic policy mirror descent: {P}olicy convergence, algorithmic
  regularization, and improved sample complexity.
\newblock {\em Mathematical Programming}, 2023.

\bibitem{ppgliu}
Jiacai Liu, Wenye Li, Dachao Lin, Ke~Wei, and Zhihua Zhang.
\newblock On the convergence of projected policy gradient for any constant step
  sizes.
\newblock {\em arXiv:2311.01104}, 2024.

\bibitem{pg-liu}
Jiacai Liu, Wenye Li, and Ke~Wei.
\newblock Elementary analysis of policy gradient methods.
\newblock {\em arxiv:2404.03372}, 2024.

\bibitem{liu2020improved}
Yanli Liu, Kaiqing Zhang, Tamer Basar, and Wotao Yin.
\newblock An improved analysis of (variance-reduced) policy gradient and
  natural policy gradient methods.
\newblock {\em Advances in Neural Information Processing Systems},
  33:7624--7636, 2020.

\bibitem{mei2021normalized}
Jincheng Mei, Yue Gao, Bo~Dai, Csaba Szepesvári, and Dale Schuurmans.
\newblock Leveraging non-uniformity in first-order non-convex optimization.
\newblock In {\em International Conference on Machine Learning}, 2021.

\bibitem{Mei_Xiao_Szepesvari_Schuurmans_2020}
Jincheng Mei, Chenjun Xiao, Csaba Szepesvári, and Dale Schuurmans.
\newblock On the global convergence rates of softmax policy gradient methods.
\newblock In {\em International Conference on Machine Learning}, pages
  6820--6829, 2020.

\bibitem{robot3}
Takahiro Miki, Joonho Lee, Jemin Hwangbo, Lorenz Wellhausen, Vladlen Koltun,
  and Marco Hutter.
\newblock Learning robust perceptive locomotion for quadrupedal robots in the
  wild.
\newblock {\em Science Robotics}, 7(62), 2022.

\bibitem{mondal2024improved}
Washim~U Mondal and Vaneet Aggarwal.
\newblock Improved sample complexity analysis of natural policy gradient
  algorithm with general parameterization for infinite horizon discounted
  reward markov decision processes.
\newblock In {\em International Conference on Artificial Intelligence and
  Statistics}, pages 3097--3105. PMLR, 2024.

\bibitem{murthy2023performance}
Yashaswini Murthy, Mehrdad Moharrami, and R~Srikant.
\newblock Performance bounds for policy-based average reward reinforcement
  learning algorithms.
\newblock {\em Advances in Neural Information Processing Systems},
  36:19386--19396, 2023.

\bibitem{peters2008natural}
Jan Peters and Stefan Schaal.
\newblock Natural actor-critic.
\newblock {\em Neurocomputing}, 71(7-9):1180--1190, 2008.

\bibitem{protopapas2024policy}
Kimon Protopapas and Anas Barakat.
\newblock Policy mirror descent with lookahead.
\newblock {\em Advances in Neural Information Processing Systems},
  37:26443--26481, 2024.

\bibitem{VI-PI}
Martin~L Puterman.
\newblock {\em Markov decision processes: discrete stochastic dynamic
  programming}.
\newblock Wiley Series in Probability and Statistics, 1994.

\bibitem{shani2020trpo}
Lior Shani, Yonathan Efroni, and Shie Mannor.
\newblock Adaptive trust region policy optimization: Global convergence and
  faster rates for regularized {MDP}s.
\newblock In {\em AAAI Conference on Artifical Intelligence}, 2020.

\bibitem{suttonRL}
Richard~S. Sutton and Andrew~G. Barto.
\newblock {\em Reinforcement Learning: An Introduction}.
\newblock MIT Press, Cambridge, 2018.

\bibitem{sutton1999policy}
Richard~S. Sutton, David McAllester, Satinder Singh, and Yishay Mansour.
\newblock Policy gradient methods for reinforcement learning with function
  approximation.
\newblock In {\em Advances in neural information processing systems}, 1999.

\bibitem{tomar2020mirror}
Manan Tomar, Lior Shani, Yonathan Efroni, and Mohammad Ghavamzadeh.
\newblock Mirror descent policy optimization.
\newblock {\em International Conference on Learning Representations}, 2022.

\bibitem{tsitsiklis2002convergence}
John~N Tsitsiklis.
\newblock On the convergence of optimistic policy iteration.
\newblock {\em Journal of Machine Learning Research}, 3(Jul):59--72, 2002.

\bibitem{vaswani2021functional}
Sharan Vaswani, Olivier Bachem, Simone Totaro, Robert Mueller, Matthieu Geist,
  Marlos~C Machado, Pablo~Samuel Castro, and Nicolas Le~Roux.
\newblock A functional mirror ascent view of policy gradient methods with
  function approximation.
\newblock {\em arXiv preprint arXiv:2108.05828}, 2021.

\bibitem{vieillard2020leverage}
Nino Vieillard, Tadashi Kozuno, Bruno Scherrer, Olivier Pietquin, R{\'e}mi
  Munos, and Matthieu Geist.
\newblock Leverage the average: an analysis of kl regularization in
  reinforcement learning.
\newblock {\em Advances in Neural Information Processing Systems},
  33:12163--12174, 2020.

\bibitem{StarCraft}
Oriol Vinyals, Timo Ewalds, Sergey Bartunov, Petko Georgiev, Alexander
  Vezhnevets, Michelle Yeo, Alireza Makhzani, Heinrich Küttler, JohnP.
  Agapiou, Julian Schrittwieser, John Quan, Stephen Gaffney, Stig Petersen,
  Karen Simonyan, Tom Schaul, Hadovan Hasselt, David Silver, TimothyP.
  Lillicrap, Kevin Calderone, Paul Keet, Anthony Brunasso, David Lawrence,
  Anders Ekermo, Jacob Repp, and Rodney Tsing.
\newblock Starcraft {II}: {A} new challenge for reinforcement learning.
\newblock {\em arXiv:1708.04782}, 2017.

\bibitem{wagener2021safe}
Nolan~C Wagener, Byron Boots, and Ching-An Cheng.
\newblock Safe reinforcement learning using advantage-based intervention.
\newblock In {\em International Conference on Machine Learning}, pages
  10630--10640. PMLR, 2021.

\bibitem{wang2019neural}
Lingxiao Wang, Qi~Cai, Zhuoran Yang, and Zhaoran Wang.
\newblock Neural policy gradient methods: Global optimality and rates of
  convergence.
\newblock In {\em International Conference on Learning Representations}, 2020.

\bibitem{wang2024non}
Yudan Wang, Yue Wang, Yi~Zhou, and Shaofeng Zou.
\newblock Non-asymptotic analysis for single-loop (natural) actor-critic with
  compatible function approximation.
\newblock In {\em International Conference on Machine Learning}, 2024.

\bibitem{winnicki2023convergence}
Anna Winnicki and R~Srikant.
\newblock On the convergence of policy iteration-based reinforcement learning
  with monte carlo policy evaluation.
\newblock In {\em International Conference on Artificial Intelligence and
  Statistics}, pages 9852--9878. PMLR, 2023.

\bibitem{Xiao_2022}
Lin Xiao.
\newblock On the convergence rates of policy gradient methods.
\newblock {\em Journal of Machine Learning Research}, 23(282):1--36, 2022.

\bibitem{xu2020improving}
Tengyu Xu, Zhe Wang, and Yingbin Liang.
\newblock Improving sample complexity bounds for (natural) actor-critic
  algorithms.
\newblock {\em Advances in Neural Information Processing Systems},
  33:4358--4369, 2020.

\bibitem{xu2020non}
Tengyu Xu, Zhe Wang, and Yingbin Liang.
\newblock Non-asymptotic convergence analysis of two time-scale (natural)
  actor-critic algorithms.
\newblock {\em arXiv preprint arXiv:2005.03557}, 2020.

\bibitem{yang2019provably}
Zhuoran Yang, Yongxin Chen, Mingyi Hong, and Zhaoran Wang.
\newblock Provably global convergence of actor-critic: A case for linear
  quadratic regulator with ergodic cost.
\newblock {\em Advances in neural information processing systems}, 32, 2019.

\bibitem{yuan2022linear}
Rui Yuan, Simon~S Du, Robert~M Gower, Alessandro Lazaric, and Lin Xiao.
\newblock Linear convergence of natural policy gradient methods with log-linear
  policies.
\newblock {\em International Conference on Learning Representations}, 2023.

\bibitem{yuan2022general}
Rui Yuan, Robert~M Gower, and Alessandro Lazaric.
\newblock A general sample complexity analysis of vanilla policy gradient.
\newblock In {\em International Conference on Artificial Intelligence and
  Statistics}, pages 3332--3380. PMLR, 2022.

\bibitem{Zhan_Cen_Huang_Chen_Lee_Chi_2021}
Wenhao Zhan, Shicong Cen, Baihe Huang, Yuxin Chen, Jason~D. Lee, and Yuejie
  Chi.
\newblock Policy mirror descent for regularized reinforcement learning: {A}
  generalized framework with linear convergence.
\newblock {\em SIAM Journal on Optimization}, 33(2):1061--1091, 2023.

\bibitem{Zhang_Koppel_Bedi_Szepesvari_Wang_2020}
Junyu Zhang, Alec Koppel, Amrit~Singh Bedi, Csaba Szepesvári, and Mengdi Wang.
\newblock {V}ariational policy gradient method for reinforcement learning with
  general utilities.
\newblock In {\em Advances in Neural Information Processing Systems}, pages
  4572--4583, 2020.

\bibitem{zhang2019convergence}
Kaiqing Zhang, Alec Koppel, Hao Zhu, and Tamer Ba{\c{s}}ar.
\newblock Convergence and iteration complexity of policy gradient method for
  infinite-horizon reinforcement learning.
\newblock In {\em 2019 IEEE 58th Conference on Decision and Control (CDC)},
  pages 7415--7422. IEEE, 2019.

\bibitem{zhang2020global}
Kaiqing Zhang, Alec Koppel, Hao Zhu, and Tamer Basar.
\newblock Global convergence of policy gradient methods to (almost) locally
  optimal policies.
\newblock {\em SIAM Journal on Control and Optimization}, 58(6):3586--3612,
  2020.

\end{thebibliography}
%%%%%%%%%%%%%%%%%%

%%%%%%%%%%%%%%%%%%%%%%%%%%%%%%%%%%%%%%%%%%%%%%%%%%%%%%%%%%%%

%\newpage
\appendix

% \begin{center}
%     \Large{\textbf{Appendix}}
% \end{center}

% \startcontents
% \printcontents{}{1}{}
%%%%%%%%%%%%%%%%%%%%%%

% \section{The Roadmap of the Convergence Analysis of Section~\ref{sec:exact-TD-PMD}}

% \begin{figure*}[ht!]
%     \centering
%     \includegraphics[width=0.9\textwidth]{Figures/Pipeline.pdf}
%     \caption{The roadmap of the convergence analysis for exact TD-PMD. {\color{red}delete, revise,how to put policy convergence}}
%     \label{fig:pipeline}
% \end{figure*}

%%%%%%%%%%%%%%%%%%%%%%
\section{Useful lemmas}
In this section we present several  lemmas that will be used throughout the analysis. The proofs of these lemmas are omitted as they can be easily verified or can be found in previous works.

\begin{lemma}
    Under the assumption of $r(s,a) \in [0,1]$, for any $\pi \in \Pi$ there holds
    \begin{align*}
        \forall \, s\in\calS, \, a\in\calA: \quad V^\pi(s) \in \ls[0, \frac{1}{1-\gamma}\rs], \;\; Q^\pi(s,a) \in \ls[ 0, \frac{1}{1-\gamma} \rs].
    \end{align*}
    \label{lem:bounded-values}
\end{lemma}

\begin{lemma}
    For any $V, V^\prime \in \mathbb{R}^{|\calS|}$ and $\pi \in \Pi$, one has
    \begin{itemize}
        \item[1)]  $\calT^\pi V^\pi = V^\pi$ and $\calT V^* = V^*$;
        \item[2)]  If $V \geq V^\prime$, then $\calT^\pi V \geq \calT^\pi V^\prime$ and $\calT V \geq \calT V^\prime$;
        \item[3)]  $
        \ls\| \calT^\pi V - \calT^\pi V^\prime  \rs\|_\infty \leq \gamma \ls\| V - V^\prime \rs\|_\infty$;
        \item[4)] $
        \ls\| \calT V - \calT V^\prime \rs\|_\infty \leq \gamma \ls\| V - V^\prime \rs\|_\infty$;
        \item[5)] $\calT^\pi[V + c\cdot \mathbf{1}] = \calT^\pi V + \gamma\cdot c\cdot \mathbf{1}$ for any constant $c$.
    \end{itemize}
    \label{lemma:Bellman-property}
\end{lemma}
%Recall the optimal state and state-action values $V^*$, $Q^*$. The optimal action set of state $s$ is defined as
%\begin{align*}
%    \calA^*_s = \underset{a\in\calA}{\arg\max} \; Q^*(s,a) = \underset{a\in\calA}{\arg\max}\;  \ls\{r(s,a) + \gamma \cdot \E_{s^\prime \sim P(\cdot|s,a)}[V^*(s^\prime)]\rs\}.
%\end{align*}

\iffalse
The following lemma shows that the optimal policy $\pi^*$ is always supported on the optimal actions sets.
\begin{lemma}
    For any optimal policy $\pi^*$, there holds
    \begin{align*}
        \forall \, s\in\calS, \; a^\prime \not\in \calA^*_s: \quad \pi^*(a^\prime|s) = 0.
    \end{align*}
    \label{lem:optimal-policy-form}
\end{lemma}
\fi

The three-point descent lemma~\cite{chen1993convergence} is crucial for the existing PMD analysis~\cite{Xiao_2022,Lan_2021,yuan2023general,Johnson_Pike-Burke_Rebeschini_2023}. %The following is the slight variation from~\cite{Xiao_2022}.

\begin{lemma}[Three-point descent lemma,~\cite{Xiao_2022,chen1993convergence}]
    Suppose that $\mathcal{C} \subset \mathbb{R}^n$ is a closed convex set, $\phi: \mathcal{C} \to \mathbb{R}$ is a proper, closed convex function, $D_h$ is the Bregman divergence generated by a function $h$ of Legendre type and $\mathrm{rint} \, \mathrm{dom}\, h \cap  \mathcal{C} \neq \varnothing$. For any $x\in\mathrm{rint}\,\mathrm{dom}\,h$, let
    \begin{align*}
        x^+ = \underset{u\in\mathcal{C}}{\arg\min} \, \{ \phi(u) + D_h(u,x) \}.
    \end{align*}
    Then $x^+ \in \mathrm{rint}\, \mathrm{dom}\, h\cap \mathcal{C}$ and for any $u\in\mathcal{C}$,
    \begin{align*}
        \phi(x^+) + D_h(x^+, x) \leq \phi(u) + D_h(u,x) - D_h(u, x^+).
    \end{align*}    
    \label{lem:original-three-point}
\end{lemma}

In the context of TD-PMD, at the policy improvement step of the $k$-th iteration (equation~\eqref{eq:TD-PMD-policy-step} in Algorithm~\ref{alg:TD-PMD} and equation~\eqref{eq:TD-PMD-policy-step-inexact} in Algorithm~\ref{alg:sampling-TD-PMD}), it is natural to set $\mathcal{C}=\Delta(\calA)$, $\phi(p) = -\eta_k\langle p, Q^k(s,\cdot) \rangle$ or $\phi(p) = -\eta_k\langle p, \widehat{Q}^k(s,\cdot)\rangle$ in Lemma~\ref{lem:original-three-point}, yielding the following useful result.

\begin{lemma}
   For the exact TD-PMD (Algorithm~\ref{alg:TD-PMD}), there holds
    \begin{align*}
        \forall\, s\in\calS,\; p\in\Delta(\calA): \;\; \eta_k \langle \pi_{k+1}(\cdot|s) - p, \, Q^k(s,\cdot) \rangle \geq D^{\pi_{k+1}}_{\pi_k}(s) + D^p_{\pi_{k+1}}(s) - D^p_{\pi_k}(s).
    \end{align*}
    Similarly, for the sample based TD-PMD (Algorithm~\ref{alg:sampling-TD-PMD}),  there holds
    \begin{align*}
        \forall\, s\in\calS,\; p\in\Delta(\calA): \;\; \eta_k \langle \pi_{k+1}(\cdot|s) - p, \, \widehat{Q}^k(s,\cdot) \rangle \geq D^{\pi_{k+1}}_{\pi_k}(s) + D^p_{\pi_{k+1}}(s) - D^p_{\pi_k}(s).
    \end{align*}
    \label{lem:TD-PMD-three-point-original}
\end{lemma}

Lastly, the following  well-known Hoeffding's inequality for bounded variables will be used in the sample complexity analysis.
\begin{lemma}[Hoeffding's inequality]
    Let $\{X_i\}_{i=1}^N$ be i.i.d. random variables with mean $\E [X_i] = \mu$. Suppose  $X_i \in [a,b],\,\forall 1\leq i\leq N$. Then
    \begin{align*}
        \Pr \ls[ \ls| \frac{1}{N} \sum_{i=1}^N X_i - \mu \rs| > t \rs] \leq 2 \exp \ls( -\frac{2Nt^2}{(b-a)^2} \rs).
    \end{align*}
    \label{lem:hoeffding}
\end{lemma}

%\section{Proofs of Results in Section~\ref{sec:preliminary}}

\section{Proof of Lemma~\ref{lem:extended-pdl}}
\label{sec:pf:lem:extended-pdl}
We generalize the proof of the standard performance difference lemma in~\cite{Agarwal_Kakade_Lee_Mahajan_2019}. For any state $s \in \calS$,
\begin{align*}
    V^{\pi}(s) - V(s) &= \E  \ls[ \sum_{t=0}^\infty \gamma^t r(s_t, a_t) \; \Big| \; s_0=s \rs] - V(s) \\
    &= \E  \ls[ \sum_{t=0}^\infty \gamma^t \ls( r(s_t, a_t) + \gamma V(s_{t+1}) - \gamma V(s_{t+1}) \rs) \; \Big | \; s_0=s \rs] - V(s) \\
    &= \E  \ls[ \sum_{t=0}^\infty \gamma^t \ls( r(s_t, a_t) + \gamma V(s_{t+1}) - \gamma V(s_{t+1}) \rs) - V(s_0) \; \Big | \; s_0=s  \rs] \\
    &= \E  \ls[ \sum_{t=0}^\infty \gamma^t \ls( r(s_t, a_t) + \gamma V(s_{t+1}) - V(s_{t}) \rs) \; \Big | \; s_0=s  \rs] \\
    &= \sum_{t=0}^\infty \gamma^t \E_{s_t} \ls[ \E_{a_t, s_{t+1}} \ls[ r(s_t,a_t) + \gamma V(s_{t+1}) - V(s_t) \; \Big | \; s_t \rs] \; \Big | \; s_0 = s \rs] \\
    &= \sum_{t=0}^\infty \gamma^t \cdot \sum_{s^\prime} P(s_t=s^\prime \, | \, s_0=s) \ls[ \calT^\pi V - V \rs](s^\prime) \\
    &= \sum_{s^\prime} \ls[ \sum_{t=0}^\infty \gamma^t P(s_t = s^\prime \, | \, s_0 = s) \rs]  \ls[ \calT^\pi V - V \rs] (s^\prime) \\
    &= \frac{1}{1-\gamma} \ls[ \calT^\pi V - V \rs](d^\pi_s)
\end{align*}
where $d_s^\pi(s^\prime) = (1-\gamma) \sum_{t=0}^\infty \gamma^t \prob(s_t=s^\prime\, |\, s_0=s, \, \pi) $. 
Taking expectation on both sides with respect to $\mu$ over the state space $\calS$ yields the general performance difference lemma.

\section{Proofs for results in Section~\ref{sec:sublinaer-convergence}}
\subsection{Proof of Lemma~\ref{lem:sublinear-convergence-lemma}}
\label{sec:pf:lem:sublinear-convergence-lemma}
The overall proof procedure is similar to the PMD analysis in~\cite{Xiao_2022,Lan_2021}. First we leverage the three-point descent lemma. Recalling the update rule of TD-PMD, by Lemma~\ref{lem:TD-PMD-three-point-original}, one has
\begin{align*}
    \forall\, s\in\calS,\; p\in\Delta(\calA): \;\; \eta \langle \pi_{k+1}(\cdot|s) - p, \, Q^k(s,\cdot) \rangle \geq D^{\pi_{k+1}}_{\pi_k}(s) + D^p_{\pi_{k+1}}(s) - D^p_{\pi_k}(s).
\end{align*}
Setting $p = \pi_k(\cdot|s)$ yields %equation~\eqref{eq:TD-PMD-three-point-1}:
\begin{align*}
    \eta \langle \pi_{k+1}(\cdot|s) - \pi_k(\cdot|s), \, Q^k(s,\cdot) \rangle &= \eta \ls[ \E_{a\sim\pi_{k+1}(\cdot|s)}[Q^k(s,a)] - \E_{a\sim\pi_k(\cdot|s)}[Q^k(s,a)] \rs] \\
    &= \eta [\calT^{\pi_{k+1}}V^k - \calT^{\pi_k}V^k](s) \\
    &\geq D^{\pi_{k+1}}_{\pi_k}(s) + D^{\pi_k}_{\pi_{k+1}}(s) - D^{\pi_k}_{\pi_k}(s) \\
    &= D^{\pi_{k+1}}_{\pi_k}(s) + D^{\pi_k}_{\pi_{k+1}}(s). \numberthis \label{eq:TD-PMD-three-point-1}
\end{align*}
Setting $p = \pi^*(\cdot|s)$ yields %equation~\eqref{eq:TD-PMD-three-point-2}:
\begin{align*}
    \eta \langle \pi_{k+1}(\cdot|s) - \pi^*(\cdot|s), \, Q^k(s,\cdot) \rangle &= \eta \ls[ \E_{a\sim\pi_{k+1}(\cdot|s)}[Q^k(s,a)] - \E_{a\sim\pi^*(\cdot|s)}[Q^k(s,a)] \rs] \\
    &= \eta [\calT^{\pi_{k+1}}V^k - \calT^{\pi^*}V^k](s) \\
    &\geq D^{\pi_{k+1}}_{\pi_k}(s) + D^{\pi^*}_{\pi_{k+1}}(s) - D^{\pi^*}_{\pi_k}(s). \numberthis \label{eq:TD-PMD-three-point-2}
\end{align*}

Subtracting $\eta V^k(s)$ on both sides of equation~\eqref{eq:TD-PMD-three-point-2}, dropping the non-negative term $D^{\pi_{k+1}}_{\pi_k}(s)$ and noting $V^{k+1} = \calT^{\pi_{k+1}} V^k$, we get
\begin{align*}
    \eta [V^{k+1} - V^k](s) &\geq \eta [\calT^{\pi^*} V^k - V^k](s) + D^{\pi^*}_{\pi_{k+1}}(s) - D^{\pi^*}_{\pi_k}(s). \numberthis \label{eq:TD-PMD-improvement-1}
\end{align*}
 Taking expectation with respect to  $s\sim d^*_\mu$ on both sides of equation~\eqref{eq:TD-PMD-improvement-1}, and further applying the general performance difference lemma (Lemma~\ref{lem:extended-pdl}) yields
\begin{align*}
    \frac{1}{1-\gamma} [V^{k+1} - V^k](d^*_\mu) &\geq [V^* - V^k](\mu) + \frac{1}{\eta(1-\gamma)} [D^{\pi^*}_{\pi_{k+1}} - D^{\pi^*}_{\pi_k}](d^*_\mu), \numberthis \label{eq:TD-PMD-improvement-2}
\end{align*}
Taking a summation on both sides of equation~\eqref{eq:TD-PMD-improvement-2}  from $0$ to $T-1$, after rearrangement, we obtain
\begin{align*}
    \frac{1}{T}\sum_{k=0}^{T-1} [V^* - V^k](\mu) &\leq \frac{1}{T(1-\gamma)} [V^{T} - V^0](d^*_\mu) + \frac{1}{T\eta(1-\gamma)}[D^{\pi^*}_{\pi_0} - D^*_{\pi_{T}}](d^*_\mu).
\end{align*}

\subsection{Proof of Lemma~\ref{lem:monotonicity-property}}
\label{sec:pf:lem:monotonicity-property}
It is trivial that $V^* \geq V^{\pi_{k+1}}$. In addition,  $\calT^{\pi_{k+1}} V^k \geq \calT^{\pi_k} V^k$ is implied by equation~\eqref{eq:TD-PMD-three-point-1}. Thus we only need to prove $(a)$ and $(b)$ in
\begin{align*}
    V^{\pi_{k+1}} \stackrel{(b)}{\geq} V^{k+1} = \calT^{\pi_{k+1}} V^k \geq \calT^{\pi_k} V^k \stackrel{(a)}{\geq} V^k
\end{align*}
for $k=0,...,T-1$.

\paragraph{(a) $\calT^{\pi_k} V^k \geq V^k$:} We will prove this result by induction. The base case holds due to the assumption on the initialization. Assume  $(a)$  holds for the $(k-1)$-th iteration. For $k$-th iteration, one has
\begin{align*}
    \calT^{\pi_k} V^k &= \calT^{\pi_k} [\calT^{\pi_k} V^{k-1}] \\
    &\geq \calT^{\pi_k} [\calT^{\pi_{k-1}} V^{k-1}] \\
    &\geq \calT^{\pi_k} [V^{k-1}] \\
    & = V^k.
\end{align*}
Here the first inequality follows from $\calT^{\pi_k} V^{k-1} \geq \calT^{\pi_{k-1}}V^{k-1}$ (equation~\eqref{eq:TD-PMD-three-point-1}), and the second inequality follows from the induction hypothesis and the monotonicity property of Bellman operator (Lemma~\ref{lemma:Bellman-property}).

\paragraph{(b) $V^{\pi_{k+1}} \geq V^{k+1}$:} After $(a)$ is established, it follows immediately that
\begin{align*}
    V^{k+1} = \calT^{\pi_{k+1}}V^k \geq \calT^{\pi_k}V^k \geq V^k.
\end{align*}
By the monotonicity property of Bellman operator,
\begin{align*}
    \calT^{\pi_{k+1}}V^{k+1} \geq \calT^{\pi_{k+1}} V^k = V^{k+1}.
\end{align*}
The application of the performance difference lemma (Lemma~\ref{lem:extended-pdl}) yields that
\begin{align*}
    \forall\, s\in\calS: \quad V^{\pi_{k+1}}(s) - V^{k+1}(s) = \frac{1}{1-\gamma} \ls[ \calT^{\pi_{k+1}}V^{k+1} - V^{k+1} \rs](d^{\pi_{k+1}}_{s}) \geq 0,
\end{align*}
which completes the proof.

\subsection{Proof of Theorem~\ref{thm:sublinear-convergence-value-spec-init}}
\label{sec:pf:thm:sublinear-convergence-value-spec-init}
Recalling  equation~\eqref{eq:TD-PMD-error-upper-bound-1}, it holds that
\begin{align*}
    \frac{1}{T+1} \sum_{k=0}^{T} [V^* - V^k](\mu) \leq \frac{1}{(T+1)(1-\gamma)} [V^{T+1}-V^0](d^*_\mu) + \frac{1}{\eta(T+1)(1-\gamma)}[D^{\pi^*}_{\pi_0} - D^{\pi^*}_{\pi_{T+1}}](d^*_\mu).
    % \label{eq:TD-PMD-error-upper-bound-2}
\end{align*}
Together with Lemma~\ref{lem:monotonicity-property}, one has
\begin{align*}
    [V^* - V^T](\mu) &\leq \frac{1}{T+1} \sum_{k=0}^{T} [V^* - V^k](\mu) \\
    &\leq \frac{1}{(T+1)(1-\gamma)} [V^{T+1}-V^0](d^*_\mu) + \frac{1}{\eta(T+1)(1-\gamma)}[D^{\pi^*}_{\pi_0} - D^{\pi^*}_{\pi_{T+1}}](d^*_\mu)\\
    &\leq \frac{1}{(T+1)(1-\gamma)} \ls[ \frac{1}{1-\gamma} - V^0 \rs](d^*_\mu) + \frac{1}{\eta(T+1)(1-\gamma)} D^{\pi^*}_{\pi_0}(d^*_\mu) \\
    &\leq \frac{1}{(T+1)(1-\gamma)} \ls[\frac{1}{1-\gamma} + \ls\|V^0\rs\|_\infty \rs] + \frac{1}{\eta(T+1)(1-\gamma)} \ls\|D^{\pi^*}_{\pi_0}\rs\|_\infty,
\end{align*}
where the third inequality follows from $V^{T+1} \leq V^* \leq 1/(1-\gamma)$ (Lemma~\ref{lem:bounded-values}) and $D^{\pi^*}_{\pi_{T+1}}(d^*_\mu) \geq 0$. Since $V^* - V^T \geq 0$ and the above inequality holds for any $\mu\in \Delta(\calS)$, it follows that
\begin{align*}
    \ls\| V^* - V^T \rs\|_\infty \leq \frac{1}{T+1} \ls( \frac{1}{(1-\gamma)^2} \!+\! \frac{\ls\| V^0 \rs\|_\infty}{1-\gamma} \!+\! \frac{\ls\| D^{\pi^*}_{\pi_0} \rs\|_\infty}{\eta(1-\gamma)} \rs).
\end{align*}
Moreover, by Lemma~\ref{lem:monotonicity-property}, one has $V^* \geq V^{\pi_T} \geq V^T$. Thus,
\begin{align*}
    \ls\| V^* - V^{\pi_T} \rs\|_\infty \leq \ls\| V^* - V^T \rs\|_\infty.
\end{align*}
The proof is now complete.

\subsection{Proof of Proposition~\ref{pro:init-is-good}}
\label{sec:pf:pro:init-is-good}
When $\kappa_0=0$, there holds  $\calT^{\pi_0} V^0 \geq V^0$, so does $\calT^{\tilde{\pi}_0} \tilde{V}^0 \geq \tilde{V}^0$ since $(\tilde{V}^0,\tilde{\pi}_0)=({V}^0,{\pi}_0)$ in this case. 

Next consider the case where $\kappa_0 = (1-\gamma)^{-1}\max_{s\in\calS} \, [V^0 - \calT^{\pi_0}V^0](s)$. It follows that
\begin{align*}
    (1-\gamma) \kappa_0 \cdot \mathbf{1} \geq V^0 - \calT^{\pi_0} V^0.
\end{align*}
Consequently,
\begin{align*}
    \calT^{\pi_0} V^0 - \gamma\cdot \kappa_0 \cdot \mathbf{1} \geq (V^0 - \kappa_0 \cdot \mathbf{1}).
\end{align*}
By the fifth property of the Bellman operator in Lemma~\ref{lemma:Bellman-property}, one has
\begin{align*}
    \calT^{\pi_0} V^0 - \gamma\cdot \kappa_0 \cdot \mathbf{1} = \calT^{\pi_0} [ V^0 - \kappa_0 \cdot \mathbf{1}] \geq [V^0 - \kappa_0 \cdot \mathbf{1}],
\end{align*}
which completes the proof.

\subsection{Proof of Lemma~\ref{lem:relation-between-two-sequences}}
\label{sec:pf:lem:relation-between-two-sequences}
For the base case, the relation holds by the construction of $\tilde{V}^0$ and $\tilde{\pi}_0$. Assume the relation holds for $(k-1)$-th iteration, i.e.,
\[
V^{k-1} = \tilde{V}^{k-1}+\gamma^{k-1}\cdot\kappa_0\cdot 1,\quad\mbox{and }\tilde{\pi}_{k-1}=\pi_{k-1}.
\]
For the $k$-th iteration, first note that
\begin{align*}
    \la p, \, {Q}^{k-1}(s,\cdot) \ra &= \sum_a p(a) Q^{k-1}(s,a) \\
    &= \sum_a p(a) [r(s,a) + \gamma \E_{s^\prime \sim P(\cdot | s,a)} [V^{k-1}(s^\prime)] ] \\
    &= \sum_{a} p(a) [r(s,a) + \gamma \E_{s^\prime}[\tilde{V}^{k-1}(s^\prime) + \gamma^{k-1} \kappa_0]] \\
    &= \sum_a p(a) [r(s,a) + \gamma \E_{s^\prime}[\tilde{V}^{k-1}(s^\prime)] + \gamma^{k}\kappa_0] \\
    &= \langle p, \, \tilde{Q}^{k-1}(s,a) \rangle + \gamma^{k} \kappa_0,
\end{align*}
where $\tilde{Q}^{k-1}$ is induced by $\tilde{V}^{k-1}$, and the third line follows from the induction hypothesis. Thus for all $s\in\calS$,
\begin{align*}
    \pi_{k}(\cdot|s) &= \arg\max \; \langle p, \, \eta \,Q^{k-1}(s,\cdot) \rangle - D^p_{\pi_{k-1}}(s)\\
    &= \arg\max \; \langle p, \eta\, \tilde{Q}^{k-1}(s,\cdot) \rangle + \eta \,\gamma^{k} \kappa_0 - D^p_{\tilde{\pi}_{k-1}}(s)\\
    &= \arg\max \; \langle p, \eta\, \tilde{Q}^{k-1}(s,\cdot) \rangle - D^{p}_{\tilde{\pi}_{k-1}}(s)\\
    &= \tilde{\pi}_k(\cdot|s),
\end{align*}
where the second line follows from the induction hypothesis. 

For the relation between $V^k$ and $\tilde{V}^k$, a direct calculation yields that for all $s\in\mathcal{S}$,
\begin{align*}
    V^{k}(s) &= \calT^{\pi_{k}} V^{k-1}(s) \\
    &= \calT^{\pi_k} [\tilde{V}^{k-1}(s) + \gamma^{k-1} \kappa_0] \\
    &= \E_{a\sim\pi_k(\cdot|s), \, s^\prime\sim P(\cdot|s,a)} [r(s,a) + \gamma \cdot [\tilde{V}^{k-1}(s^\prime) + \gamma^{k-1} \kappa_0]] \\
    &= \gamma^k \cdot \kappa_0 + \E_{a\sim\pi_k(\cdot|s), \, s^\prime\sim P(\cdot|s,a)} [r(s,a) + \gamma \tilde{V}^{k-1}(s^\prime)] \\
    &= \gamma^k \cdot \kappa_0 + \calT^{\pi_k} \tilde{V}^{k-1}(s) \\
    &= \gamma^k \cdot \kappa_0 + \calT^{\tilde{\pi}_k} \tilde{V}^{k-1}(s) \\
    &= \gamma^k \cdot \kappa_0 + \tilde{V}^k,
\end{align*}
where the second line is from the induction hypothesis and the sixth line comes from $\tilde{\pi}_k = \pi_k$. %As it holds for all $s\in\calS$ we get $V^k = \tilde{V}^k + \gamma^k \kappa_0 \cdot \mathbf{1}$.

\subsection{Proof of Theorem~\ref{thm:sublinear-convergence-of-exact-TD-PMD-formal}}
\label{sec:pf:thm:sublinear-convergence-of-exact-TD-PMD-formal}
As $\calT^{\tilde{\pi}_0} \tilde{V}^0 \geq \tilde{V}^0$, the application of Theorem~\ref{thm:sublinear-convergence-value-spec-init} yields that
\begin{align}
    \ls\| V^* - V^{\tilde{\pi}_T} \rs\|_\infty \leq \frac{1}{T+1} \ls( \frac{1}{(1-\gamma)^2} \!+\! \frac{\| \tilde{V}^0 \|_\infty}{1-\gamma} \!+\! \frac{\ls\| D^{\pi^*}_{\tilde{\pi}_0} \rs\|_\infty}{\eta(1-\gamma)} \rs)
    \label{eq:sublinear-convergence-of-output-policy-tilde}
\end{align}
and 
\begin{align}
    \ls\| V^* - \tilde{V}^T \rs\|_\infty \leq \frac{1}{T+1} \ls( \frac{1}{(1-\gamma)^2} \!+\! \frac{\| \tilde{V}^0 \|_\infty}{1-\gamma} \!+\! \frac{\ls\| D^{\pi^*}_{\tilde{\pi}_0} \rs\|_\infty}{\eta(1-\gamma)} \rs).
    \label{eq:sublinear-convergence-of-output-value-tilde}
\end{align}
By Lemma~\ref{lem:relation-between-two-sequences}, $\tilde{\pi}_k = \pi_k$ holds for $k=0$ and $T$. Thus equation~\eqref{eq:sublinear-convergence-of-output-policy-tilde} implies
\begin{align*}
    \ls\| V^* - V^{\pi_T} \rs\|_\infty &=  \ls\| V^* - V^{\tilde{\pi}_T} \rs\|_\infty \\
    &\leq \frac{1}{T+1} \ls( \frac{1}{(1-\gamma)^2} \!+\! \frac{\| \tilde{V}^0 \|_\infty}{1-\gamma} \!+\! \frac{\ls\| D^{\pi^*}_{\tilde{\pi}_0} \rs\|_\infty}{\eta(1-\gamma)} \rs) \\
    &= \frac{1}{T+1} \ls( \frac{1}{(1-\gamma)^2} \!+\! \frac{\| \tilde{V}^0 \|_\infty}{1-\gamma} \!+\! \frac{\ls\| D^{\pi^*}_{{\pi}_0} \rs\|_\infty}{\eta(1-\gamma)} \rs) \\
    &\leq \frac{1}{T+1} \ls( \frac{1}{(1-\gamma)^2} \!+\! \frac{\| {V}^0 \|_\infty + \kappa_0}{1-\gamma} \!+\! \frac{\ls\| D^{\pi^*}_{{\pi}_0} \rs\|_\infty}{\eta(1-\gamma)} \rs),
\end{align*}
where the last inequality comes from $\tilde{V}^0 = V^0 - \kappa_0 \cdot \mathbf{1}$. 

To bound $\ls\| V^* - V^T \rs\|_\infty$, first note that %by Lemma~\ref{lem:relation-between-two-sequences}
\begin{align*}
    \ls\| V^* - V^T \rs\|_\infty &= \ls\| V^* - \tilde{V}^T - \gamma^T \kappa_0 \cdot \mathbf{1} \rs\|_\infty \\
    &\leq \ls\| V^* - \tilde{V}^T \rs\|_\infty + \gamma^T \kappa_0.
\end{align*}
Plugging it into equation~\eqref{eq:sublinear-convergence-of-output-value-tilde} gives
\begin{align*}
    \ls\| V^* - V^T \rs\|_\infty &\leq \frac{1}{T+1} \ls( \frac{1}{(1-\gamma)^2} \!+\! \frac{\| \tilde{V}^0 \|_\infty}{1-\gamma} \!+\! \frac{\ls\| D^{\pi^*}_{\tilde{\pi}_0} \rs\|_\infty}{\eta(1-\gamma)} \rs) + \gamma^T \kappa_0 \\
    &\leq \frac{1}{T+1}\ls( \frac{1}{(1-\gamma)^2} \!+\! \frac{\| {V}^0 \|_\infty + \kappa_0}{1-\gamma} \!+\! \frac{\ls\| D^{\pi^*}_{{\pi}_0} \rs\|_\infty}{\eta(1-\gamma)} \rs) + \gamma^T \kappa_0,
\end{align*}
which completes the proof.

\section{Proofs for results in Section~\ref{sec:linear-convergence}}

\subsection{Proof of Lemma~\ref{lem:error-bounds}}
\label{sec:pf:lem:error-bounds}
First note the following error decomposition
\begin{align}
    \| V^* - V^{\pi_T} \|_\infty \leq \| V^* - V^T \|_\infty + \| V^{\pi_T} - V^T \|_\infty. 
    \label{eq:error-decomposition}
\end{align}
By a direct computation, 
\begin{align*}
    \ls\| V^{\pi_T} - V^T \rs\|_\infty &= \ls\| \calT^{\pi_T} V^{\pi_T} - \calT^{\pi_T} V^{T-1} \rs\|_\infty \\
    &\leq \gamma \ls\| V^{\pi_T} - V^{T-1} \rs\|_\infty \\
    &\leq \gamma \ls\| V^{\pi_T} - V^T \rs\|_\infty + \gamma \ls\| V^T - V^{T-1} \rs\|_\infty,
\end{align*}
where the first line uses $\calT^{\pi_T} V^{\pi_T} = V^{\pi_T}$ and the second line uses the contraction property of the Bellman operator (Lemma~\ref{lemma:Bellman-property}).

After rearrangement, one has
\begin{align*}
    \ls\| V^{\pi_T} - V^T \rs\|_\infty &\leq \frac{\gamma}{1-\gamma} \ls\| V^T - V^{T-1} \rs\|_\infty \\
    &\leq \frac{\gamma}{1-\gamma} \ls( \ls\| V^* - V^{T} \rs\|_\infty + \ls\| V^* - V^{T-1} \rs\|_\infty \rs).
\end{align*}
Plugging it into equation~\eqref{eq:error-decomposition} yields that
\begin{align*}
    \ls\| V^* - V^{\pi_T} \rs\|_\infty &\leq \ls\| V^* - V^T \rs\|_\infty + \frac{\gamma}{1-\gamma} \ls( \ls\| V^* - V^{T} \rs\|_\infty + \ls\| V^* - V^{T-1} \rs\|_\infty \rs) \\
    &= \frac{1}{1-\gamma} \ls\| V^* - V^{T} \rs\|_\infty + \frac{\gamma}{1-\gamma} \ls\| V^* - V^{T-1} \rs\|_\infty \\
    &\leq \frac{1}{1-\gamma} \ls( \ls\| V^* - V^T \rs\|_\infty + \ls\| V^* - V^{T-1} \rs\|_\infty \rs),
\end{align*}
which completes the proof.

\subsection{Proof of Lemma~\ref{lem:TD-PMD-three-point-linear}}
\label{sec:pf:lem:TD-PMD-three-point-linear}
Recalling Lemma~\ref{lem:TD-PMD-three-point-original}  for TD-PMD, there holds
\begin{align*}
    \forall\, s\in\calS,\; p\in\Delta(\calS): \;\; \eta_k \langle \pi_{k+1}(\cdot|s) - p, \, Q^k(s,\cdot) \rangle \geq D^{\pi_{k+1}}_{\pi_k}(s) + D^p_{\pi_{k+1}}(s) - D^p_{\pi_k}(s).
\end{align*}
Setting $p=\tilde{\pi}_k$ yields
\begin{align*}
        \langle \pi_{k+1}(\cdot|s) - \tilde{\pi}_k(\cdot|s), \, {Q}^k(s,\cdot) \rangle \geq \frac{1}{\eta_k}[D^{\pi_{k+1}}_{\pi_k}(s) + D^{\tilde{\pi}_k}_{\pi_{k+1}}(s) - D^{\tilde{\pi}_k}_{\pi_k}(s)].
\end{align*}
By the definition of $\tilde{\pi}_k$, it is easy to see that
\begin{align*}
    \forall\, s\in\calS: \quad \calT^{\tilde{\pi}_k} V^k(s) = \calT V^k(s),
\end{align*}
Therefore
\begin{align*}
    [\calT^{\pi_{k+1}} V^k - \calT V^k](s) &= [\calT^{\pi_{k+1}} V^k - \calT^{\tilde{\pi}_k}V^k](s) \\
    &= \E_{a\sim\pi_{k+1}(\cdot|s)} [Q^k(s,a)] - \E_{a\sim\tilde{\pi}(\cdot|s)}[Q^k(s,a)] \\
    &= \langle \pi_{k+1}(\cdot|s) - \tilde{\pi}_k(\cdot|s), \, {Q}^k(s,\cdot) \rangle \\
    &\geq \frac{1}{\eta_k}[D^{\pi_{k+1}}_{\pi_k}(s) + D^{\tilde{\pi}_k}_{\pi_{k+1}}(s) - D^{\tilde{\pi}_k}_{\pi_k}(s)] \\
    &\geq -\frac{1}{\eta_k} D^{\tilde{\pi}_k}_{\pi_k}(s)\\
    &\geq -\frac{1}{\eta_k} \| D_{\pi_k}^{\tilde{\pi}_k} \|_\infty.
\end{align*}

\subsection{Proof of Theorem~\ref{thm:linear-convergence-of-exact-TD-PMD}}
\label{sec:pf:thm:linear-convergence-of-exact-TD-PMD}
Notice that $\calT^{\pi_{k+1}}V^k = V^{k+1}$ and $\calT V^* = V^*$, Lemma~\ref{lem:TD-PMD-three-point-linear} implies that
\begin{align}
    [V^* - V^{k+1}](s) \leq [\calT V^* - \calT V^k](s) + \frac{1}{\eta_k} \ls\| D_{\pi_k}^{\tilde{\pi}_k} \rs\|_\infty.
    \label{eq:infty-error-bound-1}
\end{align}
On the other hand, it is evident that
\begin{align}
    [V^* - V^{k+1}](s) \geq [\calT V^* - \calT V^k](s)
    \label{eq:infty-error-bound-2}
\end{align}
as $V^{k+1} = \calT^{\pi_{k+1}}V^k \leq \calT V^k$. Combining equations \eqref{eq:infty-error-bound-1} and \eqref{eq:infty-error-bound-2} together, one has
\begin{align*}
    \ls\| V^* - V^{k+1} \rs\|_\infty &\leq \ls\| \calT V^* - \calT V^k \rs\|_\infty + \frac{1}{\eta_k} \ls\| D_{\pi_k}^{\tilde{\pi}_k} \rs\|_\infty \\
    &\leq \gamma \ls\| V^* - V^k \rs\|_\infty + \frac{1}{\eta_k} \ls\| D_{\pi_k}^{\tilde{\pi}_k} \rs\|_\infty. \numberthis \label{eq:infty-error-bound}
\end{align*}

Iterating this relation yields
\begin{align*}
    \ls\| V^* - V^{T} \rs\|_\infty &\leq \gamma \ls\| V^* - V^{T-1} \rs\|_\infty + \frac{1}{\eta_{T-1}} \ls\| D_{\pi_{T-1}}^{\tilde{\pi}_{T-1}} \rs\|_\infty \\
    &\leq \gamma \ls( \gamma \ls\| V^* - V^{T-2} \rs\|_\infty + \frac{1}{\eta_{T-2}} \| D_{\pi_{T-2}}^{\tilde{\pi}_{T-2}} \|_\infty \rs) + \frac{1}{\eta_{T-1}} \| D_{\pi_{T-1}}^{\tilde{\pi}_{T-1}} \|_\infty \\
    &= \gamma^2 \ls\| V^* - V^{T-2} \rs\|_\infty + \frac{\gamma}{\eta_{T-2}} \| D_{\pi_{T-2}}^{\tilde{\pi}_{T-2}} \|_\infty + \frac{1}{\eta_{T-1}} \| D_{\pi_{T-1}}^{\tilde{\pi}_{T-1}} \|_\infty \\
    &\leq \cdots \\
    &\leq \gamma^T \ls\| V^* - V^0 \rs\|_\infty + \sum_{k=0}^{T-1} \frac{\| D_{\pi_k}^{\tilde{\pi}_k} \|_\infty}{\eta_k} \cdot \gamma ^{T-k-1}.
\end{align*}
Thus, if $\eta_k \geq {\ls\| D_{\pi_k}^{\tilde{\pi}_k} \rs\|_\infty}\big/{(c \cdot \gamma^{2k+1})}$, there holds
\begin{align*}
    \ls\| V^* - V^{T} \rs\|_\infty &\leq \gamma^T \ls\| V^* - V^0 \rs\|_\infty + \sum_{k=0}^{T-1} c\cdot \gamma^{T+k} \\
    &\leq \gamma^T \ls\| V^* - V^0 \rs\|_\infty + \frac{c \cdot \gamma^T}{1-\gamma}.
\end{align*}
For the bound of $\ls\| V^* - V^{\pi_T} \rs\|_\infty$, by Lemma~\ref{lem:error-bounds}, one has 
\begin{align*}
    \ls\| V^* - V^{\pi_T} \rs\|_\infty &\leq \frac{1}{1-\gamma} \ls( \ls\| V^* - V^T \rs\|_\infty + \ls\| V^* - V^{T-1} \rs\|_\infty \rs) \\
    &\leq \frac{2}{1-\gamma} \ls(  \gamma^{T-1} \ls\| V^* - V^0 \rs\|_\infty + \frac{c \cdot \gamma^{T-1}}{1-\gamma}\rs),
\end{align*}
which completes the proof.

\section{Proofs for results in Section~\ref{sec:TD-PQA-and-TD-NPG}}

\subsection{Proof of Theorem~\ref{thm:finite-iteration-convergence-of-TD-PQA}}
\label{sec:pf:thm:finite-iteration-convergence-of-TD-PQA}
We first introduce the following property of the simplex projection $\Proj_{\Delta(\calA)}$. 
\begin{lemma}[\protect{\cite[Lemma~9]{ppgliu}}]
    Let $\mathcal{B}$ and $\mathcal{C}$ be two disjoint non-empty sets such that $\mathcal{A} =\mathcal{B}\cup \mathcal{C}$. Given an arbitrary vector $p=\left( p_a \right) _{a\in \mathcal{A}}\in \mathbb{R} ^{\left| \mathcal{A} \right|}$, let $y=\mathrm{Proj}_{\Delta \left( \mathcal{A} \right)}\left( p \right)$. Then
    \begin{align*}
        \forall a^\prime\in \mathcal{C},~ y_{a^\prime}=0 \Leftrightarrow \,\,\sum_{a\in \mathcal{B}}{\left( p_a-\underset{a^{\prime}\in \mathcal{C}}{\max}\,\,p_{a^{\prime}} \right) _+}\ge 1.
    \end{align*}
    \label{lem:gap-property-of-projection}
\end{lemma}
% The proof can be found in~\cite{ppgliu}.
Lemma~\ref{lem:gap-property-of-projection} implies that when the probability over $\mathcal{B}$ is larger enough than that over $\mathcal{C}$, the projection operator will exclude the action set $\mathcal{C}$ from the support set of the projection $y$. Based on Lemma~\ref{lem:gap-property-of-projection}, one can establish the following policy optimality condition for TD-PQA,
% To this end, recall the optimal action set $\calA^*_s = \arg\max_{a\in\calA} \, Q^*(s,a)$ and define the suboptimal probabilities as
% \begin{align*}
%     \forall\, s\in\calS: \quad b^\pi_s = \sum_{a^\prime\not\in\calA^*_s} \pi(a^\prime|s).
% \end{align*}
which shows that a one-step TD-PQA update will output an optimal policy when the total error is smaller than a threshold. 
\begin{lemma}[Optimality condition for TD-PQA]
    Consider TD-PQA with constant step size $\eta_k = \eta > 0$. If the $k$-th iteration of TD-PQA satisfies
    \begin{align*}
        \frac{1}{\Delta}\| V^* - V^{\pi_k} \|_\infty + \eta\gamma \| V^* - V^{k} \|_\infty \leq \frac{\eta\Delta}{2},
        \numberthis \label{eq:optimality-condition-for-TD-PQA}
    \end{align*}
    then the updated policy, i.e., $\pi^{k+1}$ is an optimal policy. 
    \label{lem:optimality-condition-for-TD-PQA}
\end{lemma}
\begin{proof}[Proof of Lemma~\ref{lem:optimality-condition-for-TD-PQA}]
    Recall the optimal action set $\calA^*_s = \arg\max_{a\in\calA} \, Q^*(s,a)$ and let $\tilde\calS = \{ s\in\calS: \; \calA^*_s \neq \calA \}$. It is convenient to denote by $b^\pi_s$ the probability on non-optimal actions,
    \begin{align*}
        \forall\, s\in\calS: \quad b^\pi_s = \sum_{a^\prime\not\in\calA^*_s} \pi(a^\prime|s). 
    \end{align*}
   According to the performance difference lemma (Lemma~\ref{lem:extended-pdl}), for any state distribution $\rho \in \Delta(\calS)$,
    \begin{align*}
        \| V^* - V^\pi \|_\infty &\geq V^*(\rho) - V^\pi(\rho)\\[.5em]
        &= -(V^\pi(\rho) - V^*(\rho)) \\
        &= -\frac{1}{1-\gamma} [\calT^\pi V^* - V^*](d^\pi_\rho) \\
        &= -\frac{1}{1-\gamma} \sum_{s\in\calS} d^\pi_\rho(s) \sum_{a\in\calA} \pi(a|s) [Q^*(s,a) - V^*(s)] \\
        &= -\frac{1}{1-\gamma} \sum_{s\in\calS} d^\pi_\rho(s) \sum_{a\in\calA} \pi(a|s) \ls[Q^*(s,a) - \max_{\tilde a\in\calA} Q^*(s,\tilde a)\rs] \\
        &= \frac{1}{1-\gamma} \sum_{s\in\tilde\calS} d^\pi_\rho(s) \sum_{a^\prime\not\in\calA^*_s} \pi(a^\prime|s) \ls[\max_{\tilde a\in\calA} Q^*(s,\tilde a) - Q^*(s,a) \rs] \\
        &\geq \frac{1}{1-\gamma} \sum_{s\in\tilde\calS} d^\pi_\rho(s) \sum_{a^\prime\not\in\calA^*_s} \pi(a^\prime|s) \cdot \Delta \\
        &\geq \frac{1}{1-\gamma} \sum_{s\in\tilde\calS} d^\pi_\rho(s) \cdot b^\pi_s \cdot \Delta  \\
        &\geq \Delta \cdot \E_{s\sim\rho} [b^\pi_s],
    \end{align*}
    where the last inequality is due to $d^\pi_\rho(s) \geq (1-\gamma)\cdot \rho(s), \, \forall \, s$. By setting $\rho$ as the point mass distribution over each state $s\in\calS$ we obtain
    \begin{align}
        \forall\, s\in\calS: \quad b_s^\pi \leq \frac{1}{\Delta} \| V^* - V^\pi \|_\infty.
        \label{eq:non-optimal-actions-bounded-by-value-error}
    \end{align}
    At the $k$-th iteration, the induced action value $Q^k$ satisfies
    \[
        \forall\, s,a: \quad | Q^k(s,a) - Q^*(s,a) | = \ls| \gamma \cdot \E_{s^\prime} [V^*(s^\prime) - V^k(s^\prime)] \rs| \leq \gamma \| V^* - V^k \|_\infty,
    \]
    which implies that 
    \begin{align*}
        \| Q^* - Q^k \|_\infty \leq \gamma \| V^* - V^k \|_\infty.
    \end{align*}
    Combining it with equation~\eqref{eq:non-optimal-actions-bounded-by-value-error}, condition~\eqref{eq:optimality-condition-for-TD-PQA} implies that
    \begin{align*}
        \forall\, s\in\calS: \quad b^{\pi_k}_s + \eta \| Q^* - Q^k \|_\infty \leq \frac{\eta\Delta}{2},
    \end{align*}
    Moreover, by a direct calculation,
    \begin{align*}
        &{\phantom{sdf}}\sum_{a^*\in\calA^*_s} \ls( \pi_k(a^*|s) + \eta Q^k(s,a^*) - \max_{a^\prime\not\in\calA^*_s} \ls[ \pi_k(a^\prime|s) + \eta Q^k(s,a^\prime) \rs] \rs) \\
        &\geq 1- 2 |\calA^*_s| b_s^k + \eta \sum_{a^*\in\calA^*_s} \ls[Q^k(s,a^*) - \max_{a^\prime\not\in\calA^*_s} Q^k(s,a^\prime) \rs] \\
        &\geq 1 - 2|\calA^*_s|b_s^k + \eta |\calA^*_s| \cdot [\Delta - 2\| Q^k - Q^* \|_\infty] \\[.5em]
        &\geq 1 - 2|\calA^*_s| \cdot [b_s^k + \eta \cdot \| Q^k - Q^* \|_\infty - \eta\Delta /2] \geq 1,
    \end{align*}
    where the second inequality is due to
    \begin{align*}
        \forall\, a^*\in\calA^*_s , \,\, a^\prime \not\in\calA^*_s: &\quad Q^k(s,a^*) - Q^k(s,a^\prime) \\
        &= Q^k(s,a^*) - Q^*(s,a^*) - [Q^k(s,a^\prime) - Q^*(s, a^\prime)] + [Q^*(s,a^*) - Q^*(s,a^\prime)] \\
        &\geq [Q^*(s,a^*) - Q^*(s,a^\prime)] - |Q^k(s,a^*) - Q^*(s,a^*)| - |Q^k(s,a^\prime) - Q^*(s, a^\prime)| \\
        &\geq \Delta - 2\| Q^* - Q^k \|_\infty.
    \end{align*}
    Therefore, for each state $s\in\calS$, setting $\mathcal{B} = \calA^*_s$, $\mathcal{C} = \calA \setminus \calA^*_s$ and $p = \pi_k(\cdot|s)$ in Lemma~\ref{lem:gap-property-of-projection} yields that
    \begin{align*}
        \forall\, s\in\calS, \, a^\prime\not\in\calA^*_s: \quad \pi_{k+1}(a^\prime|s) = \Proj_{\Delta(\calA)} \ls( \pi_k(\cdot|s) + \eta Q^k(s,\cdot) \rs)[a^\prime] = 0,
    \end{align*}
    which implies that $\pi_{k+1}$ is an optimal policy. 

    \paragraph{Proof of Theorem~\ref{thm:finite-iteration-convergence-of-TD-PQA}} With any constant step size $\eta > 0$, plugging
    \begin{align*}
        T_0 \!=\! \begin{cases}
            \!\ls\lceil \max  \ls\{ \dfrac{2\gamma}{\varepsilon} \!\ls( \dfrac{1}{(1-\gamma)^2}\!+\!\dfrac{\ls\| V_0 \rs\|_\infty \!+\! \kappa_0}{1-\gamma} \!+\! \dfrac{\ls\| D_0^* \rs\|_\infty}{\eta(1-\gamma)} \rs)\!, \; \dfrac{\log\varepsilon \!-\! \log 2\gamma \!-\! \log \kappa_0}{\log \gamma} \rs\} \!\rs\rceil, & \!\mbox{if } \kappa_0 > 0, \\[.7em]
            \!\ls\lceil \dfrac{2\gamma}{\varepsilon} \ls( \dfrac{1}{(1-\gamma)^2}\!+\!\dfrac{\ls\| V_0 \rs\|_\infty \!+\! \kappa_0}{1-\gamma} + \dfrac{\ls\| D_0^* \rs\|_\infty}{\eta(1-\gamma)} \rs) \rs\rceil, & \!\mbox{if } \kappa_0 = 0,
        \end{cases}  
    \end{align*}
    into equations~\eqref{eq:sublinear-convergence-of-output-policy} and~\eqref{eq:sublinear-convergence-of-output-value} in Theorem~\ref{thm:sublinear-convergence-of-exact-TD-PMD-formal} gives that
    \begin{align*}
        \forall\, k\geq T_0: \quad  \| V^* - V^{\pi_{k-1}} \|_\infty &\leq  \| V^* - V^{\pi_{T_0-1}} \|_\infty \leq \frac{\varepsilon}{2\gamma},\\
        \| V^* - V^{k-1} \|_\infty &\leq \| V^* - V^{T_0-1} \|_\infty \leq \frac{\varepsilon}{2\gamma} + \frac{\varepsilon}{2\gamma} = \frac{\varepsilon}{\gamma},
    \end{align*}
    where
    \begin{align*}
        \varepsilon = \frac{\eta\Delta}{2} \frac{\gamma\Delta}{\eta\gamma\Delta + 1}.
    \end{align*}
    One can verify that
    \begin{align*}
        \forall\, k\geq T_0: \quad \frac{1}{\Delta} \| V^* - V^{\pi_{k-1}} \|_\infty + \eta\gamma \| V^* - V^{k-1} \|_\infty \leq \frac{\eta\Delta}{2},
    \end{align*}
    which implies by Lemma~\ref{lem:optimality-condition-for-TD-PQA} that $ \pi_k $ is an optimal policy.
\end{proof}

\subsection{Proof of Theorem~\ref{thm:policy-convergence-of-TD-NPG}}\label{TD-NPG-policy-convergence}
% As we have established the global convergence results of $V^T$ and $V^{\pi_T}$, for any arbitrary small $\varepsilon > 0$, there exists a time $T(\varepsilon)$ such that
In the sequel, for a small enough $\varepsilon > 0$ we define
\begin{align*}
    \forall\, k\geq T(\varepsilon): \quad \| V^* - V^k \|_\infty \leq \varepsilon, \;\; \| V^* - V^{\pi_k} \|_\infty \leq \varepsilon.
\end{align*}
The existence of $T(\varepsilon)$ is guaranteed by the global convergence results of $V^T$ and $V^{\pi_T}$ in Section~\ref{sec:exact-TD-PMD}. Additionally, it will be convenient to define\footnote{One should distinguish $A^k$ here from the advantage function $A^{\pi_k}$, which is defined as $A^{\pi_k}(s,a) = Q^{\pi_k}(s,a) - V^{\pi_k}(s)$.}
\begin{align*}
    A^k(s,a) = Q^k(s,a) - V^*(s), \quad A^*(s,a) = Q^*(s,a) - V^*(s).
\end{align*}

The proof of Theorem~\ref{thm:policy-convergence-of-TD-NPG} can be decomposed into three steps.
\begin{enumerate}
    \item \textbf{Local linear convergence of $Q^{\pi_k}$.} We will show that $Q^{\pi_k}$ converges to $Q^*$ linearly when $k \geq T(\varepsilon)$, i.e., there exists a rate $\rho < 1$ such that
    \begin{align*}
        \forall\, k\geq T(\varepsilon): \quad \| Q^* - Q^{\pi_k} \|_\infty \lesssim \rho^{k-T(\varepsilon)}.
    \end{align*}
    \item \textbf{Local linear convergence of $Q^k$.} We will bound $\| Q^* - Q^k \|_\infty$ by $\| Q^* - Q^{\pi_k} \|_\infty$ to obtain
    \begin{align*}
        \forall\, k\geq T(\varepsilon): \quad \| Q^* - Q^k \|_\infty \lesssim \rho^{k-T(\varepsilon)}.
    \end{align*}
    \item \textbf{Policy convergence of TD-NPG.} Combining the local linear convergence of $Q^k$ and the update formula of TD-NPG, we complete the proof of Theorem~\ref{thm:policy-convergence-of-TD-NPG}.
\end{enumerate}

\vspace{2em}

\paragraph{1. Local linear convergence of $Q^{\pi_k}$}
By characterizing the policy ratio over non-optimal actions, here we  establish the local linear convergence of both $V^{\pi_k}$ and $Q^{\pi_k}$.
\begin{lemma}
    Consider TD-NPG with constant step size $\eta_k = \eta > 0$. There exists a rate $\gamma < \rho < 1$ such that
    \begin{align*}
        \forall\, k\geq T(\varepsilon): \quad \| Q^* - Q^{\pi_k} \|_\infty \leq \| V^* - V^{\pi_k} \|_\infty \leq \frac{|\calS|^2}{1-\gamma} \varepsilon \cdot \rho^{k-T(\varepsilon)}.
    \end{align*}
    \label{lem:local-linear-convergence-of-Q-pi-TD-NPG}
\end{lemma}
\begin{proof}
    We first establish an upper bound for the policy ratio of non-optimal actions. Note that
    \begin{align*}
        \forall\, k \geq T: \quad \| A^* - A^k \|_\infty = \| Q^* - Q^k \|_\infty \leq \| V^* - V^k \|_\infty \leq \varepsilon.
    \end{align*}
    Notice that $A^*(s,a^*) = 0$ for optimal actions $a^*\in\calA^*_s$ and $A^*(s,a^\prime) \leq -\Delta$ for non-optimal actions $a^\prime \not \in \calA^*_s$, we have
    \begin{equation}\label{eq:bounds-for-A-k}
    \begin{aligned}
        \forall\, k \geq T,\,\, \forall \, a^*\in\calA^*_s: \quad -\varepsilon \leq  A^k(s,a^*)  \leq \varepsilon, \\
        \forall\, k \geq T , \,\, \forall a^\prime\not\in\calA^*_s: \quad A^k(s,a^\prime) \leq -\Delta + \varepsilon.
    \end{aligned}
    \end{equation}
    By direct computation,
    \begin{align*}
        \forall\, s\in\calS, \; a^\prime \not\in\calA^*_s, \; t\geq T(\varepsilon): \quad \frac{\pi_{t+1}(a^\prime|s)}{\pi_t(a^\prime|s)} &= \frac{\exp (\eta A^t(s, a^\prime))}{\sum_a \pi_t(a|s) \exp(\eta A^t(s,a))} \\
        &\leq \frac{\exp(-\eta (\Delta - \varepsilon))}{\sum_{a^* \in\calA^*_s} \pi_t(a^*|s) \exp (\eta A^t(s,a^*))} \\
        &\leq \frac{\exp(-\eta (\Delta - \varepsilon))}{(1-b_s^{\pi_t}) \exp(-\eta\varepsilon)} \\
        &\leq \exp(-\eta\Delta + 2\eta\varepsilon)\cdot (1-\varepsilon/\Delta)^{-1} := \rho_0,
    \end{align*}
    where the last inequality is due to $\Delta \cdot b_s^{\pi_t} \leq \| V^* - V^{\pi_t} \|_\infty \leq \varepsilon$ (equation~\eqref{eq:non-optimal-actions-bounded-by-value-error}). As $\varepsilon > 0$ is sufficiently small, there holds $\rho_0 < 1$.

    Let $\rho$ be a constant such that $\max\{ \gamma, \, \rho_0 \} < \rho < 1$. By the performance difference lemma (Lemma~\ref{lem:extended-pdl}), for any measure $\mu \in \Delta(\calS)$ one has
    \begin{align*}
        \forall\, k \geq T(\varepsilon): \quad &\phantom{\quad\,}V^*(\mu) - V^{\pi_k}(\mu) \\[.7em]
        &= \frac{1}{1-\gamma} \sum_s d^k_\mu(s) \sum_{a^\prime \not\in\calA^*_s} \pi_k(a^\prime|s) |A^*(s,a^\prime)| \\
        &\leq \frac{1}{1-\gamma} \ls\| \frac{d^k_\mu}{d^{T(\varepsilon)}_\mu} \rs\|_\infty \ls[ \max_{s\in\calS, \, a^\prime\not\in\calA^*_s} \prod_{t=T(\varepsilon)}^{k-1} \frac{\pi_{t+1}(a|s)}{\pi_t(a|s)} \rs] \sum_s d^{T(\varepsilon)}_\mu(s) \sum_{a^\prime\not\in\calA^*_s} \pi_{T(\varepsilon)}(a^\prime|s) |A^*(s,a^\prime)| \\
        &\leq \ls\| \frac{d^k_\mu}{d^{T(\varepsilon)}_\mu} \rs\|_\infty \cdot \rho^{k-T(\varepsilon)} \cdot (V^*(\mu) - V^{\pi_{T(\varepsilon)}}(\mu)) \\
        &\leq \ls\| \frac{d^k_\mu}{d^{T(\varepsilon)}_\mu} \rs\|_\infty \cdot \rho^{k-T(\varepsilon)} \cdot \ls\| V^* - V^{\pi_{T(\varepsilon)}}\rs\|_\infty \\
        &\leq \frac{1}{(1-\gamma) \cdot \min_{s\in\calS} \, \mu(s)} \cdot \rho^{k-T(\varepsilon)} \cdot \ls\| V^* - V^{\pi_{T(\varepsilon)}}\rs\|_\infty,
    \end{align*}
    Letting $\mu$ be the uniform distribution over the state space $\calS$, one has
    \begin{align*}
        \| Q^* - Q^{\pi_k} \|_\infty &\leq \| V^* - V^{\pi_k} \|_\infty \\
        &\leq |\calS| \cdot (V^*(\mu) - V^{\pi_k}(\mu))\\ &\leq \frac{|\calS|^2}{1-\gamma} \rho ^{k-T(\varepsilon)} \cdot \ls\| V^* - V^{\pi_{T(\varepsilon)}}\rs\|_\infty \\
        &\leq \frac{|\calS|^2}{1-\gamma}\varepsilon \rho^{k-T(\varepsilon)},
    \end{align*}
    where the first inequality follows directly from the definitions of $Q^{\pi}$ and $V^{\pi}$. 
\end{proof}

\vspace{1em}
\paragraph{2. Local linear convergence of $Q^k$} 
\begin{lemma}
    Consider TD-NPG with constant step size $\eta_k = \eta > 0$. There holds
    \begin{align*}
        \forall\, k\geq T(\varepsilon)+1: \quad \| Q^* - Q^k \|_\infty \leq \ls( \ls(\frac{|\calS|^2}{1-\gamma} + \frac{2|\calS|\gamma}{(\rho-\gamma)(1-\gamma)}\rs) \varepsilon + \| Q^{\pi_{T(\varepsilon)}} - Q^{T(\varepsilon)} \|_\infty \rs) \cdot \rho^{k-T(\varepsilon)}.
    \end{align*}
    \label{lem:local-linear-convergence-of-Q-k-TD-NPG}
\end{lemma}
\begin{proof}
    First, recall the update formula of TD-NPG (same as TD-PMD),
    \begin{align*}
        Q^{k+1}(s,a) &= r(s,a) + \gamma \cdot \E_{s^\prime\sim P(\cdot|s,a)} [V^{k+1}(s^\prime)] \\
        &= r(s,a) + \gamma \cdot \E_{s^\prime\sim P(\cdot|s,a), \, a^\prime\sim \pi^{k+1}(\cdot|s^\prime)} [Q^k(s^\prime, a^\prime)] := \calF^{\pi_{k+1}} Q^k(s,a).
    \end{align*}
    Let $\calF^{\pi}$ be the action-value Bellman operator. By definition, it is obvious that
    \begin{align*}
        \forall\, \pi\in\Pi: \quad \calF Q^\pi = Q^\pi, \quad \| \calF ^\pi Q - \calF^\pi Q^\prime \|_\infty \leq \gamma \| Q - Q^\prime \|_\infty,
    \end{align*}
    where $Q$, $Q^\prime \in \mathbb{R}^{|\calS|\times|\calA|}$ are arbitrary vectors. Now by direct computation,
    \begin{align*}
        \forall\, k\geq T(\varepsilon): \quad \| Q^{\pi_k} - Q^k \|_\infty &= \| Q^{\pi_k} - \calF^{\pi_k}Q^{k-1} \|_\infty \\
        &\leq  \gamma \| Q^{\pi_k} - Q^{k-1} \|_\infty \\
        &\leq \gamma \ls[ \| Q^* - Q^{\pi_{k-1}} \|_\infty + \| Q^* - Q^{\pi_k} \|_\infty + \| Q^{\pi_{k-1}} - Q^{k-1} \|_\infty \rs].
    \end{align*}
    By Lemma~\ref{lem:local-linear-convergence-of-Q-pi-TD-NPG}, $\| Q^* - Q^{\pi_k} \|_\infty \leq C_0 \cdot \rho^{k-T(\varepsilon)}$ where $C_0 := |\calS|^2\varepsilon / (1-\gamma)$. Thus
    \begin{align*}
        \forall\, k\geq T(\varepsilon)+1: \quad \| Q^{\pi_k} - Q^k \|_\infty &\leq \gamma \ls[ \| Q^* - Q^{\pi_{k-1}} \|_\infty + \| Q^* - Q^{\pi_k} \|_\infty + \| Q^{\pi_{k-1}} - Q^{k-1} \|_\infty \rs] \\
        &\leq \gamma \ls[ 2C_0 \cdot \rho^{k-T(\varepsilon)-1} + \| Q^{\pi_{k-1}} - Q^{k-1} \|_\infty \rs] \\
        &\leq \cdots \\
        &\leq \sum_{t=1}^{k-T(\varepsilon)} [2C_0 \cdot \rho^{k-T(\varepsilon)-t} \gamma^t] + \gamma^{k-T(\varepsilon)} \| Q^{\pi_{T(\varepsilon)}} - Q^{T(\varepsilon)} \|_\infty \\
        &\leq \frac{2C_0 \gamma}{\rho-\gamma} \cdot \rho^{k-T(\varepsilon)} + \rho^{k-T(\varepsilon)} \| Q^{\pi_{T(\varepsilon)}} - Q^{T(\varepsilon)} \|_\infty \\
        &= C_1 \cdot \rho^{k-T(\varepsilon)},
    \end{align*}
    where the last inequality is due to $\gamma < \rho < 1$ and 
    \begin{align*}
        C_1 := \| Q^{\pi_{T(\varepsilon)}} - Q^{T(\varepsilon)} \|_\infty + \frac{2|\calS|\gamma \varepsilon}{(\rho-\gamma)(1-\gamma)}.
    \end{align*}
    Finally one has
    \begin{align*}
        \forall\, k\geq T(\varepsilon)+1: \quad \| Q^* - Q^k \|_\infty &\leq \| Q^* - Q^{\pi_k} \|_\infty + \| Q^{\pi_k} - Q^k \|_\infty \\
        &\leq (C_0 + C_1) \cdot \rho^{k-T},
    \end{align*}
    which completes the proof.
\end{proof}

\paragraph{3. Policy convergence of TD-NPG}
We are now in the position of proving Theorem~\ref{thm:policy-convergence-of-TD-NPG}. First, notice that the $\pi_k$ of TD-NPG can be expressed as
\begin{align*}
    \forall\, (s,a) \in \calS \times \calA : \quad \pi_{k}(a|s) &= \frac{\pi_{k-1}(a|s) \exp(\eta Q^{k-1}(s,a))}{Z_s^{k-1}} \\
        &= \frac{\pi_{k-2}(a|s) \exp(\eta Q^{k-1}(s,a) + \eta Q^{k-2}(s,a))}{Z_{s}^{k-1} \cdot Z_s^{k-2}} \\
        &= \cdots \\
        &= \frac{\pi_0(a|s) \exp (\eta \sum_{t=0}^{k-1} Q^t(s,a))}{\prod_{t=0}^{k-1} Z_s^t} \\[.8em]
        &\propto \pi_0(a|s) \exp \ls(\eta \sum_{t=0}^{k-1} Q^t(s,a) \rs) \\[.6em]
        &= \frac{\pi_0(a|s) \exp (\eta \sum_{t=0}^{k-1} Q^t(s,a))}{\sum_a \pi_0(a|s) \exp (\eta \sum_{t=0}^{k-1} Q^t(s,a))} \\[.6em]
        &= \frac{\pi_0(a|s) \exp (\eta \sum_{t=0}^{k-1} A^t(s,a))}{\sum_a \pi_0(a|s) \exp (\eta \sum_{t=0}^{k-1} A^t(s,a))} . \numberthis \label{eq:policy-expression-of-TD-NPG}
\end{align*}
In the derivation above, $Z^{k}_s := \sum_a \pi_k(a|s)\exp(\eta Q^k(s,a))$ is the normalization factor. By equation~\eqref{eq:policy-expression-of-TD-NPG}, it suffices to discuss the limit of summation $\sum_{t=0}^\infty A^t(s,a)$. For optimal actions $a^*\in\calA^*_s$, by Lemma~\ref{lem:local-linear-convergence-of-Q-k-TD-NPG},
\begin{align*}
    \forall\, s\in\calS, \, a^* \in \calA^*_s, \; \forall\, k\geq T(\varepsilon)+1: \quad | A^k(s,a^*) | = | Q^k(s,a^*) - Q^*(s,a^*) | \leq C\cdot \rho^{k-T(\varepsilon)}
\end{align*}
where $C > 0$ is some constant. Thus one has
\begin{align*}
    \sum_{k=0}^\infty |A^k(s,a^*)| &= \sum_{k=0}^{T(\varepsilon)} |A^k(s,a^*)| + \sum_{k=T(\varepsilon)+1}^\infty |A^k(s,a^*)| \\
    &\leq \sum_{k=0}^{T(\varepsilon)} |A^k(s,a^*)| + \sum_{k=T(\varepsilon)+1}^\infty C\cdot \rho^{k-T(\varepsilon)} < +\infty,
\end{align*}
which implies that $\sum_{k=0}^\infty A^k(s,a^*) \in (-\infty, +\infty)$. For non-optimal actions $a^\prime \not\in\calA^*_s$, as $Q^k \to Q^*$, there exists a time $T_1$ such that 
\begin{align*}
    \forall\, s\in\calS, \; \forall\, a^\prime \not\in\calA^*_s, \; \forall\, k\geq T_1: \quad |Q^k(s,a^\prime) - Q^*(s,a^\prime)| \leq \Delta / 2.
\end{align*}
In such local region, there holds
\begin{align*}
    A^k(s,a^\prime) &= Q^k(s,a^\prime) - Q^*(s,a^*) \\
    &= Q^k(s,a^\prime) - Q^*(s,a^\prime) + Q^*(s,a^\prime) - Q^*(s,a^*) \\
    &\leq Q^k(s,a^\prime) - Q^*(s,a^\prime) - \Delta \leq -\Delta /2.
\end{align*}
Thus
\begin{align*}
    \sum_{k=0}^\infty A^k(s,a^\prime) = \sum_{k=0}^{T_1-1} A^k(s,a^\prime) + \sum_{k=T_1}^\infty A^k(s,a^\prime) = -\infty.
\end{align*}
In summary, we have 
\begin{align*}
    \forall\, s \in \calS, \, a^* \in \calA^*_s: \quad &\sum_{k=0}^\infty A^k(s,a^*) \in (-\infty, +\infty), \\
    \forall\, s \in \calS, \, a^\prime \not \in \calA^*_s: \quad &\sum_{k=0}^\infty A^k(s,a^\prime) = -\infty.
\end{align*}
It follows that
\begin{align*}
    \forall\, s \in \calS, \, a^* \in \calA^*_s: \quad &\exp\ls(\eta\sum_{k=0}^\infty A^k(s,a^*)\rs) \in (0, +\infty), \\
    \forall\, s \in \calS, \, a^\prime \not \in \calA^*_s: \quad &\exp\ls(\eta\sum_{k=0}^\infty A^k(s,a^\prime)\rs) = 0.
\end{align*}
Finally, for any state $s \in \calS$ and any non-optimal action $a^\prime \not\in \calA^*_s$,
    \begin{align*}
        &\phantom{\quad\,}\lim_{k\to\infty} \pi_k(a^\prime|s) \\
        &= \lim_{k\to\infty} \frac{\pi_0(a^\prime|s) \exp (\eta\sum_{t=0}^{k-1} A^t(s,a^\prime))}{\sum_{a^\prime\not\in\calA^*_s} \pi_0(a^\prime|s) \exp (\eta\sum_{t=0}^{k-1} A^t(s,a^\prime)) + \sum_{a^*\in\calA^*_s} \pi_0(a^*|s) \exp (\eta\sum_{t=0}^{k-1} A^t(s,a^*))} \\
        &= \frac{\pi_0(a^\prime|s) \exp (\eta\sum_{t=0}^{\infty} A^t(s,a^\prime))}{\sum_{a^\prime\not\in\calA^*_s} \pi_0(a^\prime|s) \exp (\eta\sum_{t=0}^{\infty} A^t(s,a^\prime)) + \sum_{a^*\in\calA^*_s} \pi_0(a^*|s) \exp (\eta\sum_{t=0}^{\infty} A^t(s,a^*))} \\[.5em]
        &= 0,
    \end{align*}

    \vspace{1em}
    and for any optimal action $a^* \in\calA^*_s$,
    \begin{align*}
        &\phantom{\quad\,}\lim_{k\to\infty} \pi_k(a^*|s) \\
        &= \frac{\pi_0(a^*|s) \exp (\eta\sum_{t=0}^{\infty} A^t(s,a^*))}{\sum_{a^\prime\not\in\calA^*_s} \pi_0(a^\prime|s) \exp (\eta\sum_{t=0}^{\infty} A^t(s,a^\prime)) + \sum_{a^*\in\calA^*_s} \pi_0(a^*|s) \exp (\eta\sum_{t=0}^{\infty} A^t(s,a^*))} \\
        &= \frac{\pi_0(a^*|s) \exp (\eta\sum_{t=0}^{\infty} A^t(s,a^*))}{\sum_{a^*\in\calA^*_s} \pi_0(a^*|s) \exp (\eta\sum_{t=0}^{\infty} A^t(s,a^*))} \\[.5em]
        &\in (0,1).
    \end{align*}
Therefore, $\pi_k$ converges to some optimal policy $\pi^*$.

%%%%%%%%%%%%%%%%%

%%%%%%%%%%%%%%%%
%%%%%%%%%%%%%%%%

\section{Proofs for results in Section~\ref{sec:sample-complexity}}

\subsection{Proof of Theorem~\ref{thm:linear-convergence-of-exact-TD-PMD-inexact}}
\label{sec:pf:thm:linear-convergence-of-exact-TD-PMD-inexact}
The overall proof procedure is similar to that for Theorem~\ref{thm:linear-convergence-of-exact-TD-PMD}. To this end, we first present lemmas similar to those in Section~\ref{sec:linear-convergence}, see Lemma~\ref{lem:inexact-policy-error-bounded-by-value-error} and Lemma~\ref{lem:TD-PMD-three-point-linear-inexact} below.

\begin{lemma} 
    Consider the inexact TD-PMD. There holds
    {
        \begin{align*}
        \ls\| V^*-V^{\pi_T}\rs\| _{\infty} \le \frac{1}{1-\gamma}\left( \ls\| V^*-V^T \rs\| _{\infty}+\ls\| V^*-V^{T-1}\rs\| _{\infty} \right) +\frac{\delta}{1-\gamma}.
    \end{align*}
    }
    \label{lem:inexact-policy-error-bounded-by-value-error}
\end{lemma}
\begin{proof}
    Similar to the proof of Lemma~\ref{lem:error-bounds}, one has
    \begin{align*}
        \ls\| V^{\pi_T} - V^T \rs\|_\infty &= \ls\| \calT^{\pi_T} V^{\pi_T} - \widehat{\calT}^{\pi_T} V^{T-1} \rs\|_\infty \\
        &\leq \ls\| \calT^{\pi_T}V^{\pi_T} - \calT^{\pi_T}V^{T-1} \rs\|_\infty + \ls\| \calT^{\pi_T}V^{T-1} - \widehat{\calT}^{\pi_T}V^{T-1} \rs\|_\infty \\
        &\leq \gamma \ls\| V^{\pi_T} - V^{T-1} \rs\|_\infty + \delta \\
        &\leq \gamma \ls\| V^{\pi_T} - V^T \rs\|_\infty + \gamma \ls\| V^T - V^{T-1} \rs\|_\infty + \delta.
    \end{align*}
    It follows that
    \begin{align*}
        \ls\| V^{\pi_T} - V^T \rs\|_\infty &\leq \frac{\gamma}{1-\gamma} \ls\| V^T - V^{T-1} \rs\|_\infty + \frac{\delta}{1-\gamma} \\
        &\leq \frac{\gamma}{1-\gamma} \ls( \ls\| V^* - V^{T} \rs\|_\infty + \ls\| V^* - V^{T-1} \rs\|_\infty \rs) + \frac{\delta}{1-\gamma}.
    \end{align*}
    Therefore,
    \begin{align*}
        \ls\| V^* - V^{\pi_T} \rs\|_\infty &\leq \ls\| V^* - V^{T} \rs\|_\infty + \ls\| V^{\pi_T} - V^T \rs\|_\infty \\
        &\leq \ls\| V^* - V^T \rs\|_\infty + \frac{\gamma}{1-\gamma} \ls( \ls\| V^* - V^{T} \rs\|_\infty + \ls\| V^* - V^{T-1} \rs\|_\infty \rs) + \frac{\delta}{1-\gamma} \\
        &= \frac{1}{1-\gamma} \ls\| V^* - V^{T} \rs\|_\infty + \frac{\gamma}{1-\gamma} \ls\| V^* - V^{T-1} \rs\|_\infty + \frac{\delta}{1-\gamma} \\
        &\leq \frac{1}{1-\gamma} \ls( \ls\| V^* - V^T \rs\|_\infty + \ls\| V^* - V^{T-1} \rs\|_\infty \rs) + \frac{\delta}{1-\gamma},
    \end{align*}
    which completes the proof.
\end{proof}

\begin{lemma}
    Consider the inexact TD-PMD  with adaptive step sizes $\{ \eta_k \}$. For $k=0, 1, ..., T-1$,  there holds
    \begin{align}
       [\widehat{\calT}^{\pi_{k+1}} V^k - \calT V^k](s) \geq -\frac{1}{\eta
         _k}\ls\| {D}_{\pi_k}^{\widehat{\pi}_k} \rs\|_\infty - 3\delta, \;\; \|{D}_{\pi_k}^{\widehat{\pi}_k}\|_\infty = \max_{s} D_{\pi_k}^{\widehat{\pi}_k}(s),
        \label{eq:TD-PMD-three-point-linear-inexact}
    \end{align}
    where $\widehat{\pi}_k$ is any policy that satisfies $\langle\widehat{\pi}_k(\cdot|s),\widehat{Q}^k(s,\cdot)\rangle = \max_a\widehat{Q}^{k}(s,a),\quad\forall s$.
    \label{lem:TD-PMD-three-point-linear-inexact}
\end{lemma}
\begin{proof}
    By Lemma~\ref{lem:TD-PMD-three-point-original}, for the sample-based TD-PMD, there holds
    \begin{align*}
        \forall\, s\in\calS,\; p\in\Delta(\calS): \;\; \eta_k \langle \pi_{k+1}(\cdot|s) - p, \, \widehat{Q}^k(s,\cdot) \rangle \geq D^{\pi_{k+1}}_{\pi_k}(s) + D^p_{\pi_{k+1}}(s) - D^p_{\pi_k}(s).
    \end{align*}
    Similar to the proof of Lemma~\ref{lem:TD-PMD-three-point-linear}, setting $p=\widehat{\pi}_k $ yields
    \begin{align*}
            \langle \pi_{k+1}(\cdot|s) - \widehat{\pi}_k(\cdot|s), \, \widehat{Q}^k(s,\cdot) \rangle &\geq \frac{1}{\eta_k}[D^{\pi_{k+1}}_{\pi_k}(s) + D^{\widehat{\pi}_k}_{\pi_{k+1}}(s) - D^{\widehat{\pi}_k}_{\pi_k}(s)] \\
            &\geq -\frac{1}{\eta_k} D^{\widehat{\pi}_k}_{\pi_k}(s) \\
            &\geq -\frac{1}{\eta_k} \| {D}_{\pi_k}^{\widehat{\pi}_k} \|_\infty.
    \end{align*}
    By the definition of $\widehat{\pi}_k$, one has
    \begin{align*}
        \langle \pi_{k+1}(\cdot|s) - \widehat{\pi}_k(\cdot|s), \, \widehat{Q}^k(s,\cdot) \rangle = \E_{a\sim\pi_{k+1}(\cdot|s)}[\widehat{Q}^k(s,a)] - \max_a \, \{\widehat{Q}^k(s,a)\}.
    \end{align*}
    
    Thus the direct computation gives
    \begin{align*}
        &\phantom{\,\,\,=}[\widehat{\calT}^{\pi_{k+1}} V^k - \calT V^k](s) \\
        &= [\calT^{\pi_{k+1}} V^k - \calT V^k](s) - [\calT^{\pi_{k+1}}V^k - \widehat{\calT}^{\pi_{k+1}}V^k](s) \\
        &\geq [\calT^{\pi_{k+1}}V^k - \calT V^k](s) - \delta \\
        &= \E_{a\sim\pi_{k+1}(\cdot|s)}[Q^k(s,a)] - \max_{a} \, \{ Q^k(s,a) \} - \delta \\
        &= \E_{a\sim\pi_{k+1}(\cdot|s)}[\widehat{Q}^k(s,a)] - \max_a \, \{ \widehat{Q}^k(s,a) \} + \ls[ \E[Q^k(s,a) - \widehat{Q}^k(s,a)] - \max_a \, \{ Q^k(s,a) - \widehat{Q}^k(s,a) \} \rs] - \delta \\
       % &\geq \E_{a\sim\pi_{k+1}(\cdot|s)}[\widehat{Q}^k(s,a)] - \max_a \, \{ \widehat{Q}^k(s,a) \} - 2\delta - \delta \\
        &\geq \langle \pi_{k+1}(\cdot|s) - \widehat{\pi}(\cdot|s), \, \widehat{Q}^k(s,\cdot) \rangle - 3\delta \\
        &\geq -\frac{1}{\eta_k} \| {D}_{\pi_k}^{\widehat{\pi}_k} \|_\infty - 3\delta,
    \end{align*}
    and the proof is now complete.
\end{proof}

\paragraph{Proof of Theorem~\ref{thm:linear-convergence-of-exact-TD-PMD-inexact}} The overall proof is similar to that for Theorem~\ref{thm:linear-convergence-of-exact-TD-PMD}. On the one hand, since $\widehat{\calT}^{\pi_{k+1}}V^k = V^{k+1}$ and $\calT V^* = V^*$, the application of Lemma~\ref{lem:TD-PMD-three-point-linear-inexact} yields that 
\begin{align*}
    [V^* - V^{k+1}](s) \leq [\calT V^* - \calT V^k](s) + \frac{1}{\eta_k} \| {D}_{\pi_k}^{\widehat{\pi}_k} \|_\infty + 3\delta.
\end{align*}
On the other hand, 
\begin{align*}
    [V^* - V^{k+1}](s) &= [V^* - \widehat{\calT}^{\pi_{k+1}}V^k](s) \\
    &= [V^* - \calT^{\pi_{k+1}}V^k](s) + [\calT^{\pi_{k+1}}V^k - \widehat{\calT}^{\pi_{k+1}}V^k](s) \\
    &\geq [V^* - \calT^{\pi_{k+1}}V^k](s) - \delta \\
    &\geq [V^* - \calT V^k](s) - \delta \\
    &= [\calT V^* - \calT V^k](s) - \delta,
\end{align*}
where the second inequality follows from the definition of optimal Bellman operator. Therefore,
\begin{align*}
    \forall\, s\in\calS: \;\; [\calT V^* - \calT V^k](s) - \delta \leq [V^* - V^{k+1}](s) \leq [\calT V^* - \calT V^k](s) + 3\delta + \frac{1}{\eta_k} \| D_{\pi_k}^{\widehat{\pi}_k} \|_\infty,
\end{align*}
which implies that
\begin{align*}
    \ls\| V^* - V^{k+1} \rs\|_\infty &\leq \ls\| \calT V^* - \calT V^k \rs\|_\infty + 3\delta + \frac{1}{\eta_k} \| {D}_{\pi_k}^{\widehat{\pi}_k} \|_\infty \\
    &\leq \gamma \ls\| V^* - V^k \rs\|_\infty + 3\delta + \frac{1}{\eta_k} \| {D}_{\pi_k}^{\widehat{\pi}_k} \|_\infty.
\end{align*}
Consequently,
\begin{align*}
    \ls\| V^* - V^{T} \rs\|_\infty &\leq \gamma \ls\| V^* - V^{T-1} \rs\|_\infty + \frac{1}{\eta_{T-1}} \ls\| {D}_{\pi_{T-1}}^{\widehat{\pi}_{T-1}} \rs\|_\infty + 3\delta \\
    &\leq \gamma \ls( \gamma \ls\| V^* - V^{T-2} \rs\|_\infty + \frac{1}{\eta_{T-2}} \| {D}_{\pi_{T-2}}^{\widehat{\pi}_{T-2}} \|_\infty +3\delta \rs) + \frac{1}{\eta_{T-1}} \| {D}_{\pi_{T-1}}^{\widehat{\pi}_{T-1}} \|_\infty + 3\delta \\
    &= \gamma^2 \ls\| V^* - V^{T-2} \rs\|_\infty + \ls( \frac{\gamma}{\eta_{T-2}} \| {D}_{\pi_{T-2}}^{\widehat{\pi}_{T-2}} \|_\infty + \frac{1}{\eta_{T-1}} \| {D}_{\pi_{T-1}}^{\widehat{\pi}_{T-1}} \|_\infty \rs) + (\gamma + 1) \cdot 3\delta\\
    &\leq \cdots \\
    &\leq \gamma^T \ls\| V^* - V^0 \rs\|_\infty + \sum_{k=0}^{T-1} \frac{\| {D}_{\pi_k}^{\widehat{\pi}_k} \|_\infty}{\eta_k} \cdot \gamma ^{T-k-1} + \sum_{k=0}^{T-1}3\delta \cdot \gamma^{T-k-1}.
\end{align*}
Under the condition that $\eta_k \geq {\ls\| {D}_{\pi_k}^{\widehat{\pi}_k} \rs\|_\infty}\big/{(c \cdot \gamma^{2k+1})}$, one has
\begin{align*}
    \ls\| V^* - V^T \rs\|_\infty &\leq \gamma^T \ls\| V^* - V^0 \rs\| + \sum_{k=0}^{T-1} c\cdot \gamma^{T+k} + \sum_{k=0}^{T-1} 3\delta\cdot \gamma^{T-k-1} \\
    &\leq \gamma^T \ls\| V^* - V^0 \rs\|_\infty + \frac{c\cdot \gamma^T}{1-\gamma} + \frac{3\delta}{1-\gamma}.
\end{align*}
For the bound of $\ls\| V^* - V^{\pi_T} \rs\|_\infty$, it follows from Lemma~\ref{lem:inexact-policy-error-bounded-by-value-error} that
\begin{align*}
    \ls\| V^* - V^{\pi_T} \rs\|_\infty &\leq \frac{1}{1-\gamma} \ls( \ls\| V^* - V^T \rs\|_\infty + \ls\| V^* - V^{T-1} \rs\|_\infty \rs) + \frac{\delta}{1-\gamma} \\
    &\leq \frac{2}{1-\gamma} \ls(  \gamma^{T-1} \ls\| V^* - V^0 \rs\|_\infty + \frac{c \cdot \gamma^{T-1}}{1-\gamma} + \frac{3\delta}{1-\gamma}\rs) + \frac{\delta}{1-\gamma} \\
    &\leq \frac{2\gamma^{T-1}}{1-\gamma} \ls[ \ls\| V^* - V^0 \rs\|_\infty + \frac{c}{1-\gamma} \rs] + \frac{7\delta}{(1-\gamma)^2}.
\end{align*}

\subsection{Proof of Lemma~\ref{lem:sample-complexity-for-error-level-delta}}
\label{sec:pf:lem:sample-complexity-for-error-level-delta}
We first show that under if $\| V_0 \|_\infty \leq 1/(1-\gamma)$, then there holds 
\begin{align}
    \forall\, k\geq 1: \quad \| V^k \|_\infty \leq \frac{1}{1-\gamma}. 
    \label{eq:bounded-values-for-sample-based-TD-PMD}
\end{align}
For any $k\in\mathbb{N}^+$, if $\| V^{k} \|_\infty \leq 1/(1-\gamma)$, then
\begin{align*}
    \forall\, s\in\calS: \;\; \ls|V^{k+1}(s)\rs| &= \ls|\widehat{\calT}^{\pi_{k+1}} V^k(s)\rs| \\
    &\leq \frac{1}{M_V} \sum_{i=1}^{M_V}\ls|r(s,a_{(i)}) + \gamma V^k(s^\prime_{(i)})\rs| \\
    &\leq \frac{1}{M_V} \sum_{i=1}^{M_V} \ls|r(s,a_{(i)})\rs| + \gamma \ls| V^k(s^\prime_{(i)}) \rs| \\
    &\leq \frac{1}{M_V} \sum_{i=1}^{M_V} \ls[1 + \frac{\gamma}{1-\gamma}\rs] \\
    &= \frac{1}{1-\gamma},
\end{align*}
which implies that $\| V^{k+1} \|_\infty \leq 1/(1-\gamma)$. Thus, equation~\eqref{eq:bounded-values-for-sample-based-TD-PMD} holds by induction. 

As $\{ s^\prime_{(i)} \} \stackrel{\mathrm{i.i.d.}}{\sim} P(\cdot|s,a)$, $\E[r(s,a) + \gamma V^k(s^\prime_{(i)})] = Q^k(s,a)$, and $\ls|r(s,a) + \gamma V^k(s^\prime_{(i)}) \rs| \leq 1/(1-\gamma)$ is bounded by $\| V^k \|_\infty \leq 1/(1-\gamma)$, we can apply the Hoeffding's inequality (Lemma~\ref{lem:hoeffding}) to obtain 
\begin{align*}
    \forall\, s\in\calS, \, a\in\calA: \;\;\;\; \Pr\ls[ \ls| \widehat{Q}^k(s,a) - Q^k(s,a) \rs| > t \rs] &= \Pr \ls[ \ls| \frac{1}{M_Q} \sum_{i=1}^{M_Q} \ls[ r(s,a) + \gamma V^k(s^\prime_{(i)}) \rs] - Q^k(s,a) \rs| > t \rs]  \\
    &\leq 2 \exp \ls( -{2M_Qt^2(1-\gamma)^2}\rs).
\end{align*}
Let $[T] = \{ 0, 1, ..., T-1 \}$. Taking a union bound yields
\begin{align*}
    \Pr\ls[\, \forall\, k\in[T]: \;\; \ls\| \widehat{Q}^k - Q^k \rs\|_\infty > t \rs] &= \Pr \ls[ \forall\, k\in[T], \, s\in\calS, \, a\in\calA: \;\; \ls| \widehat{Q}^k(s,a) - Q^k(s,a) \rs| > t \rs] \\
    &\leq \sum_{k\in[T], s\in\calS, a\in\calA} \Pr \ls[ \ls| \widehat{Q}^k(s,a) - Q^k(s,a) \rs| > t \rs] \\
    &\leq 2T|\calS||\calA| \exp \ls( 2M_Q t^2(1-\gamma)^2 \rs).
\end{align*}
For $V^{k+1} = \widehat{\calT}^{\pi_{k+1}}V^k$, noting that $\E[r(s,a_{(i)}) + \gamma V^k(s^\prime_{(i)})] = \calT^{\pi_{k+1}} V^k(s)$ and $\ls|r(s,a_{(i)}) + \gamma V^k(s^\prime_{(i)})\rs| \leq 1/(1-\gamma)$ is bounded, applying the Hoeffding's inequality again gives
\begin{align*}
    \forall\, s\in\calS: \quad \Pr\ls[ \ls| V^{k+1}(s) - \calT^{\pi_{k+1}}V^k(s) \rs| > t \rs] &= \Pr \ls[ \ls| \frac{1}{M_V} \sum_{i=1}^{M_V} \ls[ r(s,a_{(i)}) + \gamma V^k(s^\prime_{(i)}) \rs] - \calT^{\pi_{k+1}}V^k(s) \rs| > t \rs] \\
    &\leq 2\exp \ls( -2M_V t^2(1-\gamma)^2 \rs).
\end{align*}
Thus
\begin{align*}
    \Pr \ls[ \, \forall \, k\in[T]: \;\; \ls\| V^{k+1} - \calT^{\pi_{k+1}}V^k \rs\|_\infty > t \rs] &= \Pr\ls[ \, \forall\, k\in[T], \, s\in\calS: \;\; \ls| V^{k+1}(s) - \calT^{\pi_{k+1}}V^k(s) \rs| > t \rs] \\
    &\leq 2T |\calS| \exp \ls( -2M_V t^2(1-\gamma)^2 \rs).
\end{align*}
Combining these two results together shows that  equations~\eqref{Q esitmation error} and ~\eqref{V esitmation error} hold with probability at least 
\begin{align*}
    1 - 2T|\calS||\calA| \exp \ls( 2M_Q t^2(1-\gamma)^2 \rs) - 2T |\calS| \exp \ls( -2M_V t^2(1-\gamma)^2 \rs),
\end{align*}
which exceeds $1-\alpha$ provided
\begin{align*}
    M_Q \geq \frac{1}{2(1-\gamma)^2 \delta^2} \log \frac{4T|\calS||\calA|}{\alpha}, \quad M_V \geq \frac{1}{2(1-\gamma)^2 \delta^2} \log \frac{4T|\calS|}{\alpha}.
\end{align*}
%the probability above is larger than $1-\alpha$.

\subsection{Proof of Theorem~\ref{thm:sample-complexity}}
\label{sec:pf:thm:sample-complexity}

First by Lemma~\ref{lem:sample-complexity-for-error-level-delta},  we know that  equations~\eqref{Q esitmation error} and ~\eqref{V esitmation error} hold with probability at least $1-\alpha$. It follows from Theorem~\ref{thm:linear-convergence-of-exact-TD-PMD-inexact} that
\begin{align*}
    \ls\| V^* - V^{\pi_T} \rs\| &\leq \frac{2\gamma^{T-1}}{1-\gamma} \ls[ \ls\| V^* - V^0 \rs\|_\infty + \frac{c}{1-\gamma} \rs] + \frac{7\delta}{(1-\gamma)^2} \\
    &\leq \frac{2\gamma^{T-1}}{1-\gamma} \ls[ \| V^* \|_\infty + \| V^0 \|_\infty + \frac{c}{1-\gamma} \rs] + \frac{7\delta}{(1-\gamma)^2} \\
    &\leq \frac{2\gamma^{T-1}}{1-\gamma} \ls[ \frac{2}{1-\gamma} + \frac{c}{1-\gamma} \rs] + \frac{7\delta}{(1-\gamma)^2} \\
    &= \frac{1}{(1-\gamma)^2} [2(2+c)\gamma^{T-1} + 7\delta],
    % \label{eq:linear-convergence-of-output-policy-inexact}
\end{align*}
which completes the proof.

%%%%%%%%%%%%%%%%%%%%%
\section{Q-TD-PMD}
\label{sec:Q-TD-PMD}
In this section, we study Q-TD-PMD (see Algorithm \ref{alg:Q-TD-PMD}), a variant of TD-PMD (Algorithm \ref{alg:TD-PMD} ) that only maintains a Q function during the iteration. Using similar arguments with some modified proof details, we can extend the theoretical results for TD-PMD to  Q-TD-PMD. In this section, we present the detailed convergence results of Q-TD-PMD. Before proceeding, we first introduce the Bellman operators for Q value functions and the Q function variant of Lemma \ref{lem:extended-pdl}.

\begin{algorithm}[ht!]
   \caption{Q-TD-PMD}
   \label{alg:Q-TD-PMD}
\begin{algorithmic}
    {
   \STATE {\bfseries Input:} Initial state-action value estimation $Q^0$, initial policy $\pi_0$, iteration number $T$, step size $\{ \eta_k \}$.
   \FOR{$k=0$ {\bfseries to} $T-1$}
   \vspace{0.2cm}
   \STATE {\bfseries (Policy improvement)} 
    Update the policy by
   \begin{align*}
       \pi_{k+1}(\cdot|s) \!=\! \underset{p\in\Delta(\calA)}{\arg\max} \ls\{ \eta_k \langle p, \, Q^{k}(s,\cdot) \rangle \!-\! D^p_{\pi_k}(s) \rs\}.\numberthis\label{eq:alg2pi}
   \end{align*}
   \STATE {\bfseries (TD evaluation)} Update the state value estimation by
   \begin{align*}
       Q^{k+1} = \mathcal{F}^{\pi_{k+1}} Q^k.\numberthis\label{eq:alg2Q}
   \end{align*}
   \ENDFOR
   \STATE {\bfseries Output:} Last iterate policy $\pi_T$, last iterate state value estimation $Q^T$.
   }
\end{algorithmic}
\end{algorithm}

\paragraph{Bellman operators for Q functions.}  For arbitrary $Q\in \mathbb{R}^{\left|\mathcal{S}\right|\times\left|\mathcal{A}\right|}$ and policy $\pi$,  the Bellman Operator $\mathcal{F}^\pi$ is defined as 
    \begin{align*}
    \forall(s,a)\in \mathcal{S}\times \mathcal{A}:\quad \mathcal{F} ^{\pi}Q\left( s,a \right) =r\left( s,a \right) +\gamma \mathbb{E} _{s^{\prime}\sim P\left( \cdot |s,a \right)}\mathbb{E} _{a^\prime\sim \pi \left( \cdot |s^{\prime} \right)}\left[ Q\left( s^{\prime},a^{\prime} \right) \right], 
\end{align*}
while the Bellman optimality Operator $\mathcal{F}$ is defined as 
    \begin{align*}
\forall (s,a)\in \mathcal{S} \times \mathcal{A} : \quad \mathcal{F} Q\left( s,a \right) =r\left( s,a \right) +\gamma \mathbb{E} _{s^{\prime}\sim P\left( \cdot |s,a \right)}\left[ \underset{a^{\prime}\in \mathcal{A}}{\max}\,\,Q\left( s^{\prime},a^{\prime} \right) \right].
\end{align*}

\begin{lemma}[General performance difference lemma for Q function] \label{lem:extended-pdl-q}
    For any policy $\pi \in \Pi$, vector $Q \in \mathbb{R}^{|\calS\times\calA|}$ and $\rho \in \Delta(\mathcal{S}\times\mathcal{A})$, there holds 
    \[
        Q^\pi(\rho) - Q(\rho) = \frac{1}{1-\gamma} [\calF^\pi Q - Q](\nu^\pi_\rho),
    \]
where $[\calF^\pi Q - Q](\nu^\pi_\rho)=\sum\limits_{s,a} \nu^\pi_\rho(s,a)[\calF^\pi Q(s,a) - Q(s,a)]$ and 
\begin{align*}
\nu _{\rho}^{\pi}\left( s,a \right) :=(1-\gamma)\cdot\mathbb{E} \left[ \sum_{t=0}^{\infty}{\gamma ^t\mathbb{I} \left\{ s_t=s,a_t=a \right\}}|(s_0,a_0)\sim \rho ,\pi \right].
\end{align*}
\end{lemma}

\subsection{Exact Q-TD-PMD}
 We first extend the sublinear convergence of exact TD-PMD into exact Q-TD-PMD (where the  $
\mathcal{F} ^{\pi _{k+1}}Q^k$ is computed exactly in Algorithm \ref{alg:Q-TD-PMD}). For the policy improvement step of the $k$-th iteration (equations~\eqref{eq:alg2pi}), by Lemma~\ref{lem:original-three-point}, it is easy to show that following lemma holds.
\begin{lemma} 
\label{lem:Q-TD-PMD-three-point-original}
For the exact Q-TD-PMD (Algorithm \ref{alg:Q-TD-PMD}), there holds \begin{align*}
    \forall\, s\in\calS,\; p\in\Delta(\calA): \;\; \eta_k \langle \pi_{k+1}(\cdot|s) - p, \, Q^k(s,\cdot) \rangle \geq D^{\pi_{k+1}}_{\pi_k}(s) + D^p_{\pi_{k+1}}(s) - D^p_{\pi_k}(s).
\end{align*}
\end{lemma}
Then by using the same arguments  as in the proof of Lemma \ref{lem:monotonicity-property} and leveraging Lemma \ref{lem:Q-TD-PMD-three-point-original}, one can establish the following lemma. The details of the proof are omitted.

\begin{lemma}
\label{lem:monotonicity-property-Q}
Consider Q-TD-PMD with constant step size $\eta_k=\eta>0$. Assume the initialization satisfies $\calF^{\pi_0} Q^0 \geq Q^0$.    Then, for $k=0,1,...,T-1$, there holds    
    \begin{align}
        Q^* \geq Q^{\pi_{k+1}} \geq Q^{k+1} = \calF^{\pi_{k+1}}Q^k \geq \calF^{\pi_k} Q^k \geq Q^k.
        \label{eq:monotonicity-property-Q}
    \end{align}
\end{lemma} 
By combining Lemma \ref{lem:extended-pdl-q}, Lemma \ref{lem:Q-TD-PMD-three-point-original} and  Lemma \ref{lem:monotonicity-property-Q}, we can extend the results in Theorem \ref{thm:sublinear-convergence-value-spec-init} to Q-TD-PMD as follows. 
\begin{theorem}
\label{thm:sublinear-convergence-value-spec-init-Q}    
 Consider Q-TD-PMD with constant step size $\eta_k=\eta>0$. Assume the initialization satisfies $\calF^{\pi_0} Q^0 \geq Q^0$. Then, one has
    \begin{align*}
\left\| Q^*-Q^{\pi _T} \right\| _{\infty}\le \left\| Q^*-Q^T \right\| _{\infty}\le \frac{1}{T+1}\left( \frac{1}{\left( 1-\gamma \right) ^2}+\frac{\left\| Q^0 \right\| _{\infty}}{\left( 1-\gamma \right)}+\frac{\gamma \left\| \hat{D}_{\pi _0}^{\pi ^*} \right\| _{\infty}}{\eta \left( 1-\gamma \right)} \right)
    \end{align*}
    where $\hat{D}_{\pi _0}^{\pi ^*}\left( s,a \right) :=\mathbb{E} _{s^\prime\sim P\left( \cdot |s,a \right)}\left[ D_{\pi _0}^{\pi ^*}\left( s^{\prime} \right) \right]$
\end{theorem}
\begin{proof} A direct computation yields that 
\begin{align*}
Q^{k+1}\left( s,a \right) &=\mathcal{F} ^{\pi _{k+1}}Q^k\left( s,a \right) 
\\
\,\,&=r\left( s,a \right) +\gamma \mathbb{E} _{s^{\prime}\sim P\left( \cdot |s,a \right)}\left[ \langle \pi _{k+1}(\cdot |s^{\prime}),\,Q^k(s^\prime,\cdot )\rangle \right] 
\\
\,\,&\overset{\left( a \right)}{\ge}r\left( s,a \right) +\gamma \mathbb{E} _{s^{\prime}\sim P\left( \cdot |s,a \right)}\left[ \langle \pi ^*(\cdot |s^{\prime}),\,Q^k(s^\prime,\cdot )\rangle +\frac{1}{\eta}\left( D_{\pi _k}^{\pi _{k+1}}(s^{\prime})+D_{\pi _{k+1}}^{\pi ^*}(s^{\prime})-D_{\pi _k}^{\pi ^*}(s^{\prime}) \right) \right] 
\\
\,\,&\ge \mathcal{F} ^{\pi ^*}Q^k\left( s,a \right) +\frac{\gamma}{\eta}\mathbb{E} _{s^{\prime}\sim P\left( \cdot |s,a \right)}\left[ D_{\pi _{k+1}}^{\pi ^*}(s^{\prime})-D_{\pi _k}^{\pi ^*}(s^{\prime}) \right] 
\\
\,\,&=\mathcal{F} ^{\pi ^*}Q^k\left( s,a \right) +\frac{\gamma}{\eta}\left( \hat{D}_{\pi _{k+1}}^{\pi ^*}\left( s,a \right) -\hat{D}_{\pi _k}^{\pi ^*}\left( s,a \right) \right) ,
\end{align*}
where (a) is due to Lemma \ref{lem:Q-TD-PMD-three-point-original}. Thus for arbitrary $\rho\in \Delta(\mathcal{S}\times\mathcal{A})$,
\begin{align*}
\forall k\in \mathbb{N} : \frac{1}{1-\gamma}\left[ Q^{k+1}-Q^k \right] \left( \nu _{\rho}^{*} \right) &=\frac{1}{1-\gamma}\left[ \mathcal{F} ^{\pi _{k+1}}Q^k-Q^k \right] \left( \nu _{\rho}^{*} \right) 
\\
\,\,                             &\ge \frac{1}{1-\gamma}\left[ \mathcal{F} ^{\pi ^*}Q^k -Q^k+\frac{\gamma}{\eta} \left( \hat{D}_{\pi _{k+1}}^{\pi ^*}-\hat{D}_{\pi _k}^{\pi ^*} \right) \right] \left( \nu _{\rho}^{*} \right) 
\\
\,\,                             &\overset{\left( a \right)}{=}Q^*\left( \rho \right) -Q^k\left( \rho \right) +\frac{\gamma}{\eta(1-\gamma)}\left[ \hat{D}_{\pi _{k+1}}^{\pi ^*}-\hat{D}_{\pi _k}^{\pi ^*} \right] \left( \nu _{\rho}^{*} \right),
\end{align*}
where (a) is due Lemma \ref{lem:extended-pdl-q}. Hence, by Lemma \ref{lem:monotonicity-property-Q},
\begin{align*}
Q^*\left( \rho \right) -Q^T\left( \rho \right) &\le \frac{1}{T+1}\sum_{k=0}^T{\left( Q^*\left( \rho \right) -Q^k\left( \rho \right) \right)}
\\
\,\,                   &\le \frac{1}{T+1}\sum_{k=0}^T{\left( \frac{1}{1-\gamma}\left[ Q^{k+1}-Q^k \right] \left( \nu _{\rho}^{*} \right) +\frac{\gamma}{\eta(1-\gamma)}\left[ \hat{D}_{\pi _k}^{\pi ^*}-\hat{D}_{\pi _{k+1}}^{\pi ^*} \right] \left( \nu _{\rho}^{*} \right) \right)}
\\
\,\,                    &=\frac{1}{T+1}\left( \frac{1}{1-\gamma}\left[ Q^{T+1}-Q^0 \right] \left( \nu _{\rho}^{*} \right) +\frac{\gamma}{\eta(1-\gamma)}\left[ \hat{D}_{\pi _0}^{\pi ^*}-\hat{D}_{\pi _{T+1}}^{\pi ^*} \right] \left( \nu _{\rho}^{*} \right) \right) 
\\
\,\,                    &\leq \frac{1}{T+1}\left( \frac{1}{\left( 1-\gamma \right) ^2}+\frac{\left\| Q^0 \right\| _{\infty}}{\left( 1-\gamma \right)}+\frac{\gamma}{\eta(1-\gamma)}\left\| \hat{D}_{\pi _0}^{\pi ^*} \right\| _{\infty} \right).
\end{align*}
\end{proof}
For general initialization that does not satisfy $\calF^{\pi_0} Q^0 \geq Q^0$, one can also establish the convergency by constructing the auxiliary sequence  as in  Section \ref{sec:Sublinear Convergence for General Initialization}. Here we present the sublinear convergence for general initialization without proof.

\begin{theorem}[Sublinear convergence]\label{thm:sublinear-convergence-of-exact-Q-TD-PMD-formal}
  Consider Q-TD-PMD with constant step size $\eta_k=\eta>0$. For any $\{Q^0,\pi_0\}$, one has
    \begin{align*}
\left\| Q^*-Q^{\pi _T} \right\| _{\infty}&\le \frac{1}{T+1}\left( \frac{1}{\left( 1-\gamma \right) ^2}+\frac{\left\| Q^0 \right\| _{\infty}+\hat{\kappa}_0}{\left( 1-\gamma \right)}+\frac{\gamma \left\| \hat{D}_{\pi _0}^{\pi ^*} \right\| _{\infty}}{\eta \left( 1-\gamma \right)} \right)  \numberthis\label{eq:sublinear-convergence-of-output-policy-Q-TD-PMD} 
        \\
\left\| Q^*-Q^T \right\| _{\infty}&\le \frac{1}{T+1}\left( \frac{1}{\left( 1-\gamma \right) ^2}+\frac{\left\| Q^0 \right\| _{\infty}+\hat{\kappa}_0}{\left( 1-\gamma \right)}+\frac{\gamma \left\| \hat{D}_{\pi _0}^{\pi ^*} \right\| _{\infty}}{\eta \left( 1-\gamma \right)} \right) +\gamma ^T\hat{\kappa}_0.
        \numberthis \label{eq:sublinear-convergence-of-output-value-Q-TD-PMD}
    \end{align*}
where 
\[\hat{\kappa}_0 :=  \max \; \ls\{0, \;\; \frac{1}{1-\gamma} \max_{(s,a)\in\calS\times\calA} \, [Q^0 - \calF^{\pi_0}Q^0](s,a) \rs\} , \]
and $\mathbf{1}$ is a vector whose entries are all one.
\end{theorem}

Next we establish the $\gamma$-rate linear convergence of  exact Q-TD-PMD. Note that $Q^\pi$ is the fixed point of $\calF^\pi$ and $\calF^\pi$ is $\gamma$-contraction operator. Thus using the same arguments as in the proof of Lemma \ref{lem:error-bounds}, we can establish the following error decomposition lemma. The proof is omitted for simplicity. 
\begin{lemma}
    \label{lem:error-bounds-Q}
    For the exact Q-TD-PMD, there holds 
    \begin{align}
        \ls\| Q^{*} - Q^{\pi_T} \rs\|_\infty \leq \frac{1}{1-\gamma} \ls( \ls\| Q^* - Q^T \rs\|_\infty + \ls\| Q^* - Q^{T-1} \rs\|_\infty\rs).
    \end{align}
\end{lemma}
The following lemma is central to the analysis of Q-TD-PMD.
\begin{lemma}
    \label{lem:Q-TD-PMD-three-point-linear}
Consider Q-TD-PMD with adaptive step sizes $\{\eta_k\}$.    For $k=0,1,...,T-1$, there holds
    \begin{align}
\mathcal{F} ^{\pi _{k+1}}Q^k(s,a)-\mathcal{F} Q^k(s,a) \ge - \frac{\gamma}{\eta _k}\left\| \hat{D}_{\pi _k}^{\tilde{\pi}_k} \right\| _{\infty},
        \label{eq:Q-TD-PMD-three-point-linear}
    \end{align}
    where $
\hat{D}_{\pi _k}^{\tilde{\pi}_k}\left( s,a \right) :=\mathbb{E} _{s^{\prime}\sim P\left( \cdot |s,a \right)}\left[ D_{\pi _k}^{\tilde{\pi}_k}\left( s^{\prime} \right) \right] $ and  $\tilde{\pi}_k$ is any policy that satisfies
    \begin{align*}
\langle\tilde{\pi}_k(\cdot|s),Q^k(s,\cdot)\rangle = \max_aQ^{k}(s,a),\;\forall s.
    \end{align*}
\end{lemma}
\begin{proof} A direct computation yields that 
\begin{align*}
&\mathcal{F} ^{\pi _{k+1}}Q^k\left( s,a \right) =r\left( s,a \right) +\gamma \mathbb{E} _{s^{\prime}\sim P\left( \cdot |s,a \right)}\left[ \langle \pi _{k+1}(\cdot |s^{\prime}),\,Q^k(s^\prime,\cdot )\rangle \right] 
\\
\,\,&\overset{\left( a \right)}{\ge}r\left( s,a \right) +\gamma \mathbb{E} _{s^{\prime}\sim P\left( \cdot |s,a \right)}\left[ \langle \tilde{\pi}(\cdot |s^{\prime}),\,Q^k(s^\prime,\cdot )\rangle +\frac{1}{\eta _k}\left( D_{\pi _k}^{\pi _{k+1}}(s^{\prime})+D_{\pi _{k+1}}^{\tilde{\pi}_k}(s^{\prime})-D_{\pi _k}^{\tilde{\pi}_k}(s^{\prime}) \right) \right] 
\\
\,\,&\ge \mathcal{F} Q^k\left( s,a \right) -\frac{\gamma}{\eta _k}\hat{D}_{\pi _k}^{\tilde{\pi}_k}\left( s,a \right) 
\\
\,\,&\ge \mathcal{F} Q^k\left( s,a \right) -\frac{\gamma}{\eta _k}\left\| \hat{D}_{\pi _k}^{\tilde{\pi}_k} \right\| _{\infty},
\end{align*}
where (a) is due to Lemma \ref{lem:Q-TD-PMD-three-point-original}.
\end{proof}
It easy to see that TD evaluation of Q-TD-PMD can be close to Q value iteration (QVI)  by setting large enough step size $\eta_k$. As QVI is known to converge $\gamma$-linearly~\cite{suttonRL}, the $\gamma$-rate linear convergence of Q-TD-PMD can be established by enlarging $\eta_k$ so that the divergence term is well controlled. Based on Lemma \ref{lem:Q-TD-PMD-three-point-linear}, one can use the similar analysis as in the proof of Theorem \ref{thm:linear-convergence-of-exact-TD-PMD} to establish the $\gamma$-rate of Q-TD-PMD. We present this result without proof. 
\begin{theorem}[Linear convergence]
    \label{thm:linear-convergence-of-exact-Q-TD-PMD}
    Consider  Q-TD-PMD  with adaptive step sizes $\{ \eta_k \}$. Assume $\eta_k \geq {(\gamma\|\hat{D}_{\pi_k}^{\tilde{\pi}_k} \|_\infty})/{(c\gamma^{2k+1})}$,
    where $c > 0$ is an arbitrary positive constant. Then, we have
    \begin{align}
        \ls\| Q^* - Q^T \rs\|_\infty &\leq \gamma^T \ls[ \ls\| Q^* - Q^0 \rs\|_\infty + \frac{c}{1-\gamma} \rs]
        \label{eq:linear-convergence-of-output-value-Q-TD-PMD} \\
        \ls\| Q^* - Q^{\pi_T} \rs\|_\infty &\leq \frac{2\gamma^{T-1}}{1-\gamma} \ls[ \ls\| Q^* - Q^0 \rs\|_\infty + \frac{c}{1-\gamma} \rs].
        \label{eq:linear-convergence-of-output-policy-Q-TD-PMD}
    \end{align}
\end{theorem}
%%%%%%
\subsection{Policy convergence of Q-TD-PQA and Q-TD-NPG}
% Finally, we give policy convergence analysis for two Q-TD-PMD instances, namely \textbf{Q-TD-PQA} and \textbf{Q-TD-NPG}. 
Parallel to TD-PQA and TD-NPG in Section~\ref{sec:TD-PQA-and-TD-NPG}, Q-TD-PQA and Q-TD-NPG have the following policy update forms:
\begin{align*}
    &\mbox{(Q-TD-PQA)} \quad \forall\, k\in\mathbb{N}, \, s\in\calS: \quad \pi_{k+1}(\cdot|s) = \Proj_{\Delta(\calA)} \ls( \pi_k(\cdot|s) + \eta\, Q^k(s,\cdot) \rs), \numberthis
    \label{eq:Q-TD-PQA-update} \\
    &\mbox{(Q-TD-NPG)} \quad \forall\, k\in\mathbb{N}, \, s\in\calS: \quad \pi_{k+1}(a|s) = \frac{\pi_k(a|s) \cdot \exp \ls\{ \eta\, Q^k(s,a) \rs\}}{\sum_{a^\prime\in\calA} \pi_k(a^\prime|s) \cdot \exp \ls\{ \eta\, Q^k(s,a^\prime) \rs\}}. \numberthis
    \label{eq:Q-TD-NPG-update}
\end{align*}
The policy convergence results for Q-TD-PQA and Q-TD-NPG are presented as Theorem~\ref{thm:finite-iteration-convergence-of-Q-TD-PQA} and Theorem~\ref{thm:policy-convergence-of-Q-TD-NPG} below. Though the results are similar to Theorem~\ref{thm:finite-iteration-convergence-of-TD-PQA} and Theorem~\ref{thm:policy-convergence-of-TD-NPG}, we would like to emphasize that the proof details are essentially different  as we cannot obtain a convergence result for $\| V^* - V^{\pi_k} \|_\infty$ for Q-TD-PMD. In order to control $\| V^* - V^{\pi_k} \|_\infty$, we need to leverage the explicit policy update formula and the convergence of $Q^{\pi_k}$ and $Q^k$ to investigate the probability on non-optimal actions. To this end, we define
\begin{align*}
    \forall\, k \in \mathbb{N}^+, \, s\in\calS: \quad h_s^k := \sum_{a^*\in\calA^*_s} \pi_k(a^*|s).
\end{align*}
It is clear that $h^k_s = 1 - b_s^k$, and the key of the  analysis lies in the characterization of $h^k_s$. 

We will first establish the finite iteration policy convergence result of Q-TD-PQA.

\begin{theorem}
    With any constant step size $\eta_k = \eta > 0$, there exists a finite time $T_0$\footnote{One can use Theorem~\ref{thm:sublinear-convergence-of-exact-Q-TD-PMD-formal} to compute an upper bound for $T_0$ as in Theorem~\ref{thm:finite-iteration-convergence-of-TD-PQA}, and the details are omitted.}, such that for all $k \geq T_0$, the policies $\pi_k$ are optimal.
    \label{thm:finite-iteration-convergence-of-Q-TD-PQA}
\end{theorem}

    As already noted, the proof route of this theorem is totally different from that for Theorem~\ref{thm:finite-iteration-convergence-of-TD-PQA}, because we can neither directly establish the convergence of $\| V^* - V^{\pi_k} \|_\infty$, nor bound $\| V^* - V^{\pi_k} \|_\infty$ using $\| Q^* - Q^{\pi_k} \|_\infty$. To begin with, we first introduce two more auxiliary results from~\cite{ppgliu}.

    \begin{lemma}
        The policy update of Q-TD-PQA (equation~\eqref{eq:Q-TD-PQA-update}) has the following equivalent expression
        \begin{align*}
            \forall\, k\in\mathbb{N}, \,\, (s,a) \in \calS\times\calA:\quad \pi_{k+1}(a|s) &= \ls[ \pi_k(a|s) + \eta \, Q^k(s,a) + \lambda_s^k \rs]_+ \\
            &= \ls[ \pi_k(a|s) + \eta \, A^k(s,a) + \lambda_s^k \rs]_+,
        \end{align*}
        where $A^k(s,a) := Q^k(s,a) - V^*(s)$ and $\lambda_s^k$ is a scalar such that $\sum_{a\in\calA} \pi_{k+1}(a|s) = 1$.
        \label{pro:Q-TD-PQA-update}
    \end{lemma}
    
    \begin{lemma}[\protect{\cite[Lemma~8]{ppgliu}}]
        Define $\calA^k_s := \arg\max_{a\in\calA} \, Q^k(s,a)$ and $\calB^k_s := \{ a: \; \pi_k(a|s) > 0 \}$. For Q-TD-PQA with constant step size, $\calB^k_s$ admits one of the following three forms:
        \begin{enumerate}
            \item $\calB^{k+1}_s \subsetneqq \calA^k_s$,
            \item $\calB^{k+1}_s = \calA^k_s$,
            \item $\calB^{k+1}_s = \calA^k_s \cup \mathcal{C}_s$, where $\mathcal{C}_s \subset \calA \setminus \calA^k_s$ is not empty.
        \end{enumerate}
        \label{lem:proj-support}
    \end{lemma}
    The proofs of Lemmas~\ref{pro:Q-TD-PQA-update}~and~\ref{lem:proj-support} follow the same arguments as in~\cite{ppgliu}, thus are omitted. 

    \begin{proof}[Proof of Theorem~\ref{thm:finite-iteration-convergence-of-Q-TD-PQA}]
    As $Q^k \to Q^*$ and $Q^{\pi_k} \to Q^*$ (see Theorem~\ref{thm:sublinear-convergence-of-exact-Q-TD-PMD-formal}), for a small enough constant $\varepsilon > 0$, there exists a finite time $T \in \mathbb{N}^+$ such that 
    \begin{align*}
        \forall\, k \geq T: \quad \| A^* - A^k \|_\infty = \| Q^* - Q^k \|_\infty \leq \varepsilon.
    \end{align*}
    The key of the analysis is to verify the following three claims:
    \begin{enumerate}[label=(\Roman*)]
        \item \label{eq:claim-1} for any $s\in\calS$, there exists a constant $c > 0$ such that
        \begin{align}
            \forall\, k\geq T: \quad \mbox{if} \;\; h_s^{k+1} < 1 \;\; \mbox{then} \;\; h_s^{k+1} > h_s^k + c;
            \label{eq:property-1}
        \end{align}
        \item \label{eq:claim-2} 
        \begin{align}
            \forall\, k \geq T: \quad  \mbox{if} \;\; h_s^k \geq 1 - c \;\; \mbox{then} \;\;  h_s^{k+1} = 1;
            \label{eq:property-2}
        \end{align}
        \item \label{eq:claim-3}
        \begin{align}
            \forall\, k\geq T: \quad \mbox{if} \;\; h_s^k = 1 \;\; \mbox{then} \;\; h_s^{k+1} = h_s^k = 1.
        \label{eq:property-3}
    \end{align}
    \end{enumerate}
    \textit{Proof of Claim~\ref{eq:claim-1}:} By Lemma~\ref{pro:Q-TD-PQA-update},
    \begin{align*}
        \forall\, k\in\mathbb{N}, \,\, (s,a) \in \calS\times\calA:\quad \pi_{k+1}(a|s)
        &= \ls[ \pi_k(a|s) + \eta \, A^k(s,a) + \lambda_s^k \rs]_+.
    \end{align*}
    Additionally, we have 
    \begin{align*}
        \forall\, k \geq T: \quad \calA^k_s \subseteq \calA^*_s, 
    \end{align*}
    as $\| A^* - A^k \|_\infty \leq \varepsilon$ is sufficiently small. 

    For any  fixed  $k \geq T$, note  that $h_s^{k+1} < 1$ implies that $\calB^{k+1}_s \setminus \calA^*_s$ is not empty. By Lemma~\ref{lem:proj-support}, as $\mathcal{C}_s = \calB_s^{k+1} \setminus \calA^k_s$ is not empty, we know that $\calA^k_s \subseteq \calB^{k+1}_s$, thus $\calB_s^{k+1} \cap \calA^*_s$ is not empty. Moreover,
    \begin{align*}
        1 &= \sum_{a^*\in\calB^{k+1}_s \cap \calA^*_s} \pi_{k+1}(a^*|s) + \sum_{a^\prime \in \calB^{k+1}_s \setminus \calA^*_s} \pi_{k+1}(a^\prime|s) \\
        &= \sum_{a^*} \ls[ \pi_k(a^*|s) + \eta A^k(s,a^*) + \lambda_s^k \rs] + \sum_{a^\prime} \ls[ \pi_k(a^\prime|s) + \eta A^k(s,a^\prime) + \lambda_s^k \rs] \\
        &\leq \sum_{a^*} \ls[ \pi_k(a^*|s) + \eta \varepsilon + \lambda_s^k \rs] + \sum_{a^\prime} \ls[ \pi_k(a^\prime|s) - \eta (\Delta - \varepsilon) + \lambda_s^k \rs] \\
        &\leq \ls[\sum_{a\in\calA} \pi_k(a|s)\rs] +  \ls[\sum_{a^*} \eta\varepsilon \rs] - \ls[ \sum_{a^\prime} \eta (\Delta - \varepsilon) \rs] + \ls[ \sum_{a \in \calB^{k+1}_s} \lambda_s^k \rs] \\
        &\leq 1 + \eta |\calA| \varepsilon - \eta (\Delta-\varepsilon) + |\calB^{k+1}_s| \lambda_s^k,
    \end{align*}
    where we leverage equation~\eqref{eq:bounds-for-A-k}. Consequently,
    \begin{align*}
        \forall\, k\geq T, \, s\in\calS: \quad  \lambda_s^k \geq \eta \cdot \frac{\Delta - (|\calA|+1)\varepsilon}{|\calB_s^{k+1}|} \geq \eta \cdot \frac{\Delta - 2|\calA|\varepsilon}{|\calA|} > 0,
    \end{align*}
    as $\varepsilon$ is sufficiently small. Thus for any optimal action $a \in \calA^*_s$, there holds
    \begin{align*}
        \pi_{k+1}(a|s) &= \ls[ \pi_k(a|s) + \eta A^k(s,a) + \lambda_s^k \rs]_+ \\
        &\geq \ls[ \pi_k(a|s) + \eta A^k(s,a) + \lambda_s^k \rs] \\ 
        &\geq \pi_k(a|s) - \eta \varepsilon + \lambda_s^k \\
        &\geq \pi_k(a|s) + \eta \frac{\Delta - 3|\calA|\varepsilon}{|\calA|}.
    \end{align*}
    By picking $0 < c_0 < \eta(\Delta - 3 |\calA| \varepsilon) / |\calA|$, we obtain
    \begin{align*}
        \forall\, k \geq T, \, s\in\calS, a\in\calA^*_s: \quad \pi_{k+1}(a|s) > \pi_k(a|s) + c_0,
    \end{align*}
    which leads to
    \begin{align*}
        \forall\, k\geq T, \, s\in\calS: \quad h_{s}^{k+1} > h_s^k + |\calA^*_s| \cdot c_0 := h_s^k + c.
    \end{align*}

    \textit{Proof of Claim~\ref{eq:claim-2}:} Assume $h_s^{k+1} < 1$. Then by Equation~\eqref{eq:property-1} we have $h_s^k < 1 - c$, which yields a contradiction.

    \textit{Proof of Claim~\ref{eq:claim-3}:} Assume $h_s^{k+1} < 1$. Then by Equation~\eqref{eq:property-1} we have $h_s^k < 1 - c$, which yields a contradiction.

  Note that $\pi_k$ is optimal if and only if $h_s^k = 1$ for all $s\in\calS$. By Claim~\ref{eq:claim-1} and~\ref{eq:claim-2}, after at most
    \begin{align*}
        T_0 := \lceil T + c^{-1} \rceil + 1
    \end{align*}
    iterations, we have $h_s^k = 1$ for all $s\in\calS$. By Claim~\ref{eq:claim-3}, $h_s^k$ remains  to be $1$ once this is attained. Thus, for any $k \geq T_0$, we know that $\pi_k$ is optimal, which completes the proof.
\end{proof}

Similar to TD-NPG, the policy generated by Q-TD-NPG converges to some optimal policy.

\begin{theorem}
    With any constant step size $\eta_k = \eta > 0$, the policy generated by Q-TD-NPG converges to some optimal policy, i.e., $\pi_k \to \pi^*$, as $k \to \infty$.
    \label{thm:policy-convergence-of-Q-TD-NPG}
\end{theorem}

\begin{proof}
    As in Theorem~\ref{thm:finite-iteration-convergence-of-Q-TD-PQA}, for arbitrary small $\varepsilon > 0$,  define the finite time $T$ as
    \begin{align*}
        \forall\, k\geq T: \quad \| A^* - A^k \|_\infty = \| Q^* - Q^k \|_\infty \leq \varepsilon, \quad \| Q^* - Q^{\pi_k} \|_\infty \leq \varepsilon.
    \end{align*}
    The key of the analysis is to show  the following two claims:
    \begin{enumerate}[label=(\Roman*)]
        \item \label{eq:NPG-Claim-1} for any $s\in\calS$, 
        \begin{align}
            \forall\, k\geq T: \quad \mbox{if} \;\; h_s^{k} \leq 1 - \frac{3\eta \varepsilon}{1-\exp(-\eta\Delta)} \;\; \mbox{then} \;\; h^{k+1}_s \geq h_s^k \cdot \frac{1-2\eta\varepsilon}{1-3\eta\varepsilon};
            \label{eq:NPG-property-1}
        \end{align}
        \item \label{eq:NPG-Claim-2} 
        \begin{align}
            \forall\, k \geq T: \quad \mbox{if} \;\; h_s^k \geq 1 - \frac{3\eta\varepsilon}{1-\exp(-\eta\Delta)} \;\; \mbox{then} \;\; h_s^{k+1} \geq 1 - \frac{3\eta\varepsilon}{1-\exp(-\eta\Delta)} \;\; \mbox{as well}.
        \end{align}
    \end{enumerate}
    \textit{Proof of Claim~\ref{eq:NPG-Claim-1}:} Recall the policy update formula of Q-TD-NPG,
    \begin{align*}
        \forall\, k\in\mathbb{N}, \, s\in\calS: \quad \pi_{k+1}(a|s) &= \frac{\pi_k(a|s) \cdot \exp \ls\{ \eta\, Q^k(s,a) \rs\}}{\sum_{a^\prime\in\calA} \pi_k(a^\prime|s) \cdot \exp \ls\{ \eta\, Q^k(s,a^\prime) \rs\}} \\
        &= \frac{\pi_k(a|s) \cdot \exp \ls\{ \eta\, A^k(s,a) \rs\}}{\sum_{a^\prime\in\calA} \pi_k(a^\prime|s) \cdot \exp \ls\{ \eta\, A^k(s,a^\prime) \rs\}}.
    \end{align*}
    Let $Z^k_s = \sum_{a\in\calA} \pi_k(a|s) \exp (\eta \, A^k(s,a))$. By equation~\eqref{eq:bounds-for-A-k}, 
    \begin{align*}
        \forall\, k \geq T: \quad Z_s^k &= \sum_{a^* \in\calA^*_s} \pi_k(a^*|s) \exp (\eta\, A^k(s,a^*)) + \sum_{a^\prime\not\in\calA^*_s} \pi_k(a^\prime|s) \exp (\eta\, A^k(s,a^\prime)) \\
        &\leq \sum_{a^*\in\calA^*_s} \pi_k(a^*|s) \exp(\eta \varepsilon) + \sum_{a^\prime\not\in\calA^*_s} \pi_k(a^\prime|s) \exp (-\eta (\Delta - \varepsilon)) \\
        &= (1-b_s^k) \cdot \exp(\eta\varepsilon) + b_s^k \cdot \exp(-\eta(\Delta-\varepsilon)) \\[.5em]
        &= \exp(\eta \varepsilon) - b_s^k \ls[ \exp(\eta\varepsilon) - \exp(-\eta (\Delta - \varepsilon)) \rs].
    \end{align*}
    Based on this,  we can provide a lower bound for the policy ratio of optimal actions. For any state $s\in\calS$,
    \begin{align*}
        \forall\, k \geq T, \; a^*\in\calA^*_s : \quad \frac{\pi_{k+1}(a^*|s)}{\pi_k(a^*|s)} &= \frac{\exp(\eta A^k(s,a^*))}{Z_s^k} \\
        &\geq \frac{\exp(-\eta\varepsilon)}{\exp(\eta\varepsilon) - b_s^k \cdot \ls[ \exp(\eta\varepsilon) - \exp(-\eta (\Delta - \varepsilon)) \rs]} \\[.7em]
        &= \exp(-2\eta\varepsilon) \cdot \frac{1}{1-(1-\exp(-\eta\Delta)) \cdot b_s^k}. 
    \end{align*}
    It implies that
    \begin{align*}
        \forall\, k \geq T: \quad \frac{h_s^{k+1}}{h_s^k} &= \frac{\sum_{a^*\in\calA^*_s} \pi_{k+1}(a^*|s)}{\sum_{a^*\in\calA^*_s} \pi_k(a^*|s)} \\
        &= \frac{\sum_{a^*} \dfrac{\pi_{k+1}(a^*|s)}{\pi_k(a^*|s)} \pi_k(a^*|s)}{\sum_{a^*} \pi_k(a^*|s)} \\[.7em]
        &\geq \exp(-2\eta\varepsilon) \cdot \frac{1}{1-(1-\exp(-\eta\Delta)) \cdot b_s^k} \cdot  \frac{\sum_{a^*} \pi_k(a^*|s)}{\sum_{a^*} \pi_k(a^*|s)} \\[.7em]
        &= \frac{\exp(-2\eta\varepsilon)}{(1-\exp(-\eta\Delta)) \cdot h_s^k  + \exp(-\eta\Delta)} \geq \frac{1-2\eta\varepsilon}{(1-\exp(-\eta\Delta)) \cdot h_s^k  + \exp(-\eta\Delta)}. \numberthis \label{eq:bound-of-h-k-ratio}
    \end{align*}
    Under the condition $h_s^k \leq 1 - 3\eta\varepsilon / (1-\exp(-\eta\Delta))$, it follows that
    \begin{align*}
        \forall\, k \geq T: \quad \frac{h_s^{k+1}}{h_s^k} \geq \frac{1-2\eta\varepsilon}{1-3\eta\varepsilon}
    \end{align*}
    which proves Claim~\ref{eq:NPG-Claim-1}.

    \textit{Proof of Claim~\ref{eq:NPG-Claim-2}:} By equation~\eqref{eq:bound-of-h-k-ratio}, we have
    \begin{align*}
        \forall\, k \geq T: \quad h_s^{k+1} \geq \frac{(1-2\eta\varepsilon) \cdot h^k_s}{(1-\exp(-\eta\Delta)) \cdot h_s^k  + \exp(-\eta\Delta)} := f(h_s^k).
    \end{align*}
    Notice that $f(x) = \frac{(1-2\eta\varepsilon) \cdot x}{(1-\exp(-\eta\Delta)) \cdot x + \exp(-\eta\Delta)}$ is increasing over $x\in(0,1)$. When $h_s^k \geq 1 - 3\eta\varepsilon / (1-\exp(-\eta\Delta))$, there holds
    \begin{align*}
        \forall\, k \geq T: \quad h_s^{k+1} \geq f(h_s^k) &\geq f\ls( 1 - \frac{3\eta\varepsilon}{1-\exp(-\eta\Delta)} \rs) \\
        &= \frac{1-2\eta\varepsilon}{1-3\eta\varepsilon} \cdot \ls[ 1 - \frac{3\eta\varepsilon}{1-\exp(-\eta\Delta)} \rs] \\
        &\geq 1 - \frac{3\eta\varepsilon}{1-\exp(-\eta\Delta)},
    \end{align*}
    which proves Claim~\ref{eq:NPG-Claim-2}.

    Combining Claim~\ref{eq:NPG-Claim-1} and Claim~\ref{eq:NPG-Claim-2}, we know that for all state $s\in\calS$, $h_s^k$ will increase linearly before it reaches $1-3\eta\varepsilon/(1-\exp(-\eta\Delta))$. Moreover, once $h_s^k$ is larger than $1 - 3\eta\varepsilon/(1-\exp(-\eta\Delta))$, it will not be smaller than it anymore. Thus there exists a time $T_0$ such that
    \begin{align*}
        \forall\, k \geq T_0 : \quad \| A^* - A^k \|_\infty = \| Q^* - Q^k \|_\infty \leq \varepsilon, \quad \| Q^* - Q^{\pi_k} \|_\infty \leq \varepsilon, \quad h_s^k \geq 1 - \frac{3\eta\varepsilon}{1-\exp(-\eta\Delta)}.
    \end{align*}

   To proceed the proof of Theorem~\ref{thm:policy-convergence-of-Q-TD-NPG}, we need to further establish the local linear convergence of both $Q^{\pi_k}$ and $Q^k$.

    \textbf{Local linear convergence of $Q^{\pi_k}$.} The key is to bound the policy ratio for non-optimal actions. By a direct computation,
    \begin{align*}
        \forall\, s\in\calS, \; a^\prime \not\in\calA^*_s, \; t\geq T_0: \quad \frac{\pi_{t+1}(a^\prime|s)}{\pi_t(a^\prime|s)} &= \frac{\exp (\eta A^t(s, a^\prime))}{\sum_a \pi_t(a|s) \exp(\eta A^t(s,a))} \\
        &\leq \frac{\exp(-\eta (\Delta - \varepsilon))}{\sum_{a^* \in\calA^*_s} \pi_t(a^*|s) \exp (\eta A^t(s,a^*))} \\
        &\leq \frac{\exp(-\eta (\Delta - \varepsilon))}{(1-b_s^{\pi_t}) \exp(-\eta\varepsilon)} \\
        &= \frac{\exp(-\eta (\Delta - \varepsilon))}{h_s^t \cdot \exp(-\eta\varepsilon)} \\
        &\leq \exp (-\eta\Delta + 2\eta\varepsilon) \cdot \ls[ 1 - \frac{3\eta\varepsilon}{1-\exp(-\eta\Delta)} \rs]^{-1} := \rho_0.
    \end{align*}
    As $\varepsilon > 0$ is arbitrary small, we have $\rho_0 < 1$. Now with the same arguments as in the proof of Lemma~\ref{lem:local-linear-convergence-of-Q-pi-TD-NPG}, we can establish the local linear convergence of $Q^{\pi_k}$,
    \begin{align*}
        \forall\, k \geq T_0: \quad \| Q^* - Q^{\pi_k} \|_\infty \leq \frac{|\calS|^2}{1-\gamma} \varepsilon \rho^{k-T_0},
    \end{align*}
    where $\max\{ \gamma, \, \rho_0 \} < \rho < 1$.

    \textbf{Local linear convergence of $Q^k$ and the policy convergence. } Note that the $Q^k$ update formula of Q-TD-NPG is the same with TD-NPG,
    \begin{align*}
        Q^{k+1}(s,a) = \calF^{\pi_{k+1}} Q^k(s,a).
    \end{align*}
    Thus using the same arguments as in the proof of Theorem~\ref{thm:policy-convergence-of-TD-NPG}, we can establish the local linear convergence of $Q^k$,
    \begin{align*}
        \forall\, k \geq T_0 + 1 : \quad \| Q^* - Q^k \|_\infty \leq \ls( \frac{(\rho+\gamma) \cdot |\calS|}{(\rho-\gamma) (1-\gamma)} \varepsilon + \| Q^{\pi_{T_0}} - Q^{T_0} \|_\infty \rs) \cdot \rho ^{k - T_0},
    \end{align*}
and then obtain the policy convergence of Q-TD-NPG.
\end{proof}

\subsection{Convergence of inexact Q-TD-PMD}
Consider the inexact Q-TD-PMD, where equation \eqref{eq:alg2Q} is replaced by\footnote{$\widehat{\mathcal{F}}^{\pi_k}$ denotes an approximate evaluation of ${\mathcal{F}}^{\pi_k}$.}
\begin{align}
    Q^{k+1}=\hat{\mathcal{F}}^{\pi_{k+1}} Q^k \;\; \mbox{with} \; \left\|Q^{k+1}-\mathcal{F} ^{\pi _{k+1}}Q^k \right\| _{\infty}\le \delta.\label{Q-TD-PMD esitmation error}
\end{align}
Under certain error level $\delta$, a variant of Lemma \ref{lem:error-bounds-Q} is presented as below. We omit its proof for simplicity.
\begin{lemma}
    \label{lem:error-bounds-Q-error}
    For the inexact Q-TD-PMD with error level $\delta$, there holds 
    \begin{align}
        \ls\| Q^{*} - Q^{\pi_T} \rs\|_\infty \leq \frac{1}{1-\gamma} \ls( \ls\| Q^* - Q^T \rs\|_\infty + \ls\| Q^* - Q^{T-1} \rs\|_\infty+\delta\rs).
    \end{align}
\end{lemma}
We can also establish a variant of Lemma \ref{lem:Q-TD-PMD-three-point-linear} that takes the evaluation error into account.
\begin{lemma}
    \label{lem:Q-TD-PMD-three-point-linear-error}
Consider inexact Q-TD-PMD with adaptive step sizes $\{\eta_k\}$ and error level $\delta$.    For $k=0,1,...,T-1$, there holds
    \begin{align}
\hat{\mathcal{F}} ^{\pi _{k+1}}Q^k(s,a)-\mathcal{F} Q^k(s,a) \ge -\frac{\gamma}{\eta _k}\left\| \hat{D}_{\pi _k}^{\tilde{\pi}_k} \right\| _{\infty}-\delta,
        \label{eq:Q-TD-PMD-three-point-linear-error}
    \end{align}
    where $
\hat{D}_{\pi _k}^{\tilde{\pi}_k}\left( s,a \right) :=\mathbb{E} _{s^{\prime}\sim P\left( \cdot |s,a \right)}\left[ D_{\pi _k}^{\tilde{\pi}_k}\left( s^{\prime} \right) \right] $ and  $\tilde{\pi}_k$ is any policy that satisfies
    \begin{align*}
\langle\tilde{\pi}_k(\cdot|s),Q^k(s,\cdot)\rangle = \max_aQ^{k}(s,a),\;\forall s.
    \end{align*}
\end{lemma}
\begin{proof} A direct computation yields that 
\begin{align*}
\left[ \hat{\mathcal{F}}^{\pi _{k+1}}Q^k-\mathcal{F} Q^k \right] \left( s,a \right) &=\left[ \mathcal{F} ^{\pi _{k+1}}Q^k-\mathcal{F} Q^k \right] \left( s,a \right) +\left[ \hat{\mathcal{F}}^{\pi _{k+1}}Q^k-\mathcal{F} ^{\pi _{k+1}}Q^k \right] \left( s,a \right) 
\\
\,\,                             &\ge \left[ \mathcal{F} ^{\pi _{k+1}}Q^k-\mathcal{F} Q^k \right] \left( s,a \right) -\delta 
\\
\,\,                             &\overset{\left( a \right)}{\ge}-\frac{\gamma}{\eta _k}\left\| \hat{D}_{\pi _k}^{\tilde{\pi}_k} \right\| _{\infty}-\delta,
\end{align*}
where (a) is due to Lemma \ref{lem:Q-TD-PMD-three-point-linear}.
\end{proof}
Based on Lemma \ref{lem:error-bounds-Q-error} and  Lemma \ref{lem:Q-TD-PMD-three-point-linear-error}, one can establish the linear convergence of inexact Q-TD-PMD using similar analysis as in the proof of Theorem \ref{thm:linear-convergence-of-exact-TD-PMD-inexact} and Theorem \ref{thm:linear-convergence-of-exact-Q-TD-PMD}. We present the results in the following without providing the detailed proofs for simplicity. 

\begin{theorem}[Linear convergence with error level $\delta$]
    \label{thm:linear-convergence-of-exact-Q-TD-PMD-inexact}
    Consider the inexact Q-TD-PMD  with adaptive step sizes $\{\eta_k\}$. Assume $\eta_k \geq {(\gamma\| {\hat{D}}_{\pi_k}^{\tilde{\pi}_k} \|_\infty})/{(c \gamma^{2k+1})}$, where $c > 0$ is an arbitrary positive constant. Then, we have
    \begin{align*}
        \ls\| Q^* - Q^T \rs\|_\infty &\leq \gamma^T \ls[ \ls\| Q^* - Q^0 \rs\|_\infty + \frac{c}{1-\gamma} \rs] + \frac{\delta}{1-\gamma} \\
        % \label{eq:linear-convergence-of-output-value-inexact}
\| Q^*-Q^{\pi _T}\| _{\infty}&\le \frac{2\gamma ^{T-1}}{1-\gamma}\left( \left\| Q^*-Q^0 \right\| _{\infty}+\frac{c}{1-\gamma} \right) +\frac{3\delta}{\left( 1-\gamma \right) ^2}.
\end{align*}
\end{theorem}

\subsection{Sample complexity}
\begin{algorithm}[!htbp]
   \caption{Sample-based Q-TD-PMD}
   \label{alg:sample-based-Q-TD-PMD}
\begin{algorithmic}
    {
   \STATE {\bfseries Input:} Initial state-action value estimation $Q^0$, initial policy $\pi_0$, iteration number $T$, step size $\{ \eta_k \}$, the number of samples $M_Q$ for TD evaluation.
   \FOR{$k=0$ {\bfseries to} $T-1$}
   \vspace{0.2cm}
   \STATE {\bfseries (Policy improvement)} 
    Update the policy by
   \begin{align*}
       \pi_{k+1}(\cdot|s) \!=\! \underset{p\in\Delta(\calA)}{\arg\max} \ls\{ \eta_k \langle p, \, Q^{k}(s,\cdot) \rangle \!-\! D^p_{\pi_k}(s) \rs\}.\numberthis\label{eq:alg3pi}
   \end{align*}
   \STATE {\bfseries (TD evaluation)} For each $(s,a)$, sample i.i.d. $\{(s^\prime_{(i)},a^\prime_{(i)}) \}_{i=1}^{M_Q} \sim P(\cdot|s,a)\times \pi_{k+1}(\cdot|s^\prime_{(i)})$ and compute 
\begin{align*}
Q^{k+1}\left( s,a \right) =\frac{1}{M_Q}\sum\nolimits_{i=1}^{M_Q}{\left( r\left( s,a \right) +\gamma Q^k\left( s_{\left( i \right)}^{\prime},a_{\left( i \right)}^{\prime} \right) \right)}
\end{align*}
   \ENDFOR
   \STATE {\bfseries Output:} Last iterate policy $\pi_T$, last iterate action value estimation $Q^T$.
   }
\end{algorithmic}
\end{algorithm}

We now provide the sample complexity of the sample-based Q-TD-PMD. Similar to the analysis of TD-PMD, we still assume that there is a generative model which allows us to estimate $\mathcal{F}^{\pi_{k+1}}Q^k$ via a number of independent samples at every $(s,a)$. More concretely, at each iteration $k$, for each $(s,a) \in \calS\times\calA$, we sample i.i.d. $\{ (s^\prime_{(i)},a^\prime_{(i)}) \}_{i=1}^{M_Q} \sim P(\cdot|s,a)\times \pi_{k+1}(\cdot|s^\prime)$ to compute $Q^{k+1}$,
\begin{align*}
\forall \,(s,a)\in \mathcal{S} \times \mathcal{A} \quad Q^{k+1}\left( s,a \right) =\hat{\mathcal{F}}^{\pi _{k+1}}Q^k(s,a)=\frac{1}{M_Q}\sum\nolimits_{i=1}^{M_Q}{\left( r\left( s,a \right) +\gamma Q^k\left( s_{\left( i \right)}^{\prime},a_{\left( i \right)}^{\prime} \right) \right)}.
\end{align*}
The detailed  description of sample-based Q-TD-PMD is presented in Algorithm~\ref{alg:sample-based-Q-TD-PMD}. Using the similar analysis as in Section \ref{sec:sample-complexity}, we can obtain the sample complexity such that the last iterate policy satisfies
\[\| Q^*-Q^{\pi _T}\| _{\infty}\le \varepsilon \]
with high probability. Firstly, using the Hoeffding’s inequality (Lemma~\ref{lem:hoeffding}), we can obtain the number of samples needed to satisfy \eqref{Q-TD-PMD esitmation error} with high probability.

\begin{lemma} 
\label{lem:sample-complexity-for-error-level-delta-Q-TD-PMD}
Consider the sample-based Q-TD-PMD with  $\ls\| Q_0 \rs\|_\infty \leq {1}/{(1-\gamma)}$. For any $\alpha \in (0,1)$, if 
\begin{align*}M_Q \geq \frac{1}{2(1-\gamma)^2 \delta^2} \log \frac{2T|\calS||\calA|}{\alpha},\end{align*} then \eqref{Q-TD-PMD esitmation error} is satisfied for $k=0,\cdots,T-1$ with  probability at least $1-\alpha$.
\end{lemma}

The proof of Lemma~\ref{lem:sample-complexity-for-error-level-delta-Q-TD-PMD} is very similar to that of Lemma~\ref{lem:sample-complexity-for-error-level-delta} thus is omitted here.  Together with Theorem \ref{thm:linear-convergence-of-exact-Q-TD-PMD-inexact}, one can finally obtain the sample complexity for the sample-based Q-TD-PMD.

\begin{theorem} Consider the sample-based Q-TD-PMD with $\ls\| Q_0 \rs\|_\infty \leq {1}/{(1-\gamma)}$ and  $  \eta_k \geq {\gamma\ls\| {\hat{D}}_{\pi_k}^{\tilde{\pi}_k} \rs\|_\infty}/{(c \cdot \gamma^{2k+1})},$
    where $c > 0$ is an arbitrary positive constant.
    For any $\alpha \in (0,1)$, if $M_Q$ satisfies the conditions in Lemma~\ref{lem:sample-complexity-for-error-level-delta-Q-TD-PMD}, then 
    \begin{align*}
\ls\| Q^*-Q^{\pi _T}\rs\| _{\infty}\le \frac{1}{\left( 1-\gamma \right) ^2}\left[ 2\left( 2+c \right) \gamma ^{T-1}+3\delta \right]
    \end{align*}
   holds with probability at least $1-\alpha$.
    \label{thm:sample-complexity-Q}
\end{theorem}
We omit the proof of Theorem~\ref{thm:sample-complexity-Q} for simplicity here since it is overall similar to that of Theorem~\ref{thm:sample-complexity}. Setting $T = \tilde{O}((1-\gamma)^{-1}
)$ and $M_Q=\tilde{O}(\varepsilon^{-2}(1-\gamma)^{-6})$, it follows immediately from Theorem~\ref{thm:sample-complexity-Q} that the last iterate policy produced by the sample-based Q-TD-PMD achieves $
\| Q^*-Q^{\pi _T}\| _{\infty}\le \varepsilon $ with 
\[
    T\times \ls(|\calS| M_V + |\calS||\calA| M_Q \rs) = \tilde{O} \ls( \varepsilon^{-2} |\calS||\calA| (1-\gamma)^{-7} \rs)
\]
number of samples.

%%%%%%%%%%%%%%%%%%

%%%%%%%%%%%%%%%%%%%%%%%%%%%%%%%%%%%%%%%%%%%%%%%%%%%%%%%%%%%%

\end{document}